\documentclass[3p]{elsarticle}
\journal{Computer Methods in Applied Mechanics and Engineering}

\usepackage{graphicx}
\usepackage{mathtools}
\usepackage{amssymb}
\usepackage{amsthm}
\usepackage{latexsym}
\usepackage{bm}
\usepackage{color}
\usepackage{stmaryrd}
\usepackage{mathrsfs}
\usepackage{array}
\usepackage{algorithm}
\usepackage{algpseudocode}
\usepackage{multirow}
\usepackage{url}
\usepackage{mdframed}
\usepackage{longtable}

\makeatletter
\newcommand*{\rom}[1]{\expandafter\@slowromancap\romannumeral #1@}
\makeatother

\newenvironment{myenv}[1]
  {\mdfsetup{
    frametitle={\colorbox{white}{\space#1\space}},
    innertopmargin=10pt,
    frametitleaboveskip=-\ht\strutbox,
    frametitlealignment=\center
    }
  \begin{mdframed}
  }
  {\end{mdframed}}

\DeclareMathAlphabet\mathbfcal{OMS}{cmsy}{b}{n}

\SetSymbolFont{stmry}{bold}{U}{stmry}{m}{n}

\newtheorem{lemma}{Lemma}

\newtheorem{proposition}{Proposition}

\newtheorem{remark}{Remark}

\AfterEndEnvironment{proof}{\noindent\ignorespaces}

\numberwithin{equation}{section}

\biboptions{sort&compress}

\begin{document}
\begin{frontmatter}
\title{A continuum and computational framework for viscoelastodynamics: \rom{2}. Strain-driven and energy-momentum consistent schemes}

\author[label1,label2]{Ju Liu}
\ead{liuj36@sustech.edu.cn,liujuy@gmail.com}

\author[label1]{Jiashen Guan}

\address[label1]{Department of Mechanics and Aerospace Engineering, Southern University of Science and Technology, Shenzhen, Guangdong 518055, P.R.China}

\address[label2]{Guangdong-Hong Kong-Macao Joint Laboratory for Data-Driven Fluid Mechanics and Engineering Applications, Southern University of Science and Technology, Shenzhen, Guangdong, 518055, P.R.China}

\begin{abstract}
We continue our investigation of finite deformation linear viscoelastodynamics by focusing on constructing accurate and reliable numerical schemes. The concrete thermomechanical foundation developed in the previous study paves the way for pursuing discrete formulations with critical physical and mathematical structures preserved. Energy stability, momentum conservation, and temporal accuracy constitute the primary factors in our algorithm design. For inelastic materials, the directionality condition, a property for the stress to be energy consistent, is extended with the dissipation effect taken into account. Moreover, the integration of the constitutive relations calls for an algorithm design of the internal state variables and their conjugate variables. A directionality condition for the conjugate variables is introduced as an indispensable ingredient for ensuring physically correct numerical dissipation. By leveraging the particular structure of the configurational free energy, a set of update formulas for the internal state variables is obtained. Detailed analysis reveals that the overall discrete schemes are energy-momentum consistent and achieve first- and second-order accuracy in time, respectively. Numerical examples are provided to justify the appealing features of the proposed methodology.

\end{abstract}

\end{frontmatter}

\section{Introduction}
The research motivation of this work comes from the belief that reliable numerical methods are built based on a concrete continuum basis. In this study, we refer to Liu, Latorre, and Marsden \cite{Liu2021b} as Part \rom{1} and investigate reliable numerical designs for finite deformation linear viscoelastic models.  Previous efforts in this direction relied crucially on a set of evolution equations for stress-like variables, which are typically denoted by $\bm Q^{\alpha}$. Also, the forms of the evolution equations for $\bm Q^{\alpha}$ were given in a heuristic manner and often adopted different forms by different authors \cite{Simo1987,Holzapfel1996a,Holzapfel2001,Simo2006,Holzapfel1996,Gueltekin2016}. One contribution of Part \rom{1} was the derivation of the evolution equations based on a rational approach, and the evolution equations take the form
\begin{align*}
\frac{d}{dt} \bm Q^{\alpha} + \frac{\bm Q^{\alpha}}{\tau^{\alpha}} = \frac{d}{dt} \tilde{\bm S}^{\alpha}_{\mathrm{iso}} - \frac{d\mu^{\alpha}}{dt}\left( \bm \Gamma^{\alpha} - \bm I \right).
\end{align*}
The internal state variable $\bm \Gamma^{\alpha}$ can be viewed as a deformation tensor analogous to the right Cauchy-Green tensor and is conjugate to $\bm Q^{\alpha}$ with respect to the free energy. In Part \rom{1}, a set of evolution equations in terms of the inelastic strain $\left(\bm \Gamma^{\alpha} - \bm I\right) / 2$ was also provided, which reads as
\begin{align*}
\eta^{\alpha} \frac{d}{dt} \left( \frac{\bm \Gamma^{\alpha} - \bm I}{2} \right) + \mu^{\alpha} \left(\frac{ \bm \Gamma^{\alpha} - \bm I}{2} \right) = \frac12 \left( \tilde{\bm S}^{\alpha}_{\mathrm{iso}} - \hat{\bm S}^{\alpha}_{0} \right).
\end{align*}
In fact, the evolution equations for $\bm Q^{\alpha}$ result from taking time derivatives at both sides of the above equations.

Treating the strain as the basic independent variable, known as the strain-driven format, has been viewed as an appropriate approach when compared against the stress-driven format \cite{Simo2006}. One reason is that the strain cannot be uniquely determined from a given stress state in nonlinear elasticity (e.g., in the presence of material softening). A second reason is that the strain-driven format naturally relates the strain and stress through the potential energy, and this relation often serves as the basis for demonstrating numerical stability in the discrete formulation, especially in nonlinear problems. In this study, the relation between the stress-like variables $\bm Q^{\alpha}$ and the internal state variables $\bm \Gamma^{\alpha}$ is linear, meaning the two can interchange conveniently. Here, we choose to integrate the inelastic deformation by using the evolution equation written in terms of $\bm \Gamma^{\alpha}$ primarily due to the second reason. In specific, we will show that the discrete counterparts of $\bm \Gamma^{\alpha}$ can be integrated in such a way that the algorithmic stress is provably dissipative, and the numerical dissipation is consistent with the physical dissipation.

In the construction of the numerical algorithms, a primary design criterion is the \textit{directionality property}, which is a discrete analogue of the continuous relation between the stress and the deformation tensor. For elastodynamics, this property demands that the contraction between the algorithmic stress and the difference of the discrete deformation tensors equals twice the difference of the discrete energy. This concept was originally introduced in the context of $N$-body dynamics. Several algorithmic formulas for the conservative force were developed that satisfy the directionality property, including the difference quotient formula introduced by LaBudde and Greenspan \cite{Greenspan1984,LaBudde1974,LaBudde1976} and the scaled mid-point formula introduced by Chorin, et al. \cite{Chorin1978}. Inspired by those formulas, Simo introduced a collocation scheme for nonlinear elastodynamics based on a mean-value theorem argument \cite{Simo1992a,Simo1992c}. This algorithm necessitates a local Newton-Raphson iteration to determine the collocation parameter for general nonlinear models \cite{Laursen2001}. An alternate strategy that bypasses the burdensome local iteration is the discrete gradient formula proposed by Gonzalez \cite{Gonzalez2000}, in which a closed-form algorithmic stress formula was provided. Ramifications of this formula were later developed \cite{Armero2007,Bui2007,Romero2012} with applications in thin-walled structures \cite{Miehe2001,Sansour2004}, contact mechanics \cite{Hauret2006}, elastoplasticity \cite{Armero2006,Armero2007,Meng2002}, multiphysics coupled problems \cite{Franke2022}, to name a few. In addition to the directionality property, we mention that a consistent numerical design also demands the algorithmic stress to be symmetric and second-order accurate \cite{Romero2012}. The former ensures the discrete angular momentum conservation; the latter is often analyzed by examining the algorithmic stress against the mid-point formula. We also mention that higher-order accuracy can be achieved through a compositional approach \cite{Tarnow1994}.

In addition to the aforementioned physical problems, energy-momentum consistent numerical schemes have been designed for viscoelasticity. A majority of efforts focused on the viscoelastic models based on the multiplicative decomposition of the deformation gradient \cite{Reese1998,Sidoroff1974}. An initial effort was made by interpolating the strain-like internal state variable independently in the variational formulation \cite{Gross2010}. That strategy was later generalized with the aid of the GENERIC formalism \cite{Romero2009}, rendering an energy-entropy consistent integrator for thermo-viscoelasticity \cite{Krueger2016,Schiebl2021}. 

There are fewer studies on constructing structure-preserving schemes for the viscoelastic model discussed in Part \rom{1}. To the authors' knowledge, the only exception is the work of Mart\'in, et al. \cite{Martin2014}. In that work, the continuum model followed the framework of Holzapfel and Simo \cite{Holzapfel1996}, which will lead to non-vanishing non-equilibrium stress in the thermodynamic equilibrium limit (see Section 2.4.3 of Part \rom{1}); aside from that, a sophisticated algorithmic definition for the stress-like variable $\bm Q^{\alpha}$ was proposed by leveraging the discrete gradient formula extensively. It should be pointed out that their discrete gradient formula of $\bm Q^{\alpha}$ inevitably leads to a nonlinear algebraic relation for the internal state variables, necessitating an iterative method of the Newton-Raphson type at each quadrature point during the integration of the constitutive relations. In this regard, using the discrete gradient formula is not particularly advantageous to Simo's collocation approach (see our discussion in Section \ref{subsec:remark-on-two-alternative-approaches}). Another drawback is that the quotient form of the discrete gradient formula is a potential source of numerical instability and demands special care in practice \cite{Liu2023,Mohr2008}. In this study, the linearity between the internal state variables $\bm \Gamma^{\alpha}$ and their conjugate counterparts $\bm Q^{\alpha}$ is exploited, and two update formulas are derived that circumvent the two difficulties conveniently.

Although the focus of this study is primarily on the integration of the balance and constitutive equations, some recently developed spatial discretization techniques have been adopted. Following Part \rom{1}, the variational problems are formulated in terms of a finite strain generalization of the Herrmann variational principle \cite{Herrmann1965,Liu2018}, and the smooth generalization of the Taylor-Hood element is adopted as the element technology. The grad-div stabilization technique is supplemented to enhance the element pairs' performance in terms of discrete mass conservation \cite{Liu2019a,Liu2021b,Guan2023}. Indeed, it has been proven that this combination of technology leads to a discrete divergence-free velocity field asymptotically in fluid and transport problems \cite{Olshanskii2002,Colomes2016}. In the meantime, the stabilization term respects the underlying physical structure as it is energy dissipative and momentum conserving. It thus constitutes an appealing ingredient in the overall proposed numerical methodology.

The structure of the work is organized as follows. The continuum formulation of the problem is briefly summarized in Section \ref{sec:viscoelastodynamics}. Following that, the spatial discretization with grad-div stabilization is performed to obtain the semi-discrete formulation in Section \ref{subsec:semi-discrete}. In Section \ref{sec:time-integration}, the time integration scheme is designed and analyzed. After a general discussion on the design criteria for energy-momentum consistent temporal schemes, a constitutive integration algorithm together with an algorithmic second Piola-Kirchhoff stress is provided to demonstrate the main numerical design strategies. Following that, a more sophisticated formula is constructed in Section \ref{sec:second-order-update-formula-ISV}, which improves the temporal accuracy from first-order to second-order. In Section \ref{subsec:remark-on-two-alternative-approaches}, the connection with and superiority over a previously existing formula are discussed. A non-singular evolution equation for nonlinear viscoelastic models is proposed. In Section \ref{sec:implementation}, the implementation details are outlined. A special predictor in the Newton-Raphson iteration is proposed, which helps the segregation of the implicit algorithm without losing consistency. In Section \ref{sec:numerical_examples}, a set of numerical examples is presented, and we draw conclusions in Section \ref{sec:conclusion}. In \ref{sec:appendix-A} and \ref{sec:appendix-B}, we provide implementation details of the algorithmic elasticity tensors. In \ref{sec:appendix-C}, a glossary of terms is provided to facilitate the understanding of this work.

\section{Viscoelastodynamics}
\label{sec:viscoelastodynamics}
In this section, we present the continuum model based on which we construct the numerical algorithms. More details about the thermomechanical background can be found in Part \rom{1}. Following the notations used in Part \rom{1}, the norms of a vector $\bm W$ and a rank-two tensor $\bm A$ are, respectively, defined as 
\begin{align*}
\left \lvert \bm W \right\rvert := \left( \bm W \cdot \bm W \right)^{\frac12} \quad \mbox{ and } \quad \left \lvert \bm A \right\rvert := \left( \mathrm{tr}[\bm A \bm A^T] \right)^{\frac12} = \left( \bm A : \bm A \right)^{\frac12}.
\end{align*}

\subsection{Continuum basis}
\label{sec:continuum-basis}
In this section, we briefly review the continuum basis for finite deformation linear viscoelasticity that was derived in \cite{Liu2021}. Let $\Omega_{\bm X}$ and $\Omega_{\bm x}^t$ be bounded open sets in $\mathbb R^{3}$ with Lipschitz boundary denoted by $\Gamma_{\bm X}$ and $\Gamma_{\bm x}^t$. The unit outward normal vectors to $\Gamma_{\bm X}$ and $\Gamma_{\bm x}^t$ are denoted as $\bm N$ and $\bm n$, respectively. The motion of $\Omega_{\bm X}$ is given by a family of invertible, smooth, and orientation-preserving mappings parameterized by the time coordinate $t$,
\begin{align*}
\bm \varphi_t( \cdot ) := \bm \varphi(\cdot, t) : \Omega_{\bm X} \rightarrow \Omega^t_{\bm x} := \bm \varphi(\Omega_{\bm X},t) = \bm \varphi_t(\Omega_{\bm X}), \quad \forall t \geq 0, \qquad \bm X \mapsto \bm x = \bm \varphi(\bm X,t) = \bm \varphi_t(\bm X),
\end{align*}
in which $\bm X \in \Omega_{\bm X}$ and $\bm x \in \Omega^t_{\bm x}$ represent the material and spatial points, respectively. It is convenient to regard $\bm X$ as the original location of a material particle currently located at $\bm x$, which demands $\bm \varphi(\cdot, 0)$ to be an identity map. The displacement and velocity of the material particle originally located at $\bm X$ are, respectively, defined as $\bm U(\bm X,t) := \bm \varphi(\bm X,t) - \bm \varphi(\bm X, 0) = \bm \varphi(\bm X, t) - \bm X$ and $\bm V := d\bm U/dt$, where $d(\cdot)/dt$ denotes the total time derivative. Let $\rho_0 : \Omega_{\bm X} \rightarrow \mathbb R_{+}$ be the reference density, and the total linear momentum $\bm L$ and total angular momentum $\bm J$ are defined as
\begin{align*}
\bm L(t) := \int_{\Omega_{\bm X}} \rho_0 \bm V d\Omega_{\bm X}, \qquad \bm J(t) := \int_{\Omega_{\bm X}} \rho_0 \bm \varphi_t \times  \bm V d\Omega_{\bm X}.
\end{align*}
Assuming $\bm \varphi_t$ is sufficiently regular (i.e., at least $C^1$), the deformation gradient $\bm F$, Cauchy-Green deformation tensor $\bm C$, Jacobian determinant $J$, and distortional parts of $\bm F$ and $\bm C$ are given by
\begin{align*}
\bm F := \frac{\partial \bm \varphi_t}{\partial \bm X}, \quad \bm C := \bm F^T \bm F, \quad J := \mathrm{det}(\bm F), \quad \tilde{\bm F} := J^{-\frac13} \bm F, \quad \tilde{\bm C} := J^{-\frac23} \bm C.
\end{align*}
It is known that $\partial \tilde{\bm C}/\partial \bm C = J^{-\frac23} \mathbb P^{T}$ with $\mathbb P := \mathbb I - \frac13 \bm C^{-1} \otimes \bm C$ and $\mathbb I$ being the rank-four identity tensor. We denote the thermodynamic pressure of the continuum body on the reference configurations by $P$, and its counterpart on the current configuration is given by $p = P \circ \bm \varphi_t^{-1}$. A set of strain-like internal state variables $\lbrace \bm \Gamma^{\alpha} \rbrace_{\alpha=1}^{m}$ is introduced to characterize the inelastic deformation. The mechanical property of a viscoelastic material can be described via a Gibbs free energy $G(\tilde{\bm C}, P, \bm \Gamma^1, \cdots, \bm \Gamma^{m})$, which can be additively split into volumetric, isochoric equilibrium, and isochoric non-equilibrium parts,
\begin{align}
\label{eq:gibbs-energy}
G(\tilde{\bm C}, P, \bm \Gamma^1, \cdots, \bm \Gamma^{m}) =&  G^{\infty}_{\mathrm{vol}}(P) + G_{\mathrm{iso}}(\tilde{\bm C}, \bm \Gamma^{\alpha}, \cdots, \bm \Gamma^m)
= G^{\infty}_{\mathrm{vol}}(P) + G_{\mathrm{iso}}^{\infty}(\tilde{\bm C}) + \sum_{\alpha=1}^{m}\Upsilon^{\alpha}\left( \tilde{\bm C}, \bm \Gamma^{\alpha} \right).
\end{align}
In the above, $G^{\infty}_{\mathrm{vol}}$ and $G_{\mathrm{iso}}^{\infty}$ represent, respectively, the volumetric and isochoric hyperelastic material behavior at the equilibrium state as the time approaches infinity; $\Upsilon^{\alpha}$ characterize the dissipative, non-equilibrium behavior. Readers may refer to \cite[Section~2.3]{Liu2021} and references therein for a discussion on the rationale behind the energy split. The finite linear viscoelastic material is characterized by the following particular form of $\Upsilon^{\alpha}$,
\begin{align}
\Upsilon^{\alpha}\left(\tilde{\bm C}, \bm \Gamma^{\alpha} \right) &=  \mu^{\alpha} \left\lvert \frac{\bm \Gamma^{\alpha} - \bm I}{2} \right\rvert^2 + \left( \hat{\bm S}^{\alpha}_{0} - \tilde{\bm S}^{\alpha}_{\mathrm{iso}} \right) : \frac{\bm \Gamma^{\alpha} - \bm I}{2} + \frac{1}{4\mu^{\alpha}} \left\lvert \tilde{\bm S}^{\alpha}_{\mathrm{iso}} - \hat{\bm S}^{\alpha}_0 \right\lvert^2 \nonumber \\
 \label{eq:def-flv-Upsilon}
 &= \frac{1}{4\mu^{\alpha}} \left\lvert \tilde{\bm S}^{\alpha}_{\mathrm{iso}} - \hat{\bm S}^{\alpha}_{0} - \mu^{\alpha} \left( \bm \Gamma^{\alpha} - \bm I \right)\right\lvert^2,
\end{align}
in which
\begin{align}
\label{eq:def-flv-Upsilon-terms}
 \tilde{\bm S}^{\alpha}_{\mathrm{iso}} := 2 \frac{\partial G^{\alpha}(\tilde{\bm C})}{\partial \tilde{\bm C}}, \qquad \hat{\bm S}^{\alpha}_{0} :=& \tilde{\bm S}^{\alpha}_{\mathrm{iso} \: 0} - \bm Q^{\alpha}_{0}, \qquad  \tilde{\bm S}^{\alpha}_{\mathrm{iso} \: 0} := \tilde{\bm S}^{\alpha}_{\mathrm{iso}}|_{t=0} \qquad \bm Q^{\alpha}_0 := \bm Q^{\alpha}|_{t=0},
\end{align}
Here, $G^{\alpha}(\tilde{\bm C})$ are scalar-valued energy-like functions to be provided for material modeling. Given the energy, one may introduce a set of stress-like variables $\bm Q^{\alpha}$ conjugating to $\bm \Gamma^{\alpha}$,
\begin{align}
\label{eq:def-Q}
\bm Q^{\alpha} :=& -2 \frac{\partial G(\tilde{\bm C}, P, \bm \Gamma^1, \cdots, \bm \Gamma^{m})}{\partial \bm \Gamma^{\alpha}} = -2 \frac{\partial \Upsilon^{\alpha}(\tilde{\bm C}, \bm \Gamma^{\alpha})}{\partial \bm \Gamma^{\alpha}} = \tilde{\bm S}^{\alpha}_{\mathrm{iso}} - \hat{\bm S}^{\alpha}_{0} - \mu^{\alpha} \left( \bm \Gamma^{\alpha} - \bm I \right).
\end{align}
The constitutive relations for the density $\rho$, the deviatoric and volumetric parts of the Cauchy stress, and the conjugate variables $\bm Q^{\alpha}$ were systematically derived in \cite{Liu2021} and are stated as follows,
\begin{gather}
\label{eq:constitutive-relations}
\rho(p) = \rho_0 \left( \frac{dG_{\mathrm{vol}}}{dP} \right)^{-1} \circ \bm \varphi_t^{-1}, \quad \bm \sigma_{\mathrm{dev}} := J^{-1} \tilde{\bm F} \left( \mathbb P : \tilde{\bm S} \right) \tilde{\bm F}^T, \quad \frac13 \mathrm{tr}\left[\bm \sigma \right] = -p, \displaybreak[2] \\
\label{eq:split-fictitious-S}
\tilde{\bm S} := 2 \frac{\partial G}{\partial \tilde{\bm C}} = \tilde{\bm S}^{\infty}_{\mathrm{iso}} + \sum_{\alpha=1}^{m} \tilde{\bm S}^{\alpha}_{\mathrm{neq}}, \quad \tilde{\bm S}^{\infty}_{\mathrm{iso}} := 2 \frac{\partial G^{\infty}_{\mathrm{iso}}}{\partial\tilde{\bm C}}, \displaybreak[2] \\
\label{eq:non-equilbrium-fictitious-S-def}
\tilde{\bm S}^{\alpha}_{\mathrm{neq}} := 2 \frac{\partial \Upsilon^{\alpha}}{\partial\tilde{\bm C}} = \frac{1}{\mu^{\alpha}} \frac{\partial \tilde{\bm S}^{\alpha}_{\mathrm{iso}}}{\partial \tilde{\bm C}} : \left(\tilde{\bm S}^{\alpha}_{\mathrm{iso}} - \hat{\bm S}^{\alpha}_{0} - \mu^{\alpha} \left( \bm \Gamma^{\alpha} - \bm I \right) \right) =  \frac{1}{\mu^{\alpha}} \frac{\partial \tilde{\bm S}^{\alpha}_{\mathrm{iso}}}{\partial \tilde{\bm C}} : \bm Q^{\alpha}, \displaybreak[2] \\
\label{eq:constitutive-flv-Q}
 \bm Q^{\alpha} = \mathbb V^{\alpha} : \left( \frac12 \frac{d}{dt}\bm \Gamma^{\alpha} \right) = \eta^{\alpha} \frac{d \bm \Gamma^{\alpha}}{dt}.
\end{gather}
In the above, $\tilde{\bm S}$ is known as the fictitious second Piola-Kirchhoff stress, $\tilde{\bm S}^{\infty}_{\mathrm{iso}}$ and $\tilde{\bm S}^{\alpha}_{\mathrm{neq}}$, respectively, represent the equilibrium and non-equilibrium parts of $\tilde{\bm S}$, and $\mathbb V^{\alpha}$ is a rank-four, positive semi-definite viscosity tensor. Here, $\mathbb V^{\alpha} = 2 \eta^{\alpha} \mathbb I$, and $\eta^{\alpha}>0$ are the viscosities. The second Piola-Kirchhoff stress $\bm S$ can be written as
\begin{align}
\label{eq:constitutive_S}
& \bm S := \bm S_{\mathrm{iso}} + \bm S_{\mathrm{vol}}, \displaybreak[2] \\
\label{eq:constitutive_S_vol}
& \bm S_{\mathrm{vol}} := \frac13 \mathrm{tr}\left[ \bm \sigma \right] J \bm F^{-1} \bm F^{-T} = -J P \bm C^{-1}, \displaybreak[2] \\
\label{eq:constitutive_S_iso}
& \bm S_{\mathrm{iso}} := J \bm F^{-1} \bm \sigma_{\mathrm{dev}} \bm F^{-T} = J^{-\frac23} \mathbb P : \tilde{\bm S} = J^{-\frac23} \mathbb P : \left(  \tilde{\bm S}^{\infty}_{\mathrm{iso}} + \sum_{\alpha=1}^{m} \tilde{\bm S}^{\alpha}_{\mathrm{neq}} \right) = \bm S^{\infty}_{\mathrm{iso}} + \sum_{\alpha=1}^{m} \bm S^{\alpha}_{\mathrm{neq}}, \displaybreak[2] \\
\label{eq:constitutive_S_infty_iso_S_neq_alpha}
& \bm S^{\infty}_{\mathrm{iso}} := J^{-\frac23} \mathbb P : \tilde{\bm S}^{\infty}_{\mathrm{iso}}, \quad \mbox{and} \quad \bm S^{\alpha}_{\mathrm{neq}} := J^{-\frac23} \mathbb P : \tilde{\bm S}^{\alpha}_{\mathrm{neq}}.
\end{align}
To facilitate subsequent discussions, we introduce $\bm P_{\mathrm{iso}}$ as
\begin{align}
\label{eq:constitutive-P-iso}
\bm P_{\mathrm{iso}} := \bm F \bm S_{\mathrm{iso}} = J \bm \sigma_{\mathrm{dev}} \bm F^{-T}.
\end{align}

\begin{remark}
A more general form of finite linear viscoelasticity model was considered in Part \rom{1}, whose configurational free energy takes the form
\begin{align}
\label{eq:general_form_Upsilon}
\Upsilon^{\alpha}\left(\tilde{\bm C}, \bm \Gamma^{\alpha} \right) =  \mu^{\alpha} \left\lvert \frac{\bm \Gamma^{\alpha} - \bm I}{2} \right\rvert^2 + \left( \hat{\bm S}^{\alpha}_{0} - \tilde{\bm S}^{\alpha}_{\mathrm{iso}} \right) : \frac{\bm \Gamma^{\alpha} - \bm I}{2} + F^{\alpha}(\tilde{\bm C}).
\end{align}
However, a major issue is that the energy \eqref{eq:general_form_Upsilon} does not guarantee the vanishment of the non-equilibrium stress $\bm S^{\alpha}_{\mathrm{neq}}$ in the thermodynamic equilibrium limit. In Part \rom{1}, it is shown that the necessary condition for the vanishment of $\bm S^{\alpha}_{\mathrm{neq}}$ is
\begin{align*}
F^{\alpha}(\tilde{\bm C}) \Big|_{\mathrm{eq}} = \left( \frac{1}{4\mu^{\alpha}} \left\lvert \tilde{\bm S}^{\alpha}_{\mathrm{iso}} - \hat{\bm S}^{\alpha}_0 \right\lvert^2  \right)\Big|_{\mathrm{eq}},
\end{align*}
where the notation $\big|_{\mathrm{eq}}$ indicates that equality holds at the thermodynamic equilibrium limit. Consequently, the configurational free energy \eqref{eq:def-flv-Upsilon} is obtained by simply taking 
\begin{align}
\label{eq:general_form_F_alpha}
F^{\alpha}(\tilde{\bm C}) = \frac{1}{4\mu^{\alpha}} \left\lvert \tilde{\bm S}^{\alpha}_{\mathrm{iso}} - \hat{\bm S}^{\alpha}_0 \right\lvert^2,
\end{align}
which automatically ensures the relaxation of $\bm S^{\alpha}_{\mathrm{neq}}$. It can also be shown that the configurational free energy \eqref{eq:def-flv-Upsilon} can be reorganized as
\begin{align*}
\Upsilon^{\alpha}\left(\tilde{\bm C}, \bm \Gamma^{\alpha} \right) = \frac{1}{4\mu^{\alpha}} \left\lvert \bm Q^{\alpha} \right\rvert^2,
\end{align*}
by invoking the constitutive relation \eqref{eq:def-Q}.
\end{remark}

\begin{remark}
Choosing $F^{\alpha} = G^{\alpha}$ in \eqref{eq:general_form_Upsilon} recovers the configurational free energy introduced in \cite{Holzapfel1996}. Based on the analysis in Part \rom{1}, that model cannot guarantee the full relaxation of the non-equilibrium stress, except in the special case of
\begin{align*}
G^{\alpha} = \mu^{\alpha} \left\lvert \frac{\tilde{\bm C} - \bm I}{2} \right\lvert^2.
\end{align*}
Following the terminology introduced in Part \rom{1}, we refer to models using the above form of $G^{\alpha}$ as the Holzapfel-Simo-Saint Venant-Kirchhoff (HS) model. The identical-polymer-chain assumption was often utilized to design the energy as 
\begin{align}
\label{eq:G_alpha_def_MIPC}
G^{\alpha} = \beta_{\alpha}^{\infty}G^{\infty}_{\mathrm{iso}},
\end{align}
where $\beta^{\infty}_{\alpha} \in (0,+\infty)$ are non-dimensional parameters known as the strain-energy factors \cite{Govindjee1992}. However, the original design takes $F^{\alpha} = G^{\alpha}$ \cite{Holzapfel1996,Martin2014}, resulting in a pathological definition of the non-equilibrium stress. In Part \rom{1}, a rectification was made by constructing $F^{\alpha}$ following \eqref{eq:general_form_F_alpha}, which is referred to as the modified identical-polymer-chain (MIPC) model.
\end{remark}

\subsection{Evolution equations}
Combing the definition of $\bm Q^{\alpha}$ \eqref{eq:def-Q} and its constitutive relation \eqref{eq:constitutive-flv-Q}, we have a suite of equations governing the evolution of $\bm \Gamma^{\alpha}$,
\begin{align*}
\bm Q^{\alpha} = \eta^{\alpha} \frac{d\bm \Gamma^{\alpha}}{dt} = \tilde{\bm S}^{\alpha}_{\mathrm{iso}} - \hat{\bm S}^{\alpha}_{0} - \mu^{\alpha} \left( \bm \Gamma^{\alpha} - \bm I \right).
\end{align*}
The above equations can be reorganized and written as the evolution equations for the inelastic strain, 
\begin{align}
\label{eq:evolution-eqn-gamma}
\boxed{
\eta^{\alpha} \frac{d}{dt} \left( \frac{\bm \Gamma^{\alpha} - \bm I}{2} \right) + \mu^{\alpha} \left(\frac{ \bm \Gamma^{\alpha} - \bm I}{2} \right) = \frac12 \left( \tilde{\bm S}^{\alpha}_{\mathrm{iso}} - \hat{\bm S}^{\alpha}_{0} \right).}
\end{align}
Taking time derivatives at both sides of \eqref{eq:evolution-eqn-gamma} leads to
\begin{align}
\label{eq:evolution-eqn-gamma-strain}
\frac{d \bm Q^{\alpha}}{dt} = \frac{d \tilde{\bm S}^{\alpha}_{\mathrm{iso}}}{dt} - \mu^{\alpha} \frac{d\Gamma^{\alpha}}{dt} = \frac{d \tilde{\bm S}^{\alpha}_{\mathrm{iso}}}{dt} - \frac{\mu^{\alpha}}{\eta^{\alpha}} \bm Q^{\alpha},
\end{align}
which recovers the classical evolution equations for $\bm Q^{\alpha}$,
\begin{align}
\label{eq:evolution-eqn-Q}
\frac{d}{dt} \bm Q^{\alpha} + \frac{\mu^{\alpha}}{\eta^{\alpha}}\bm Q^{\alpha} = \frac{d}{dt} \tilde{\bm S}^{\alpha}_{\mathrm{iso}}.
\end{align}
Based on the evolution equations \eqref{eq:evolution-eqn-Q}, recursive stress update formulas have been designed as an integration rule for the constitutive relations \cite{Simo2006}. Yet, the evolution equations for the strain-like variables \eqref{eq:evolution-eqn-gamma} (or, equivalently, \eqref{eq:evolution-eqn-gamma-strain}) remain largely unnoticed in the literature. In general, strain-driven formulations have been advocated to be superior to the stress-driven ones, especially in the computation of plasticity \cite{Simo2006}. One may conveniently devise an analogous recurrence update formula for $\bm \Gamma^{\alpha}$. Instead of doing that, we resort to a different constitutive integration approach that preserves physical features in this study.

\subsection{Governing equations}
We consider a non-overlapping subdivision of the boundary $\Gamma_{\bm X} := \partial \Omega_{\bm X}$ into $\Gamma^{G}_{\bm X} \cup \Gamma^{H}_{\bm X} = \Gamma_{\bm X}$. The system of equations governing the motion of the body posed on the initial configuration is stated as follows,
\begin{align}
\label{eq:strong-kinematics}
& \frac{d\bm U}{dt} - \bm V = \bm 0, && \mbox{ in } \Omega_{\bm X} \times (0,T), \displaybreak[2] \\
\label{eq:strong-mass}
& J \beta_{\theta}(P) \frac{dP}{dt} + \nabla_{\bm X} \bm V : (J \bm F^{-T}) = 0, && \mbox{ in } \Omega_{\bm X} \times (0,T), \displaybreak[2] \\
\label{eq:strong-momentum}
& \rho_0 \frac{d\bm V}{dt} - \nabla_{\bm X} \cdot \bm P_{\mathrm{iso}} + \nabla_{\bm X} \cdot (JP\bm F^{-T}) - \rho_0 \bm B = \bm 0, && \mbox{ in } \Omega_{\bm X} \times (0,T), \displaybreak[2] \\
& \bm U = \bm G \quad \mbox{ and }  \quad \bm V = \frac{\partial \bm G}{\partial t}, && \mbox{ on } \Gamma^{G}_{\bm X} \times (0,T), \\
& \bm P \bm N = \bm H, && \mbox{ on } \Gamma^{H}_{\bm X} \times (0,T), \displaybreak[2] \\
\label{eq:strong-initial}
& \bm U(\bm x, 0) = \bm U_0(\bm x), \quad P(\bm x, 0) = P_0(\bm x), \quad \bm V(\bm x, 0) = \bm V_0(\bm x), && \mbox{ in } \Omega_{\bm X}.
\end{align}
In the above, $\bm G$ is the prescribed boundary displacement defined on $\bm \Gamma^{G}_{\bm X}$, $\bm H$ is the prescribed traction on $\bm \Gamma_{\bm X}^{H}$, $\bm U_0$, $\bm V_0$, and $P_0$ are the initial displacement, velocity, and pressure data, respectively. In the incompressible limit, the volumetric part of the free energy becomes $G_{\mathrm{vol}} = P$ and $\beta_{\theta}$ is zero, resulting in a divergence-free constraint on the velocity field. In the following, we will restrict our discussion to the fully incompressible scenario.

\section{Semi-discrete formulation}
\label{subsec:semi-discrete}
In this section, we perform spatial discretization for the strong-form problem \eqref{eq:strong-kinematics}-\eqref{eq:strong-initial} using a smooth generalization of the Taylor-Hood element based on Non-Uniform Rational B-Splines (NURBS). For a single-patch geometry, we may construct the spline spaces based on a parametric domain $\hat{\Omega} := (0,1)^3$. A set of open knot vectors $\Xi_d$, $d=1,2,3$, is provided as a Cartesian mesh for $\hat{\Omega}$. With a set of weights, the space of the multivariate NURBS functions can be determined and is denoted by $\mathcal R^{\mathsf p_1, \mathsf p_2, \mathsf p_3}_{ \bm \alpha_1, \bm \alpha_2, \bm \alpha_3}$, where $\mathsf p_d$ and $\bm \alpha_d$ represent the degree and interelement regularity in the $d$-th parametric direction. The length of $\bm \alpha_d$ equals $\mathsf m_d$, the number of unique knots in $\Xi_d$. The values of $\mathsf p_d$ and $\bm \alpha_d$ are completely determined by the knot vectors $\Xi_d$. For an open-knot vector $\Xi_d$, the interelement regularity vector $\bm \alpha_d := \{\alpha_1, \alpha_2, \cdots, \alpha_{\mathsf m_d-1}, \alpha_{\mathsf m_d}\}$ has $\alpha_1 = \alpha_{\mathsf m_d} = -1$.  A detailed description on the construction of the function space $\mathcal R^{\mathsf p_1, \mathsf p_2, \mathsf p_3}_{ \alpha_1, \alpha_2, \alpha_3}$ can be found in \cite{Buffa2011}. In this work, we consider the following two NURBS function spaces,
\begin{align*}
\hat{\mathcal S}_h :=
\mathcal R^{\mathsf p+\mathsf a, \mathsf p+\mathsf a, \mathsf p+\mathsf a}_{\bm \alpha_1+\mathsf b, \bm \alpha_2+\mathsf b, \bm \alpha_3+\mathsf b} \times \mathcal R^{\mathsf p+\mathsf a, \mathsf p+\mathsf a, \mathsf p+\mathsf a}_{\bm \alpha_1+\mathsf b, \bm \alpha_2+\mathsf b, \bm \alpha_3+\mathsf b} \times \mathcal R^{\mathsf p+\mathsf a, \mathsf p+\mathsf a, \mathsf p+\mathsf a}_{\bm \alpha_1+\mathsf b, \bm \alpha_2+\mathsf b, \bm \alpha_3+\mathsf b}, \quad
\hat{\mathcal P}_h := \mathcal R^{\mathsf p, \mathsf p, \mathsf p}_{\bm \alpha_1, \bm \alpha_2, \bm \alpha_3},
\end{align*}
where $\bm \alpha_d + \mathsf b:= \{\alpha_1, \alpha_2+\mathsf b, \cdots, \alpha_{\mathsf m_d-1} +\mathsf b, \alpha_{\mathsf m_d}\}$. We demand that the integer parameters $\mathsf a$ and $\mathsf b$ satisfy $1 \leq \mathsf a $ and $0 \leq \mathsf b < \mathsf a$. Given a set of knot vectors, the space $\hat{\mathcal P}_h$ is generated first with a uniform polynomial degree $\mathsf p$ in all three directions. In this work, unless otherwise specified, the space $\hat{\mathcal P}_h$ is generated with the maximum possible continuity. Based on $\hat{\mathcal P}_h$, two steps of refinement are performed to construct $\hat{\mathcal S}_h$. One first performs a $k$-refinement to raise the polynomial degree and continuity by $\mathsf b$ and then performs a $p$-refinement to raise the polynomial degree to $\mathsf p + \mathsf a$ while maintaining the continuity. Numerical tests reveal that $\mathsf b + 1 \leq \mathsf a$ is necessary for the inf-sup stability of the above function spaces \cite{Liu2019a}, meaning the $p$-refinement is necessary for generating stable element pairs. 

Given a smooth geometrical mapping $\bm \psi : \hat{\Omega} \rightarrow \Omega_{\bm X}$ with piecewise smooth inverse, the discrete function spaces on $\Omega_{\bm X}$ can be defined through the pull-back operation,
\begin{align*}
& \mathcal S_h :=  \{ \bm w : \bm w \circ \bm \psi \in \hat{\mathcal S}_h \}, \quad \mathcal P_h := \{ q : q \circ \bm \psi \in \hat{\mathcal P}_h \}.
\end{align*}
In this work, the mapping $\bm \psi$ is constructed by functions in $\hat{\mathcal S}_h$, rendering $\mathcal S_h$ to be an \textit{isoparametric} discrete space. The trial solution spaces for the displacement, pressure, and velocity can be specified as
\begin{align*}
\mathcal S_{\bm U_h} &= \Big\lbrace \bm U_h : \bm U_h(\cdot, t) \in \mathcal S_h, \quad \bm U_h(\cdot,t) = \bm G \mbox{ on } \Gamma_{\bm X}^{G}, \quad t \in [0,T]  \Big\rbrace , \displaybreak[2] \\
\mathcal S_{P_h} &= \Big\lbrace P_h : P_h(\cdot, t) \in \mathcal P_h, \quad t \in [0,T] \Big\rbrace , \displaybreak[2] \\
 \mathcal S_{\bm V_h} &= \left\lbrace \bm V_h : \bm V_h(\cdot, t) \in \mathcal S_h, \quad \bm V_h(\cdot,t) = \frac{d\bm G}{dt} \mbox{ on } \Gamma_{\bm X}^{G}, \quad t \in [0,T] \right\rbrace ,
\end{align*}
and the test function spaces are given by
\begin{align*}
\mathcal V_{P_h} &=\mathcal S_{P_h} \quad
\mathcal V_{\bm V_h} = \left\lbrace \bm W_h : \bm W_h(\cdot, t) \in \mathcal S_h, \quad \bm W_h(\cdot,t) = \bm 0 \mbox{ on } \Gamma_{\bm X}^{G}, \quad t \in [0,T] \right\rbrace.
\end{align*}
Given a displacement $\bm U_h \in \mathcal S_{\bm U_h}$, the placement field is given by $\bm \varphi_{t,h} = \bm U_h + \bm X$.  With the above discrete function spaces defined, the semi-discrete formulation on the referential configuration can be stated as follows. Find $\bm Y_h(t) := \left\lbrace \bm U_h(t), P_h(t), \bm V_h(t)\right\rbrace^T \in \mathcal S_{\bm U_h} \times \mathcal S_{P_h} \times \mathcal S_{\bm V_h}$ such that
\begin{align}
\label{eq:mix_solids_kinematics_ref}
& \bm 0 = \mathbf B^k\left( \dot{\bm Y}_h, \bm Y_h \right) :=  \frac{d\bm U_h}{dt} - \bm V_h, \displaybreak[2]\\
\label{eq:mix_solids_mass_ref}
& 0 = \mathbf B^p\left( Q_h; \dot{\bm Y}_h, \bm Y_h \right) := \int_{\Omega_{\bm X}} Q_h J_h \nabla_{\bm X} \bm V_h : \bm F^{-T}_h d\Omega_{\bm X}, \displaybreak[2] \\
\label{eq:mix_solids_momentum_ref}
& 0 = \mathbf B^m\left( \bm W_h; \dot{\bm Y}_h, \bm Y_h \right) := \int_{\Omega_{\bm X}} \Big( \bm W_h \cdot \rho_0 \frac{d\bm V_h}{dt} + \left( \bm F^T_h \nabla_{\bm X} \bm W_h \right) :  \bm S_{\mathrm{iso}} - J_h P_h \nabla_{\bm X} \bm W_h : \bm F^{-T}_h - \bm W_h \cdot \rho_0  \bm B \Big)  d\Omega_{\bm X} \displaybreak[2] \nonumber \\
& \hspace{3.9cm} - \int_{\Gamma_{\bm X}^{H}} \bm W_h \cdot \bm H d\Gamma_{\bm X} + \int_{\Omega_{\bm X}} \gamma J_h \left( \nabla_{\bm X} \bm W_h : \bm F^{-T}_{h} \right)\left( \nabla_{\bm X} \bm V_h : \bm F^{-T}_{h} \right)d\Omega_{\bm X}, 
\end{align}
for $\forall \left\lbrace  Q_h, \bm W_h \right\rbrace \in \mathcal V_{P_h} \times \mathcal V_{\bm V_h}$, with $\bm Y_h(0) = \left\lbrace \bm U_{h0}, P_{h0}, \bm V_{h0} \right\rbrace^T$. The initial data are obtained through the $\mathcal L^2$-projection of the initial data onto the discrete trial solution spaces. The last term in \eqref{eq:mix_solids_momentum_ref} involves a parameter $\gamma \geq 0$ and introduces the grad-div stabilization. The parameter $\gamma$ in general varies as a function of $\bm X$ and $t$ \cite{Olshanskii2009}. 

If the boundary data $\bm G$ is time-independent, one may set $\bm W_h = \bm V_h$ and $Q_h = P_h$ in \eqref{eq:mix_solids_mass_ref} and \eqref{eq:mix_solids_momentum_ref}, leading to \cite[p.~19]{Liu2021}
\begin{align}
\label{eq:semi-discrete-energy-stability}
&\frac{d}{dt} \int_{\Omega_{\bm X}} \rho_0 \|\bm V_h\|^2 + G(\tilde{\bm C}, \bm \Gamma^1, \cdots, \bm \Gamma^m) d\Omega_{\bm X} = \int_{\Omega_{\bm X}} \rho_0 \bm V_h \cdot \bm B d\Omega_{\bm X} + \int_{\Gamma_{\bm X}^{H}} \bm V_h \cdot \bm H d\Gamma_{\bm X} \displaybreak[2] \nonumber \\
& \qquad - \sum_{\alpha=1}^{m} \int_{\Omega_{\bm X}}\frac14 \left(\frac{d}{dt}\bm \Gamma^{\alpha}_h \right) : \mathbb V^{\alpha} : \left(\frac{d}{dt}\bm \Gamma^{\alpha}_h \right) d\Omega_{\bm X} - \int_{\Omega_{\bm X}} \gamma \left( \nabla_{\bm X} \bm V_h : \bm F^{-T}_{h} \right)^2 d\Omega_{\bm X}.
\end{align}
The above is known as the theorem of expended power or the a priori energy stability of the weak (semi-discrete) formulation. The positive semi-definiteness of the tensors $\mathbb V^{\alpha}$ guarantees that
\begin{align*}
\int_{\Omega_{\bm X}} \frac14 \left(\frac{d}{dt}\bm \Gamma^{\alpha}_h \right) : \mathbb V^{\alpha} : \left(\frac{d}{dt}\bm \Gamma^{\alpha}_h \right) d\Omega_{\bm X} = \int_{\Omega_{\bm X}} \bm Q^{\alpha} : \left( \mathbb V^{\alpha} \right)^{-1} : \bm Q^{\alpha} d\Omega_{\bm X} \geq 0, \quad \mbox{ for } 1 \leq \alpha \leq m,
\end{align*}
which represents the dissipation due to the inelastic deformation. The last term in \eqref{eq:semi-discrete-energy-stability} is a numerical dissipation term due to the fact that the grad-div stabilization parameter $\gamma$ is non-negative by design.

If $\Gamma^{G}_{\bm X} = \emptyset$, the semi-discrete formulation enjoys the following property \cite{Liu2021},
\begin{align}
\frac{d}{dt}\bm L_h :=& \frac{d}{dt} \int_{\Omega_{\bm X}} \rho_0 \bm V_h d\Omega_{\bm X} = \int_{\Omega_{\bm X}} \rho_0 \bm B d\Omega_{\bm X} + \int_{\Gamma_{\bm X}^{H}} \bm H d\Gamma_{\bm X}, \\
\label{eq:semi-discrete-angular-momentum-conservation}
\frac{d}{dt}\bm L_h :=& \frac{d}{dt} \int_{\Omega_{\bm X}} \rho_0 \bm \varphi_{t~h} \times \bm V_h d\Omega_{\bm X} = \int_{\Omega_{\bm X}} \rho_0 \bm \varphi_{t~h} \times \bm B d\Omega_{\bm X} + \int_{\Gamma_{\bm X}^{H}} \bm \varphi_{t~h} \times \bm H d\Gamma_{\bm X}.
\end{align}
The above gives the conservation of momenta in the semi-discrete formulation. In particular, we mention that the grad-div stabilization term does not harm the momentum conservation. The primary objective of this work is to design integration schemes that inherit the above properties to the fully discrete level, and we will focus on the design strategies in the next section.

\begin{remark}
The grad-div stabilization term was originally introduced in the context of computational fluid dynamics as a sub-grid scale model for pressure \cite{Olshanskii2009}. It renders the solution of the Taylor-Hood element approaching that of the Scott-Vogelius element \cite{Scott1985}, an element generating discrete divergence-free velocity, as the stabilization parameter $\gamma$ goes to infinity \cite{Case2011}. Therefore, it is an appealing technique for enhancing the discrete mass conservation in flow and transport problems \cite{Colomes2016,John2010,Olshanskii2002}. In this study, the grad-div stabilization is introduced with the expectation of enhancing the performance of the smooth generalization of the Taylor-Hood element built based on NURBS. As can be seen from the above propositions, this term does not harm the momentum conservation property and is numerically dissipative in energy. Also, due to the analysis made in \cite{Case2011}, the dissipation due to the grad-div stabilization (i.e., the last term in \eqref{eq:semi-discrete-energy-stability}) annihilates as $\gamma$ goes to infinity, at least for the Navier-Stokes equations. In practice, the parameter $\gamma$ cannot be chosen arbitrarily large as over-stabilization results in difficulty in achieving convergence in the solver. 
\end{remark}

\section{Time integration}
\label{sec:time-integration}
We construct temporal schemes based on the semi-discrete formulation \eqref{eq:mix_solids_kinematics_ref}-\eqref{eq:mix_solids_momentum_ref} with the primary goal of inheriting the properties \eqref{eq:semi-discrete-energy-stability}-\eqref{eq:semi-discrete-angular-momentum-conservation} to the fully discrete level. To simplify the notation, we will henceforth neglect the subscript $h$ used for spatially discrete quantities. We consider a discrete time vector $\left\lbrace t_n \right\rbrace_{n=0}^{n_{\mathrm{ts}}}$ that delimits the time interval $(0,T)$ into a set of $n_{\mathrm{ts}}$ subintervals of size $\Delta t_n := t_{n+1} - t_{n} > 0$. The solution vector $\bm Y_n := \left\lbrace  \bm U_n, P_n, \bm V_n\right\rbrace^T$ represents the discrete approximations at time $t_n$ to the semi-discrete solution $\bm Y_h(t)$; the approximation of the internal state variables $\bm \Gamma^{\alpha}(t)$ at time $t_n$ are denoted as $\bm \Gamma^{\alpha}_n$. We additionally introduce the following notations,
\begin{gather*}
\bm F_n := \nabla_{\bm X} \bm U_{n} + \bm I, \quad J_n := \mathrm{det}\left( \bm F_n \right), \quad \bm C_n := \bm F^{T}_n \bm F_n, \quad \tilde{\bm C}_n := J^{-\frac23}_n \bm C_n, \quad \bm \Gamma^{\alpha}_{n+\frac12} := \frac12 \left( \bm \Gamma^{\alpha}_{n+1} + \bm \Gamma^{\alpha}_{n} \right), \displaybreak[2] \\
\bm Z_n = \bm C_{n+1} - \bm C_{n}, \quad H_n := \int_{\Omega_{\bm X}} \frac{\rho_0 \|\bm V_{n}\|^2}{2} + G_{\mathrm{iso}}\left( \tilde{\bm C}_{n}, \bm \Gamma^1_n, \cdots \bm \Gamma^{m}_{n} \right) d\Omega_{\bm X}.
\end{gather*}
Here we consider the following single-step implicit scheme,
\begin{align}
\label{eq:em-ts-kinematic}
& \bm 0 = \frac{\bm U_{n+1} - \bm U_{n}}{\Delta t_n} - \bm V_{n+\frac12}, \displaybreak[2] \\
\label{eq:em-ts-mass}
& 0 = \int_{\Omega_{\bm X}} Q J_{n+\frac12} \nabla_{\bm X} \bm V_{n+\frac12} : \bm F^{-T}_{n+\frac12} d\Omega_{\bm X}, \displaybreak[2] \\
\label{eq:em-ts-momentum}
& 0 = \int_{\Omega_{\bm X}} \bm W \cdot \rho_0 \frac{\bm V_{n+1} - \bm V_n}{\Delta t_n} + \left( \bm F^T_{n+\frac12}\nabla_{\bm X} \bm W \right) : \bm S_{\mathrm{iso} \: \mathrm{alg}} - J_{n+\frac12} P_{n+\frac12} \nabla_{\bm X} \bm W : \bm F^{-T}_{n+\frac12} - \bm W \cdot \rho_0 \bm B_{n+\frac12} d\Omega_{\bm X}  \nonumber \displaybreak[2] \\
& \qquad - \int_{\Gamma_{\bm X}^H} \bm W \cdot \bm H_{n+\frac12} d\Gamma_{\bm X} + \int_{\Omega_{\bm X}} \gamma J_{n+\frac12} \left(\nabla \bm_{\bm X} W : \bm F^{-T}_{n+\frac12} \right) \left(\nabla_{\bm X} \bm V_{n+\frac12} : \bm F^{-T}_{n+\frac12} \right) d\Omega_{\bm X},
\end{align}
where 
\begin{gather}
 \left\lbrace  \bm U_{n+\frac12}, P_{n+\frac12}, \bm V_{n+\frac12} \right\rbrace^T :=  \frac12 \left\lbrace \bm U_n + \bm U_{n+1}, P_n + P_{n+1}, \bm V_n + \bm V_{n+1} \right\rbrace^T, \displaybreak[2] \\
 \bm F_{n+\frac12} := \nabla_{\bm X} \bm U_{n+\frac12} + \bm I, \quad J_{n+\frac12} := \mathrm{det}\left( \bm F_{n+\frac12} \right), \displaybreak[2] \\
 \bm C_{n+\frac12} := \frac12 \left( \bm C_{n+1} + \bm C_n \right), \qquad \tilde{\bm C}_{n+\frac12} := J^{-\frac23}_{n+\frac12} \bm C_{n+\frac12}, \displaybreak[2] \\
\label{eq:em-ts-S-alg-def}
 \bm S_{\mathrm{iso} \: \mathrm{alg}} := \bm S_{\mathrm{iso} \: n+\frac12} + \bm S_{\mathrm{iso} \: \mathrm{enh}}, \quad \bm S_{\mathrm{iso} \: n+\frac12} := \bm S_{\mathrm{iso}}\left( \bm C_{n+\frac12}, \bm \Gamma^{1}_{n+\frac12}, \cdots, \bm \Gamma^{m}_{n+\frac12} \right).
\end{gather}
Notice that $\bm C_{n+\frac12} \neq \bm F^{T}_{n+\frac12} \bm F_{n+\frac12}$ in general, and adopting $\bm C_{n+\frac12}$ in the evaluation of $\bm S_{\mathrm{iso} \: n+\frac12}$ is known to preserve the exact relative equilibria \cite{Armero2001,Gonzalez1996,Romero2012}. The discrete scheme \eqref{eq:em-ts-kinematic}-\eqref{eq:em-ts-S-alg-def} is complete once the algorithmic stress $\bm S_{\mathrm{iso} \: \mathrm{alg}}$ is explicitly given. The design of the algorithmic stress here is the discrete stress $\bm S_{\mathrm{iso} \: n+\frac12}$ additively perturbed by a ``stress enhancement" $\bm S_{\mathrm{iso} \: \mathrm{enh}}$, following the terminology of \cite{Mohr2008}. This general formula \eqref{eq:em-ts-S-alg-def} incorporates several widespread algorithmic stress designs in elastodynamics, including the discrete gradient formula \cite{Gonzalez2000} and several more involved formulas \cite{Armero2007,Romero2012}. Yet, there are other options for the algorithmic stress, which do not take the form \eqref{eq:em-ts-S-alg-def}. The scaled mid-point formula \cite{Chorin1978} is designed based on a multiplicative perturbation and maintains the consistency in energy and momentum \cite{Bui2007,Gonzalez1996,Orden2021}. However, the numerical robustness of the scaled mid-point formula has been questioned recently \cite{Guan2023,Romero2012}. In the following, we discuss a set of general design criteria that guarantee the scheme to be energy-momentum consistent for the viscoelasticity models considered in this work.

\begin{lemma}
\label{lemma:ts-energy-dissipation}
Assuming the boundary data $\bm G$ is time-independent, and the algorithmic stress satisfies the following relation, 
\begin{align}
\label{eq:directionality-property}
\bm Z_n : \bm S_{\mathrm{iso} \: \mathrm{alg}} =& G_{\mathrm{iso}}(\tilde{\bm C}_{n+1}, \bm \Gamma^{1}_{n+1}, \cdots, \bm \Gamma^{m}_{n+1}) - G_{\mathrm{iso}}(\tilde{\bm C}_{n}, \bm \Gamma^{1}_{n}, \cdots, \bm \Gamma^{m}_{n}) + \frac{\Delta t_n}{2} \sum_{\alpha=1}^{m} \eta^{\alpha} \left\lvert \frac{ \bm \Gamma^{\alpha}_{n+1} - \bm \Gamma^{\alpha}_n}{\Delta t_n} \right\rvert^2,
\end{align}
with $\bm Z_n := \left( \bm C_{n+1} - \bm C_n \right)/2$, the fully-discrete scheme \eqref{eq:em-ts-kinematic}-\eqref{eq:em-ts-momentum} enjoys the following energy stability property,
\begin{align}
\label{eq:em-prop-1}
\frac{1}{\Delta t_n} \left( H_{n+1} - H_n \right) =& \mathcal P_{\mathrm{ext}} - \mathcal D_{\mathrm{phy}} - \mathcal D_{\mathrm{num}},
\end{align}
with the power of external loadings $\mathcal P_{\mathrm{ext}}$, the physical dissipation $\mathcal D_{\mathrm{phy}}$, and the numerical dissipation $\mathcal D_{\mathrm{num}}$ given by
\begin{align*}
&\mathcal P_{\mathrm{ext}} := \int_{\Omega_{\bm X}} \rho_0 \bm V_{n+\frac12} \cdot \bm B d\Omega_{\bm X} + \int_{\Gamma_{\bm X}^H} \bm V_{n+\frac12} \cdot \bm H d\Gamma_{\bm X}, \displaybreak[2] \\  
& \mathcal D_{\mathrm{phy}} := \frac12 \sum_{\alpha=1}^{m} \int_{\Omega_{\bm X}} \eta^{\alpha} \left\lvert \frac{ \bm \Gamma^{\alpha}_{n+1} - \bm \Gamma^{\alpha}_n}{\Delta t_n} \right\rvert^2 d\Omega_{\bm X}, \displaybreak[2] \\
& \mathcal D_{\mathrm{num}} := \int_{\Omega_{\bm X}} \gamma J_{n+\frac12} \left(\nabla_{\bm X} \bm V_{n+\frac12} : \bm F^{-T}_{n+\frac12} \right)^2 d\Omega_{\bm X}.
\end{align*}
\end{lemma}
\begin{proof}
Since the boundary data $\bm G$ is time-independent, the function spaces $\mathcal S_{\bm V_h}$ and $\mathcal V_{\bm V_h}$ become identical. It is legitimate to choose $\bm W = \bm V_{n+\frac12}$ in \eqref{eq:em-ts-momentum}, which results in
\begin{align}
\label{eq:em-prop-2}
0 =& \int_{\Omega_{\bm X}} \bm V_{n+\frac12} \cdot \rho_0 \frac{\bm V_{n+1} - \bm V_n}{\Delta t_n} + \left( \bm F^T_{n+\frac12}\nabla_{\bm X} \bm V_{n+\frac12} \right) : \bm S_{\mathrm{iso} \: \mathrm{alg}} - J_{n+\frac12} P_{n+\frac12} \nabla_{\bm X} \bm V_{n+\frac12} : \bm F^{-T}_{n+\frac12}  \nonumber \displaybreak[2] \\
& - \bm V_{n+\frac12} \cdot \rho_0 \bm B d\Omega_{\bm X}  - \int_{\Gamma_{\bm X}^H} \bm V_{n+\frac12} \cdot \bm H d\Gamma_{\bm X} + \int_{\Omega_{\bm X}} \gamma J_{n+\frac12} \left(\nabla \bm V_{n+\frac12} : \bm F^{-T}_{n+\frac12} \right)^2 d\Omega_{\bm X}.
\end{align}
The first term in the above results in the jump of the kinetic energy,
\begin{align*}
\int_{\Omega_{\bm X}} \bm V_{n+\frac12} \cdot \rho_0 \frac{\bm V_{n+1} - \bm V_n}{\Delta t_n} d\Omega_{\bm X} = \int_{\Omega_{\bm X}} \rho_0 \frac{|\bm V_{n+1}|^2 - |\bm V_n|^2}{2 \Delta t_n} d\Omega_{\bm X}.
\end{align*}
As for the second term of \eqref{eq:em-prop-2}, we make use of the kinematic relation \eqref{eq:em-ts-kinematic} and get
\begin{align*}
& \int_{\Omega_{\bm X}} \left( \bm F^{T}_{n+\frac12} \nabla_{\bm X} \bm V_{n+\frac12} \right) : \bm S_{\mathrm{iso} \: \mathrm{alg}} d\Omega_{\bm X} = \frac{1}{\Delta t_n} \int_{\Omega_{\bm X}} \left( \bm F^{T}_{n+\frac12} \left( \bm F_{n+1} - \bm F_n \right) \right) : \bm S_{\mathrm{iso} \: \mathrm{alg}} d\Omega_{\bm X} \displaybreak[2] \\
=& \frac{1}{\Delta t_n} \int_{\Omega_{\bm X}} \frac12 \left( \bm F^{T}_{n+1} \bm F_{n+1} - \bm F^T_n \bm F_n + \bm F^T_n \bm F_{n+1} - \bm F^T_{n+1} \bm F_n \right) : \bm S_{\mathrm{iso} \: \mathrm{alg}} d\Omega_{\bm X} \displaybreak[2] \\
=& \frac{1}{\Delta t_n} \int_{\Omega_{\bm X}} \bm Z_n : \bm S_{\mathrm{iso} \: \mathrm{alg}} d\Omega_{\bm X} \displaybreak[2] \\
=& \frac{1}{\Delta t_n} \int_{\Omega_{\bm X}} \left( G_{\mathrm{iso}}(\tilde{\bm C}_{n+1}, \bm \Gamma^1_{n+1}, \cdots, \bm \Gamma^{m}_{n+1}) - G_{\mathrm{iso}}(\tilde{\bm C}_{n}, \bm \Gamma^1_{n}, \cdots, \bm \Gamma^{m}_{n}) \right) d\Omega_{\bm X} + \frac12 \sum_{\alpha=1}^{m} \int_{\Omega_{\bm X}} \eta^{\alpha} \left\lvert \frac{ \bm \Gamma^{\alpha}_{n+1} - \bm \Gamma^{\alpha}_n}{\Delta t_n} \right\rvert^2 d\Omega_{\bm X}.
\end{align*}
The last equality in the above results from the relation \eqref{eq:directionality-property}.
Next, we choose $Q = P_{n+\frac12}$ in \eqref{eq:em-ts-mass}, which cancels out the third term in \eqref{eq:em-prop-2}. With the above derivations, the relation \eqref{eq:em-prop-2} can be reorganized and leads to the following,
\begin{align*}
\frac{1}{\Delta t_n}\left( H_{n+1} - H_n \right) = \int_{\Omega_{\bm X}} \rho_0 \bm V_{n+\frac12} \cdot \bm B d\Omega_{\bm X} + \int_{\Gamma_{\bm X}^H} \bm V_{n+\frac12} \cdot \bm H d\Gamma_{\bm X} &- \frac12 \sum_{\alpha=1}^{m} \int_{\Omega_{\bm X}} \eta^{\alpha} \left\lvert \frac{ \bm \Gamma^{\alpha}_{n+1} - \bm \Gamma^{\alpha}_n}{\Delta t_n} \right\rvert^2 d\Omega_{\bm X} \displaybreak[2] \\
&- \int_{\Omega_{\bm X}} \gamma J_{n+\frac12} \left(\nabla \bm V_{n+\frac12} : \bm F^{-T}_{n+\frac12} \right)^2 d\Omega_{\bm X},
\end{align*}
which completes the proof.
\end{proof}

\begin{remark}
The grad-div stabilization can be turned off by setting $\gamma = 0$. Consequently, the relation \eqref{eq:em-prop-1} indicates the dissipation of the solution precisely equals the physical dissipation $\mathcal D_{\mathrm{phy}}$.
\end{remark}

\begin{remark}
The condition \eqref{eq:directionality-property} is the discrete analogue of the continuum relation
\begin{align*}
\frac{dG_{\mathrm{iso}}}{dt} = \bm S_{\mathrm{iso}} : \frac{d\bm C}{dt} - \sum_{\alpha=1}^{m} \frac{\bm Q^{\alpha}}{2} : \frac{d\bm \Gamma^{\alpha}}{dt} = \bm S_{\mathrm{iso}} : \frac{d\bm C}{dt} - \frac12 \sum_{\alpha=1}^{m} \eta^{\alpha} \left\lvert  \frac{d \bm \Gamma^{\alpha}}{dt} \right\rvert^2.
\end{align*}
It is an extension of the directionality condition for elastodynamics \cite{Romero2012} by taking the physical dissipation into account.
\end{remark}

\begin{remark}
Following the discrete gradient formula \cite{Gonzalez2000}, one design of the stress enhancement is given as
\begin{align*}
\bm S_{\mathrm{iso} \: \mathrm{enh}} =& \Bigg( G_{\mathrm{iso}}(\tilde{\bm C}_{n+1}, \bm \Gamma^{1}_{n+1}, \cdots, \bm \Gamma^{m}_{n+1}) - G_{\mathrm{iso}}(\tilde{\bm C}_{n}, \bm \Gamma^{1}_{n}, \cdots, \bm \Gamma^{m}_{n}) + \frac{\Delta t_n}{2} \sum_{\alpha=1}^{m} \eta^{\alpha} \left\lvert \frac{ \bm \Gamma^{\alpha}_{n+1} - \bm \Gamma^{\alpha}_n}{\Delta t_n} \right\rvert^2 \displaybreak[2] \\
& - \bm S_{\mathrm{iso} \: n+\frac12}:\bm Z_n \Bigg) \frac{\bm Z_n}{\left\lvert\bm Z_n \right\rvert^2}.
\end{align*}
It is straightforward to verify that the above stress enhancement satisfies the directionality condition \eqref{eq:directionality-property}. Yet, for inelastic problems, the formula is incomplete without the discrete evolution equation for the internal state variables $\lbrace \bm \Gamma^{\alpha}_{n+1} \rbrace_{\alpha=1}^{m}$. The discrete evolution equation impacts the accuracy of the algorithmic stress and thus requires thoughtful studies.
\end{remark}

The conservation of the angular momentum demands the algorithmic stress to be symmetric. We summarize this design criterion in the following proposition.
\begin{lemma}
\label{lemma:ts-momentum-conservation}
Assuming $\Gamma^{G}_{\bm X} = \emptyset$ and the algorithmic stress $\bm S_{\mathrm{iso} \: \mathrm{alg}}$ is symmetric, the linear and angular momenta are conserved in the following sense,
\begin{align*}
& \frac{1}{\Delta t_n} \left( \bm L_{n+1} - \bm L_n \right) = \int_{\Omega_{\bm X}} \rho_0 \bm B d\Omega_{\bm X} + \int_{\Gamma_{\bm X}} \bm H d\Gamma_{\bm X}, \\
& \frac{1}{\Delta t_n} \left( \bm J_{n+1} - \bm J_n \right) = \int_{\Omega_{\bm X}} \rho_0 \bm \varphi_{n+\frac12} \times \bm B d\Omega_{\bm X} + \int_{\Gamma_{\bm X}} \bm \varphi_{n+\frac12} \times \bm H d\Gamma_{\bm X},
\end{align*}
in which 
\begin{align*}
\bm L_n := \int_{\Omega_{\bm X}} \rho_0 \bm V_{n} d\Omega_{\bm X} \quad \mbox{ and } \quad \bm J_n := \int_{\Omega_{\bm X}} \rho_0 \bm \varphi_{n} \times \bm V_n d\Omega_{\bm X}
\end{align*}
are the discrete approximations of $\bm L$ and $\bm J$ at time $t_n$.
\end{lemma}
\begin{proof}
First, we notice that $\Gamma^{G}_{\bm X} = \emptyset$ implies $\Gamma_{\bm X} = \Gamma^H_{\bm X}$. This implies $\bm \xi$ and $\bm \xi \times \bm \varphi_{n+\frac12}$ are both admissible test functions for a given constant vector $\bm \xi \in \mathbb R^3$. Choosing them as the test functions in \eqref{eq:em-ts-momentum} gives the discrete conservation of momenta.
\end{proof}

In addition to the energy dissipation and momentum conservation, the design of the algorithmic stress needs to maintain consistency. The time integration scheme \eqref{eq:em-ts-kinematic}-\eqref{eq:em-ts-momentum} recovers the implicit mid-point scheme if the algorithmic stress is replaced by the mid-point evaluation of the stress, i.e.,
\begin{align*}
\bm S_{\mathrm{iso} \: \mathrm{mp}} := \bm S_{\mathrm{iso}}\left( \tilde{\bm F}^{T}_{n+\frac12} \tilde{\bm F}_{n+\frac12}, \bm \Gamma^1_{n+\frac12}, \cdots, \bm \Gamma^{m}_{n+\frac12} \right).
\end{align*}
Consequently, the accuracy of the temporal scheme is contingent upon the form of the algorithmic stress. In particular, we introduce the local truncation error as $\tau(t_n) := \left \lvert \bm S_{\mathrm{iso} \: \mathrm{alg}}(t) - \bm S_{\mathrm{iso} \:\mathrm{mp}}(t)\right \rvert$, where the terms $\bm S_{\mathrm{iso} \: \mathrm{alg}}(t)$ and $\bm S_{\mathrm{iso} \: \mathrm{mp}}(t)$ are obtained through replacing the discrete quantities by the time continuous solutions at the corresponding time instances. If 
\begin{align*}
\tau(t_n) \leq \mathcal K \Delta t_n^{\mathsf k},
\end{align*}
$\mathcal K$ is independent of $\Delta t_n$, and $\mathsf k > 0$, the algorithmic stress is consistent with the order of accuracy being $\mathsf k$. Invoking the formula \eqref{eq:em-ts-S-alg-def}, we have
\begin{align}
\label{eq:genereal-error-analysis}
\tau(t_n) = \left \lvert \bm S_{\mathrm{iso} \: n+\frac12}(t) - \bm S_{\mathrm{iso} \: \mathrm{mp}}(t) + \bm S_{\mathrm{iso} \: \mathrm{enh}}(t) \right\rvert \leq \left \lvert \bm S_{\mathrm{iso} \: n+\frac12}(t) - \bm S_{\mathrm{iso} \: \mathrm{mp}}(t) \right\rvert + \left \lvert \bm S_{\mathrm{iso} \: \mathrm{enh}}(t) \right\rvert.
\end{align}
It is known that $\left \lvert \bm S_{\mathrm{iso} \: n+\frac12}(t) - \bm S_{\mathrm{iso} \: \mathrm{mp}}(t) \right\rvert = \mathcal O(\Delta t^2)$ \cite{Gonzalez2000}. Therefore, the stress enhancement is the determining factor in the temporal accuracy. As a design criterion, we demand it to be at least a first-order, and preferably second-order, with respect to $\Delta t_n$, asymptotically. Its analysis relies on the specific integration algorithm for the viscoelastic constitutive equations.

\subsection{A first-order update formula for the time-discrete internal state variables}
\label{sec:first-order-update-formula-ISV}
In this section, we give an integration formula for the constitutive relation that satisfies the design criteria of $\bm S_{\mathrm{iso} \: \mathrm{alg}}$ outlined above. It results in a sub-optimal scheme with only first-order temporal accuracy. Nevertheless, its analysis allows a simple and relatively straightforward presentation of the main idea in the algorithm design. We, therefore, choose to start with this scheme as an inspiring example. We introduce the algorithmic approximation of $\bm Q^{\alpha}$ as $\bm Q^{\alpha}_{\mathrm{alg1}}$, which is demanded to satisfy the following relation,
\begin{align}
\label{eq:def-Q-alg-1}
\bm Q^{\alpha}_{\mathrm{alg1}} : \left( \bm \Gamma^{\alpha}_{n+1} - \bm \Gamma^{\alpha}_{n} \right) = -2 \left( \Upsilon^{\alpha}( \tilde{\bm C}_{n}, \bm \Gamma^{\alpha}_{n+1} ) - \Upsilon^{\alpha}( \tilde{\bm C}_{n}, \bm \Gamma^{\alpha}_{n} ) \right).
\end{align}
The above design criterion can be viewed as a \textit{directionality} property for the stress-like conjugate variables $\bm Q^{\alpha}$ inspired by its definition \eqref{eq:def-Q}. Making use of the form of the configurational free energy \eqref{eq:def-flv-Upsilon}, the right-hand side of \eqref{eq:def-Q-alg-1} can be explicitly expressed as
\begin{align}
\label{eq:directionality-alg-1}
-2 \left( \Upsilon(\tilde{\bm C}_n, \bm \Gamma^{\alpha}_{n+1}) - \Upsilon(\tilde{\bm C}_n, \bm \Gamma^{\alpha}_{n}) \right) = \left( \tilde{\bm S}^{\alpha}_{\mathrm{iso}}(\tilde{\bm C_n}) - \hat{\bm S}^{\alpha}_{0} - \mu^{\alpha}\left( \frac{\bm \Gamma^{\alpha}_{n+1} + \bm \Gamma^{\alpha}_n}{2} - \bm I \right) \right) : \left( \bm \Gamma^{\alpha}_{n+1} - \bm \Gamma^{\alpha}_{n} \right).
\end{align}
Combining \eqref{eq:def-Q-alg-1} and \eqref{eq:directionality-alg-1}, the algorithmic conjugate variables $\bm Q^{\alpha}_{\mathrm{alg1}}$ can be designed as
\begin{align}
\label{eq:constitutive-integration-Q-1}
\bm Q^{\alpha}_{\mathrm{alg1}} = \tilde{\bm S}^{\alpha}_{\mathrm{iso}}(\tilde{\bm C_n}) - \hat{\bm S}^{\alpha}_{0} - \mu^{\alpha}\left( \bm \Gamma^{\alpha}_{n+\frac12} - \bm I \right),
\end{align}
which ensures the satisfaction of the directionality property \eqref{eq:def-Q-alg-1}. Considering the constitutive relations for $\bm Q^{\alpha}$ given by \eqref{eq:constitutive-flv-Q}, its discrete counterpart can be given as
\begin{align}
\label{eq:constitutive-integration-Gamma-Q-1}
\eta^{\alpha} \frac{\bm \Gamma^{\alpha}_{n+1} - \bm \Gamma^{\alpha}_n}{\Delta t_n} = \bm Q^{\alpha}_{\mathrm{alg1}}.
\end{align}
Combining \eqref{eq:constitutive-integration-Q-1} and \eqref{eq:constitutive-integration-Gamma-Q-1}, we have a set of evolution equations for the discrete internal state variables
\begin{align*}
\eta^{\alpha} \frac{\bm \Gamma^{\alpha}_{n+1} - \bm \Gamma^{\alpha}_n}{\Delta t_n} = \tilde{\bm S}^{\alpha}_{\mathrm{iso}}(\tilde{\bm C_n}) - \hat{\bm S}^{\alpha}_{0} - \mu^{\alpha}\left( \bm \Gamma^{\alpha}_{n+\frac12} - \bm I \right),
\end{align*}
and the variables $\bm \Gamma^{\alpha}_{n+1}$ are readily obtained from the evolution equations as
\begin{align}
\label{eq:constitutive-integration-Gamma-1}
\bm \Gamma^{\alpha}_{n+1} = \frac{\Delta t_n}{\eta^{\alpha} + \mu^{\alpha} \Delta t_n / 2} \left( \tilde{\bm S}^{\alpha}_{\mathrm{iso}}(\tilde{\bm C}_n) - \hat{\bm S}^{\alpha}_{0} + \left( \frac{\eta^{\alpha}}{\Delta t_n} - \frac{\mu^{\alpha}}{2} \right) \bm \Gamma^{\alpha}_n  + \mu^{\alpha} \bm I \right).
\end{align}
The algorithmic definition of $\bm S_{\mathrm{iso}}$ is thus given as
\begin{align}
\label{eq:constitutive-integration-Salg-1}
& \bm S_{\mathrm{iso} \: \mathrm{alg1}} := \bm S_{\mathrm{iso} \: n+\frac12} + \bm S_{\mathrm{iso} \: \mathrm{enh1}}, \displaybreak[2] \displaybreak[2] \\
\label{eq:constitutive-integration-Salg-1-detailed}
& \bm S_{\mathrm{iso} \: \mathrm{enh1}} := \frac{ G_{\mathrm{iso}}(\tilde{\bm C}_{n+1}, \bm \Gamma^1_{n+1}, \cdots, \bm \Gamma^{m}_{n+1}) - G_{\mathrm{iso}}(\tilde{\bm C}_{n}, \bm \Gamma^1_{n+1}, \cdots, \bm \Gamma^{m}_{n+1}) - \bm S_{\mathrm{iso} \: n+\frac12} : \bm Z_{n} }{\left\lvert\bm Z_{n}\right\rvert} \frac{\bm Z_{n}}{\left\lvert\bm Z_{n}\right\rvert}.
\end{align}
The algorithmic stresses can be evaluated as follows. One first obtains the discrete internal state variables $\bm \Gamma^{\alpha}_{n+1}$ from $\tilde{\bm C}_{n}$ and $\bm \Gamma^{\alpha}_n$ according to \eqref{eq:constitutive-integration-Gamma-1}. Next, the algorithmic stress $\bm S_{\mathrm{iso} \: \mathrm{alg1}}$ is determined according to \eqref{eq:constitutive-integration-Salg-1}. We notice that the algorithmic conjugate variable $\bm Q^{\alpha}_{\mathrm{alg1}}$ given by \eqref{eq:constitutive-integration-Gamma-Q-1} is, in fact, unneeded in evaluating the constitutive relations. The algorithmic stress $\bm S_{\mathrm{iso} \: \mathrm{alg1}}$ is symmetric by construction and is thus momentum conserving according to Lemma \ref{lemma:ts-momentum-conservation}. Here we show that it satisfies the directionality property \eqref{eq:directionality-property} and consequently guarantees the energy stability.
 
\begin{proposition}
Assuming the boundary data $\bm G$ is time-independent, the algorithmic stress $\bm S_{\mathrm{iso} \: \mathrm{alg1}}$ given by \eqref{eq:constitutive-integration-Salg-1} together with the integration rule \eqref{eq:constitutive-integration-Gamma-1} ensures that the discrete scheme \eqref{eq:em-ts-kinematic}-\eqref{eq:em-ts-momentum} is energy stable in the sense of \eqref{eq:em-prop-1}.
\end{proposition}
\begin{proof}
According to Lemma \ref{lemma:ts-energy-dissipation}, it is sufficient to show the energy stability by verifying that the relation \eqref{eq:directionality-property} holds, i.e.,
\begin{align*}
\bm Z_n : \bm S_{\mathrm{iso} \: \mathrm{alg1}} =& G_{\mathrm{iso}}(\tilde{\bm C}_{n+1}, \bm \Gamma^{1}_{n+1}, \cdots, \bm \Gamma^{m}_{n+1}) - G_{\mathrm{iso}}(\tilde{\bm C}_{n}, \bm \Gamma^{1}_{n}, \cdots, \bm \Gamma^{m}_{n}) + \frac{\Delta t_n}{2} \sum_{\alpha=1}^{m} \eta^{\alpha} \left\lvert \frac{ \bm \Gamma^{\alpha}_{n+1} - \bm \Gamma^{\alpha}_n}{\Delta t_n} \right\rvert^2.
\end{align*}
The contraction of $\bm Z_n$ with $\bm S_{\mathrm{iso} \: \mathrm{alg1}}$ directly leads to
\begin{align*}
\bm Z_n : \bm S_{\mathrm{iso} \: \mathrm{alg1}} =& G_{\mathrm{iso}}(\tilde{\bm C}_{n+1}, \bm \Gamma^1_{n+1}, \cdots, \bm \Gamma^{m}_{n+1}) - G_{\mathrm{iso}}(\tilde{\bm C}_{n}, \bm \Gamma^1_{n+1}, \cdots, \bm \Gamma^{m}_{n+1}) \displaybreak[2] \\
=& G_{\mathrm{iso}}(\tilde{\bm C}_{n+1}, \bm \Gamma^1_{n+1}, \cdots, \bm \Gamma^{m}_{n+1}) - G_{\mathrm{iso}}(\tilde{\bm C}_{n}, \bm \Gamma^1_{n}, \cdots, \bm \Gamma^{m}_{n}) \displaybreak[2] \\
& + G_{\mathrm{iso}}(\tilde{\bm C}_{n}, \bm \Gamma^1_{n}, \cdots, \bm \Gamma^{m}_{n}) - G_{\mathrm{iso}}(\tilde{\bm C}_{n}, \bm \Gamma^1_{n+1}, \cdots, \bm \Gamma^{m}_{n+1}).
\end{align*}
Invoking the definition of $G_{\mathrm{iso}}$, straightforward calculations reveal that the last two terms can be expressed as
\begin{align*}
G_{\mathrm{iso}}(\tilde{\bm C}_{n}, \bm \Gamma^1_{n}, \cdots, \bm \Gamma^{m}_{n}) - G_{\mathrm{iso}}(\tilde{\bm C}_{n}, \bm \Gamma^1_{n+1}, \cdots, \bm \Gamma^{m}_{n+1}) =& - \sum_{\alpha=1}^{m} \left( \Upsilon^{\alpha}\left(\tilde{\bm C}_n, \bm \Gamma^{\alpha}_{n+1} \right) - \Upsilon^{\alpha}\left(\tilde{\bm C}_n, \bm \Gamma^{\alpha}_{n} \right) \right) \displaybreak[2] \\
=& \frac12 \sum_{\alpha=1}^{m} \Big( \bm Q^{\alpha}_{\mathrm{alg1}} : \left( \bm \Gamma^{\alpha}_{n+1} - \bm \Gamma^{\alpha}_{n} \right) \Big) \displaybreak[2] \\
=& \frac{\Delta t_n}{2} \sum_{\alpha=1}^{m} \eta^{\alpha} \left\lvert \frac{ \bm \Gamma^{\alpha}_{n+1} - \bm \Gamma^{\alpha}_n}{\Delta t_n} \right\rvert^2.
\end{align*}
The second equality in the above is due to the definition of $\bm Q^{\alpha}_{\mathrm{alg1}}$ \eqref{eq:def-Q-alg-1}; the last equality is due to the discrete relation \eqref{eq:constitutive-integration-Gamma-Q-1}. With the above, we conclude that the algorithmic stress $\bm S_{\mathrm{iso} \: \mathrm{alg1}}$ satisfies the relation \eqref{eq:directionality-property}, and the overall scheme is thus energy stable.
\end{proof}

\begin{proposition}
Assuming the time-continuous solutions are sufficiently smooth, the stress enhancement term $\bm S_{\mathrm{iso} \: \mathrm{enh1}}$ is asymptotically a first-order term with respect to $\Delta t_n$.
\end{proposition}
\begin{proof}
Taking the norm on $\bm S_{\mathrm{iso} \: \mathrm{enh1}}(t)$ gives
\begin{align*}
&\left\lvert \bm S_{\mathrm{iso} \: \mathrm{enh1}}(t) \right\rvert = \left\lvert G_{\mathrm{iso}}(\tilde{\bm C}(t_{n+1}), \bm \Gamma^1(t_{n+1}), \cdots, \bm \Gamma^{m}(t_{n+1})) - G_{\mathrm{iso}}(\tilde{\bm C}(t_{n}), \bm \Gamma^1(t_{n+1}), \cdots, \bm \Gamma^{m}(t_{n+1})) \right. \displaybreak[2] \\
& \left. - \bm S_{\mathrm{iso}}\left( \bm C(t_{n+\frac12}), \bm \Gamma^{1}(t_{n+\frac12}), \cdots, \bm \Gamma^{m}(t_{n+\frac12}) \right) : \frac12 \left( \bm C(t_{n+1}) - \bm C(t_n) \right) \right\rvert / \left\lvert \frac12 \left( \bm C(t_{n+1}) - \bm C(t_n) \right) \right\rvert .
\end{align*}
Performing Taylor expansion about $t_{n+\frac12}$ gives
\begin{align*}
& \bm C(t_{n+1}) - \bm C(t_n) = \dot{\bm C}(t_{n+\frac12})\Delta t_n + \mathcal O(\Delta t_n^3) \dddot{\bm C}(t_{n+\frac12}), \\
& G_{\mathrm{iso}}(\tilde{\bm C}(t_{n+1}), \bm \Gamma^1(t_{n+1}), \cdots, \bm \Gamma^{m}(t_{n+1})) - G_{\mathrm{iso}}(\tilde{\bm C}(t_{n}), \bm \Gamma^1(t_{n+1}), \cdots, \bm \Gamma^{m}(t_{n+1})) \\
= & \frac12 \bm S_{\mathrm{iso}}\left( \bm C(t_{n+\frac12}), \bm \Gamma^{1}(t_{n+1}), \cdots, \bm \Gamma^{m}(t_{n+1}) \right) : \dot{\bm C}(t_{n+\frac12}) \Delta t_n + \mathcal O(\Delta t_n^3) \\
= & \frac12 \bm S_{\mathrm{iso}}\left( \bm C(t_{n+\frac12}), \bm \Gamma^{1}(t_{n+\frac12}), \cdots, \bm \Gamma^{m}(t_{n+\frac12}) \right) : \dot{\bm C}(t_{n+\frac12}) \Delta t_n + \mathcal O(\Delta t_n^2)
\end{align*}
In the above, we use $\dot{\left( \cdot \right)}$ to indicate the time derivative of the quantity $\left( \cdot \right)$. Consequently, it follows that
\begin{align*}
\left\lvert \bm S_{\mathrm{iso} \: \mathrm{enh1}}(t) \right\rvert = \mathcal O(\Delta t^2_n) /  \left\lvert \frac12 \left( \bm C(t_{n+1}) - \bm C(t_n) \right) \right\rvert = \mathcal O(\Delta t_n).
\end{align*}
\end{proof}
Based on the analysis made in \eqref{eq:genereal-error-analysis}, we may conclude that the stress enhancement $\bm S_{\mathrm{iso} \: \mathrm{enh1}}$ induces a first-order perturbation to the mid-point scheme, rendering the scheme to be first-order accurate in time.

\subsection{A second-order update formula for the time-discrete internal state variables}
\label{sec:second-order-update-formula-ISV}
Here, we derive an alternative update formula for the inelastic constitutive relations, which achieves second-order accuracy. At the time step $t_{n+1}$, we demand that the algorithmic approximations to $\bm Q^{\alpha}$, denoted by $\bm Q^{\alpha}_{\mathrm{alg2}}$, satisfy the following,
\begin{align}
\label{eq:def-alg-Q2}
\bm Q^{\alpha}_{\mathrm{alg2}} : \left( \bm \Gamma^{\alpha}_{n+1} - \bm \Gamma^{\alpha}_{n} \right) =& -2 \mu^{\alpha} \left( \left\lvert \frac{\bm \Gamma^{\alpha}_{n+1} - \bm I}{2} \right\rvert^2 - \left\lvert \frac{\bm \Gamma^{\alpha}_{n} - \bm I}{2} \right\rvert^2  \right) \nonumber \\
& \quad + \left( \frac{\tilde{\bm S}_{\mathrm{iso}}^{\alpha}(\tilde{\bm C}_{n+1}) + \tilde{\bm S}_{\mathrm{iso}}^{\alpha}(\tilde{\bm C}_{n})}{2} - \hat{\bm S}^{\alpha}_{0} \right) : \left( \bm \Gamma^{\alpha}_{n+1} - \bm \Gamma^{\alpha}_{n} \right),
\end{align}
for $1\leq \alpha \leq m$. It can be shown that
\begin{align*}
\left\lvert \frac{\bm \Gamma^{\alpha}_{n+1} - \bm I}{2} \right\rvert^2 - \left\lvert \frac{\bm \Gamma^{\alpha}_{n} - \bm I}{2} \right\rvert^2 = \left( \bm \Gamma^{\alpha}_{n+\frac12} - \bm I \right) : \frac{\bm \Gamma^{\alpha}_{n+1} - \bm \Gamma^{\alpha}_n}{2},
\end{align*}
and \eqref{eq:def-alg-Q2} can be reorganized as
\begin{align*}
\bm Q^{\alpha}_{\mathrm{alg2}} : \left( \bm \Gamma^{\alpha}_{n+1} - \bm \Gamma^{\alpha}_{n} \right) = \left( \frac{\tilde{\bm S}_{\mathrm{iso}}^{\alpha}(\tilde{\bm C}_{n+1}) + \tilde{\bm S}_{\mathrm{iso}}^{\alpha}(\tilde{\bm C}_{n})}{2} - \hat{\bm S}^{\alpha}_{0} - \mu^{\alpha}\left( \bm \Gamma^{\alpha}_{n+\frac12} - \bm I \right) \right) : \left( \bm \Gamma^{\alpha}_{n+1} - \bm \Gamma^{\alpha}_{n} \right).
\end{align*}
Therefore, the explicit form of $\bm Q^{\alpha}_{\mathrm{alg2}}$ is given by
\begin{align}
\label{eq:explicit-def-alg-Q2}
\bm Q^{\alpha}_{\mathrm{alg2}} = \frac{\tilde{\bm S}_{\mathrm{iso}}^{\alpha}(\tilde{\bm C}_{n+1}) + \tilde{\bm S}_{\mathrm{iso}}^{\alpha}(\tilde{\bm C}_{n})}{2} - \hat{\bm S}^{\alpha}_{0} - \mu^{\alpha}\left( \bm \Gamma^{\alpha}_{n+\frac12} - \bm I \right),
\end{align}
which guarantees the satisfaction of \eqref{eq:def-alg-Q2}. Compared with the definition of $\bm Q^{\alpha}_{\mathrm{alg1}}$ given in \eqref{eq:constitutive-integration-Q-1}, the difference resides in the evaluation of the term $\tilde{\bm S}^{\alpha}_{\mathrm{iso}}$. To reveal its connection with the free energy, we state and prove the following lemma.
\begin{lemma}
\label{lemma-Upsilon-split-relation}
For the configurational free energy defined in \eqref{eq:def-flv-Upsilon} and $\bm Q^{\alpha}_{\mathrm{alg2}}$ given in \eqref{eq:explicit-def-alg-Q2}, the following holds for $1 \leq \alpha \leq m$,
\begin{align}
\label{eq:lemma-Upsilon-split-relation}
\Upsilon^{\alpha}(\tilde{\bm C}_{n+1}, \bm \Gamma^{\alpha}_{n+1}) - \Upsilon^{\alpha}(\tilde{\bm C}_{n+1}, \bm \Gamma^{\alpha}_{n+\frac12}) + \Upsilon^{\alpha}(\tilde{\bm C}_{n}, \bm \Gamma^{\alpha}_{n+\frac12}) - \Upsilon^{\alpha}(\tilde{\bm C}_{n}, \bm \Gamma^{\alpha}_{n}) = -\frac12 \bm Q^{\alpha}_{\mathrm{alg2}} : \left( \bm \Gamma^{\alpha}_{n+1} - \bm \Gamma^{\alpha}_{n}\right). 
\end{align}
\end{lemma}
\begin{proof}
One may calculate $\Upsilon^{\alpha}(\tilde{\bm C}_{n+1}, \bm \Gamma^{\alpha}_{n+1}) - \Upsilon^{\alpha}(\tilde{\bm C}_{n+1}, \bm \Gamma^{\alpha}_{n+\frac12})$ as follows,
\begin{align*}
& \Upsilon^{\alpha}(\tilde{\bm C}_{n+1}, \bm \Gamma^{\alpha}_{n+1}) - \Upsilon^{\alpha}(\tilde{\bm C}_{n+1}, \bm \Gamma^{\alpha}_{n+\frac12}) \displaybreak[2] \nonumber \\
=& \mu^{\alpha} \left\lvert \frac{\bm \Gamma^{\alpha}_{n+1} - \bm I}{2} \right\rvert^2 + \left( \hat{\bm S}^{\alpha}_{0} - \tilde{\bm S}^{\alpha}_{\mathrm{iso}}(\tilde{\bm C}_{n+1}) \right) : \frac{\bm \Gamma^{\alpha}_{n+1} - \bm I}{2} + \frac{1}{4\mu^{\alpha}} \left\lvert \tilde{\bm S}^{\alpha}_{\mathrm{iso}}(\tilde{\bm C}_{n+1}) - \hat{\bm S}^{\alpha}_0 \right\lvert^2 \displaybreak[2] \nonumber \\
& - \mu^{\alpha} \left\lvert \frac{\bm \Gamma^{\alpha}_{n+\frac12} - \bm I}{2} \right\rvert^2 - \left( \hat{\bm S}^{\alpha}_{0} - \tilde{\bm S}^{\alpha}_{\mathrm{iso}}(\tilde{\bm C}_{n+1}) \right) : \frac{\bm \Gamma^{\alpha}_{n+\frac12} - \bm I}{2} - \frac{1}{4\mu^{\alpha}} \left\lvert \tilde{\bm S}^{\alpha}_{\mathrm{iso}}(\tilde{\bm C}_{n+1}) - \hat{\bm S}^{\alpha}_0 \right\lvert^2 \displaybreak[2] \nonumber \\
=& \mu^{\alpha} \left( \left\lvert \frac{\bm \Gamma^{\alpha}_{n+1} - \bm I}{2} \right\rvert^2 - \left\lvert \frac{\bm \Gamma^{\alpha}_{n+\frac12} - \bm I}{2} \right\rvert^2\right) + \left( \hat{\bm S}^{\alpha}_{0} - \tilde{\bm S}^{\alpha}_{\mathrm{iso}}(\tilde{\bm C}_{n+1}) \right) : \frac{\bm \Gamma^{\alpha}_{n+1} - \bm \Gamma^{\alpha}_{n}}{4}.
\end{align*}
Similarily, it can be shown that
\begin{align*}
& \Upsilon^{\alpha}(\tilde{\bm C}_{n}, \bm \Gamma^{\alpha}_{n+\frac12}) - \Upsilon^{\alpha}(\tilde{\bm C}_{n}, \bm \Gamma^{\alpha}_{n}) = \mu^{\alpha} \left( \left\lvert \frac{\bm \Gamma^{\alpha}_{n+\frac12} - \bm I}{2} \right\rvert^2 - \left\lvert \frac{\bm \Gamma^{\alpha}_{n} - \bm I}{2} \right\rvert^2\right) + \left( \hat{\bm S}^{\alpha}_{0} - \tilde{\bm S}^{\alpha}_{\mathrm{iso}}(\tilde{\bm C}_{n}) \right) : \frac{\bm \Gamma^{\alpha}_{n+1} - \bm \Gamma^{\alpha}_{n}}{4}.
\end{align*}
Combining the above two yields
\begin{align*}
& \Upsilon^{\alpha}(\tilde{\bm C}_{n+1}, \bm \Gamma^{\alpha}_{n+1}) - \Upsilon^{\alpha}(\tilde{\bm C}_{n+1}, \bm \Gamma^{\alpha}_{n+\frac12}) + \Upsilon^{\alpha}(\tilde{\bm C}_{n}, \bm \Gamma^{\alpha}_{n+\frac12}) - \Upsilon^{\alpha}(\tilde{\bm C}_{n}, \bm \Gamma^{\alpha}_{n})  \nonumber \\
=& \mu^{\alpha} \left( \left\lvert \frac{\bm \Gamma^{\alpha}_{n+1} - \bm I}{2} \right\rvert^2 - \left\lvert \frac{\bm \Gamma^{\alpha}_{n} - \bm I}{2} \right\rvert^2  \right) - \left( \frac{\tilde{\bm S}_{\mathrm{iso}}^{\alpha}(\tilde{\bm C}_{n+1}) + \tilde{\bm S}_{\mathrm{iso}}^{\alpha}(\tilde{\bm C}_{n})}{2} - \hat{\bm S}^{\alpha}_{0} \right) : \left( \frac{\bm \Gamma^{\alpha}_{n+1} - \bm \Gamma^{\alpha}_{n}}{2} \right) \nonumber \\
=& -\frac12 \bm Q^{\alpha}_{\mathrm{alg2}} : \left( \bm \Gamma^{\alpha}_{n+1} - \bm \Gamma^{\alpha}_{n}\right).
\end{align*}
The last equality of the above derivation is due to the definition of $\bm Q^{\alpha}_{\mathrm{alg2}}$ \eqref{eq:def-alg-Q2}.
\end{proof}
The relation \eqref{eq:lemma-Upsilon-split-relation} is in fact a \textit{directionality} property for the stress-like variables $\bm Q^{\alpha}$, mimicking their definitions \eqref{eq:def-Q}. Next, we discretize the constitutive relations \eqref{eq:constitutive-flv-Q} as
\begin{align}
\label{eq:isv-Q-alg-discretization}
\bm Q^{\alpha}_{\mathrm{alg2}} =  \eta^{\alpha} \frac{\bm \Gamma^{\alpha}_{n+1} - \bm \Gamma^{\alpha}_n}{\Delta t_n}.
\end{align}
Relating the relations \eqref{eq:explicit-def-alg-Q2} with \eqref{eq:isv-Q-alg-discretization} leads to a set of discrete equations for the internal state variables,
\begin{align}
\label{eq:isv-update-formula}
\eta^{\alpha} \frac{\bm \Gamma^{\alpha}_{n+1} - \bm \Gamma^{\alpha}_n}{\Delta t_n} = \frac{\tilde{\bm S}_{\mathrm{iso}}^{\alpha}(\tilde{\bm C}_{n+1}) + \tilde{\bm S}_{\mathrm{iso}}^{\alpha}(\tilde{\bm C}_{n})}{2} - \hat{\bm S}^{\alpha}_{0} - \mu^{\alpha}\left( \bm \Gamma^{\alpha}_{n+\frac12} - \bm I \right).
\end{align}
From the above equations, we obtain the update formula for the internal state variables as 
\begin{align}
\label{eq:isv-update-formula-2}
\bm \Gamma^{\alpha}_{n+1} = \frac{\Delta t_n}{\eta^{\alpha} + \mu^{\alpha} \Delta t_n / 2} \left( \frac{\tilde{\bm S}_{\mathrm{iso}}^{\alpha}(\tilde{\bm C}_{n+1}) + \tilde{\bm S}_{\mathrm{iso}}^{\alpha}(\tilde{\bm C}_{n})}{2} - \hat{\bm S}^{\alpha}_{0} - \left( \frac{\mu^{\alpha}}{2} - \frac{\eta^{\alpha}}{\Delta t_n} \right) \bm \Gamma^{\alpha}_n  + \mu^{\alpha} \bm I \right).
\end{align}
We next state and prove a lemma that reveals the significance of the relation \eqref{eq:isv-update-formula-2} in terms of energy dissipation.
\begin{lemma}
\label{lemma:update-formula-isv-relation-summed}
For the configurational free energy defined in \eqref{eq:def-flv-Upsilon}, the discrete evolution equations \eqref{eq:isv-update-formula-2} guarantee the following,
\begin{align*}
& G_{\mathrm{iso}}\left( \tilde{\bm C}_{n+1}, \bm \Gamma^1_{n+1}, \cdots, \bm \Gamma^m_{n+1} \right) - G_{\mathrm{iso}}\left( \tilde{\bm C}_{n+1}, \bm \Gamma^1_{n+\frac12}, \cdots, \bm \Gamma^m_{n+\frac12} \right) +  G_{\mathrm{iso}}\left( \tilde{\bm C}_{n}, \bm \Gamma^1_{n+\frac12}, \cdots, \bm \Gamma^m_{n+\frac12} \right) \nonumber \displaybreak[2] \\
& \quad - G_{\mathrm{iso}}\left( \tilde{\bm C}_{n}, \bm \Gamma^1_{n}, \cdots, \bm \Gamma^m_{n} \right) = -\frac12 \Delta t_n \sum_{\alpha=1}^{m}\eta^{\alpha} \left\lvert \frac{\bm \Gamma^{\alpha}_{n+1} - \bm \Gamma^{\alpha}_{n}}{\Delta t_n} \right\rvert^2.
\end{align*}
\end{lemma}
\begin{proof}
Summing the relations \eqref{eq:lemma-Upsilon-split-relation} for all relaxation processes results in
\begin{align*}
& G_{\mathrm{iso}}\left( \tilde{\bm C}_{n+1}, \bm \Gamma^1_{n+1}, \cdots, \bm \Gamma^m_{n+1} \right) - G_{\mathrm{iso}}\left( \tilde{\bm C}_{n+1}, \bm \Gamma^1_{n+\frac12}, \cdots, \bm \Gamma^m_{n+\frac12} \right) +  G_{\mathrm{iso}}\left( \tilde{\bm C}_{n}, \bm \Gamma^1_{n+\frac12}, \cdots, \bm \Gamma^m_{n+\frac12} \right) \nonumber \displaybreak[2] \\
& - G_{\mathrm{iso}}\left( \tilde{\bm C}_{n}, \bm \Gamma^1_{n}, \cdots, \bm \Gamma^m_{n} \right) \nonumber \displaybreak[2] \\
= & -\frac12 \sum_{\alpha=1}^{m} \bm Q^{\alpha}_{\mathrm{alg2}} : \left( \bm \Gamma^{\alpha}_{n+1} - \bm \Gamma^{\alpha}_{n} \right) \displaybreak[2] \\
=& -\frac12 \sum_{\alpha=1}^{m} \eta^{\alpha} \frac{\left( \bm \Gamma^{\alpha}_{n+1} - \bm \Gamma^{\alpha}_n \right)}{\Delta t_n} : \left( \bm \Gamma^{\alpha}_{n+1} - \bm \Gamma^{\alpha}_{n} \right) \displaybreak[2] \\
=& -\frac12 \Delta t_n \sum_{\alpha=1}^{m}\eta^{\alpha} \left\lvert \frac{\bm \Gamma^{\alpha}_{n+1} - \bm \Gamma^{\alpha}_{n}}{\Delta t_n} \right\rvert^2.
\end{align*}
In the second equality of the above derivation, the relations \eqref{eq:isv-Q-alg-discretization} are utilized.
\end{proof}
With the integration of the constitutive relations stated above, the algorithmic stress can be given as follows,
\begin{align}
\label{eq:def-S-alg2}
\bm S_{\mathrm{iso} \: \mathrm{alg2}} :=& \bm S_{\mathrm{iso} \: n+\frac12} + \bm S_{\mathrm{iso} \: \mathrm{enh2}}, \displaybreak[2] \\
\label{eq:def-S-enh2}
\bm S_{\mathrm{iso} \: \mathrm{enh2}} :=& \frac{ G_{\mathrm{iso}}(\tilde{\bm C}_{n+1}, \bm \Gamma^1_{n+\frac12}, \cdots, \bm \Gamma^{m}_{n+\frac12}) - G_{\mathrm{iso}}(\tilde{\bm C}_{n}, \bm \Gamma^1_{n+\frac12}, \cdots, \bm \Gamma^{m}_{n+\frac12}) - \bm S_{\mathrm{iso} \: n+\frac12} : \bm Z_{n} }{\left\lvert\bm Z_{n}\right\rvert} \frac{\bm Z_{n}}{\left\lvert \bm Z_{n} \right\rvert },
\end{align}
where the definition of $\bm S_{\mathrm{iso} \: n+\frac12}$ has been given in \eqref{eq:em-ts-S-alg-def}. With the algorithmic stress $\bm S_{\mathrm{iso} \: \mathrm{alg2}}$, we have a new time-stepping scheme \eqref{eq:em-ts-kinematic}-\eqref{eq:em-ts-momentum}. In the following, we show that it is energy-momentum consistent and second-order accurate. The symmetry of the algorithmic stress is ensured by construction, and the momentum conservation is guaranteed due to Lemma \ref{lemma:ts-momentum-conservation}. The energy stability and temporal accuracy of this scheme are analyzed by the following propositions. 
\begin{proposition}
Assuming the boundary data $\bm G$ is time-independent, the algorithmic stress $\bm S_{\mathrm{iso} \: \mathrm{alg2}}$ given by \eqref{eq:def-S-alg2} together with the integration rule \eqref{eq:isv-update-formula-2} ensures that the discrete scheme \eqref{eq:em-ts-kinematic}-\eqref{eq:em-ts-momentum} is energy stable in the sense of \eqref{eq:em-prop-1}.
\end{proposition}
\begin{proof}
According to Lemma \ref{lemma:ts-energy-dissipation}, it is sufficient to show the energy stability by verifying that the relation \eqref{eq:directionality-property} holds, i.e.,
\begin{align*}
\bm Z_n : \bm S_{\mathrm{iso} \: \mathrm{alg2}} =& G_{\mathrm{iso}}(\tilde{\bm C}_{n+1}, \bm \Gamma^{1}_{n+1}, \cdots, \bm \Gamma^{m}_{n+1}) - G_{\mathrm{iso}}(\tilde{\bm C}_{n}, \bm \Gamma^{1}_{n}, \cdots, \bm \Gamma^{m}_{n}) + \frac{\Delta t_n}{2} \sum_{\alpha=1}^{m} \eta^{\alpha} \left\lvert \frac{ \bm \Gamma^{\alpha}_{n+1} - \bm \Gamma^{\alpha}_n}{\Delta t_n} \right\rvert^2.
\end{align*}
Here we directly calculate of the contraction between $\bm Z_n$ and $\bm S_{\mathrm{iso} \: \mathrm{alg2}}$ and get the following,
\begin{align*}
& \bm Z_n : \bm S_{\mathrm{iso} \: \mathrm{alg2}} = G_{\mathrm{iso}}(\tilde{\bm C}_{n+1}, \bm \Gamma^1_{n+\frac12}, \cdots, \bm \Gamma^{m}_{n+\frac12}) - G_{\mathrm{iso}}(\tilde{\bm C}_{n}, \bm \Gamma^1_{n+\frac12}, \cdots, \bm \Gamma^{m}_{n+\frac12}) \displaybreak[2] \\
=& G_{\mathrm{iso}}(\tilde{\bm C}_{n+1}, \bm \Gamma^1_{n+1}, \cdots, \bm \Gamma^{m}_{n+1}) - G_{\mathrm{iso}}(\tilde{\bm C}_{n}, \bm \Gamma^1_{n}, \cdots, \bm \Gamma^{m}_{n}) + G_{\mathrm{iso}}(\tilde{\bm C}_{n+1}, \bm \Gamma^1_{n+\frac12}, \cdots, \bm \Gamma^{m}_{n+\frac12}) \displaybreak[2] \\
& - G_{\mathrm{iso}}(\tilde{\bm C}_{n}, \bm \Gamma^1_{n+\frac12}, \cdots, \bm \Gamma^{m}_{n+\frac12}) - G_{\mathrm{iso}}(\tilde{\bm C}_{n+1}, \bm \Gamma^1_{n+1}, \cdots, \bm \Gamma^{m}_{n+1}) + G_{\mathrm{iso}}(\tilde{\bm C}_{n}, \bm \Gamma^1_{n}, \cdots, \bm \Gamma^{m}_{n}) \displaybreak[2] \\
=& G_{\mathrm{iso}}(\tilde{\bm C}_{n+1}, \bm \Gamma^1_{n+1}, \cdots, \bm \Gamma^{m}_{n+1}) - G_{\mathrm{iso}}(\tilde{\bm C}_{n}, \bm \Gamma^1_{n}, \cdots, \bm \Gamma^{m}_{n}) + \frac{\Delta t_n}{2}\sum_{\alpha=1}^{m}\eta^{\alpha} \left\lvert \frac{\bm \Gamma^{\alpha}_{n+1} - \bm \Gamma^{\alpha}_{n}}{\Delta t_n} \right\rvert^2.
\end{align*}
The last equality of the above is due to Lemma \ref{lemma:update-formula-isv-relation-summed}.
\end{proof}

\begin{proposition}
Assuming the time-continuous solutions are sufficiently smooth, the stress enhancement term $\bm S_{\mathrm{iso} \: \mathrm{enh2}}$ is asymptotically a second-order term with respect to $\Delta t_n$.
\end{proposition}
\begin{proof}
Taking the norm on $\bm S_{\mathrm{iso} \: \mathrm{enh2}}(t)$ gives
\begin{align*}
&\left\lvert \bm S_{\mathrm{iso} \: \mathrm{enh2}}(t) \right\rvert = \left\lvert G_{\mathrm{iso}}(\tilde{\bm C}(t_{n+1}), \bm \Gamma^1(t_{n+\frac12}), \cdots, \bm \Gamma^{m}(t_{n+\frac12})) - G_{\mathrm{iso}}(\tilde{\bm C}(t_{n}), \bm \Gamma^1(t_{n+\frac12}), \cdots, \bm \Gamma^{m}(t_{n+\frac12})) \right. \displaybreak[2] \\
& \left. - \bm S_{\mathrm{iso}}\left( \bm C(t_{n+\frac12}), \bm \Gamma^{1}(t_{n+\frac12}), \cdots, \bm \Gamma^{m}(t_{n+\frac12}) \right) : \frac12 \left( \bm C(t_{n+1}) - \bm C(t_n) \right) \right\rvert / \left\lvert \frac12 \left( \bm C(t_{n+1}) - \bm C(t_n) \right) \right\rvert .
\end{align*}
Performing Taylor expansion about $t_{n+\frac12}$ gives
\begin{align*}
& \bm C(t_{n+1}) - \bm C(t_n) = \dot{\bm C}(t_{n+\frac12})\Delta t_n + \mathcal O(\Delta t_n^3) \dddot{\bm C}(t_{n+\frac12}), \\
& G_{\mathrm{iso}}(\tilde{\bm C}(t_{n+1}), \bm \Gamma^1(t_{n+\frac12}), \cdots, \bm \Gamma^{m}(t_{n+\frac12})) - G_{\mathrm{iso}}(\tilde{\bm C}(t_{n}), \bm \Gamma^1(t_{n+\frac12}), \cdots, \bm \Gamma^{m}(t_{n+\frac12})) \\
= & \frac12 \bm S_{\mathrm{iso}}\left( \bm C(t_{n+\frac12}), \bm \Gamma^{1}(t_{n+\frac12}), \cdots, \bm \Gamma^{m}(t_{n+\frac12}) \right) : \dot{\bm C}(t_{n+\frac12}) \Delta t_n + \mathcal O(\Delta t_n^3)
\end{align*}
It follows that
\begin{align*}
\left\lvert \bm S_{\mathrm{iso} \: \mathrm{enh2}}(t) \right\rvert = \mathcal O(\Delta t^3_n) /  \left\lvert \frac12 \left( \bm C(t_{n+1}) - \bm C(t_n) \right) \right\rvert = \mathcal O(\Delta t^2_n).
\end{align*}
\end{proof}
Based on the analysis made in \eqref{eq:genereal-error-analysis}, we may conclude that the stress enhancement $\bm S_{\mathrm{iso} \: \mathrm{enh2}}$ yields a second-order perturbation to the mid-point scheme. It thus maintains the second-order temporal accuracy.

\subsection{A remark on two alternative approaches}
\label{subsec:remark-on-two-alternative-approaches}
We mention that there exists an alternative approach proposed by Mart\'{i}n, et al. \cite{Martin2014}. In their work, the discrete gradient formula is exploited exclusively. The algorithmic design for $\bm Q^{\alpha}$ takes the following form,
\begin{align*}
& \bm Q^{\alpha}_{\mathrm{alg3}} := \bm Q^{\alpha}_{\star} - \left( \frac12 \left(  \Upsilon^{\alpha}(\tilde{\bm C}_{n+1}, \bm \Gamma^{\alpha}_{n+1}) - \Upsilon^{\alpha}(\tilde{\bm C}_{n+1}, \bm \Gamma^{\alpha}_{n}) + \Upsilon^{\alpha}(\tilde{\bm C}_{n}, \bm \Gamma^{\alpha}_{n+1}) - \Upsilon^{\alpha}(\tilde{\bm C}_{n}, \bm \Gamma^{\alpha}_{n})  \right) - \bm Q^{\alpha}_{\star} : \bm D^{\alpha}_n \right) \frac{\bm D^{\alpha}_n}{\left \lvert \bm D^{\alpha}_n \right \rvert^2},
\end{align*}
with
\begin{align*}
\bm Q^{\alpha}_{\star} := \frac12 \left( \bm Q^{\alpha}(\bm C_{n+1}, \bm \Gamma^{\alpha}_{n+\frac12}) + \bm Q^{\alpha}(\bm C_{n}, \bm \Gamma^{\alpha}_{n+\frac12}) \right) \quad \mbox{and} \quad \bm D^{\alpha}_n := \frac12 \left( \bm \Gamma^{\alpha}_{n+1} - \bm \Gamma^{\alpha}_n \right).
\end{align*}
It can be shown that 
\begin{align}
\label{eq:contraction-D-Q-alg3}
\bm D^{\alpha}_{n} : \bm Q^{\alpha}_{\mathrm{alg3}} =& -\frac12 \left( \Upsilon^{\alpha}(\tilde{\bm C}_{n+1}, \bm \Gamma^{\alpha}_{n+1}) - \Upsilon^{\alpha}(\tilde{\bm C}_{n+1}, \bm \Gamma^{\alpha}_{n}) + \Upsilon^{\alpha}(\tilde{\bm C}_{n}, \bm \Gamma^{\alpha}_{n+1}) - \Upsilon^{\alpha}(\tilde{\bm C}_{n}, \bm \Gamma^{\alpha}_{n}) \right) \displaybreak[2] \\
=& \mu^{\alpha} \left( \left\lvert \frac{\bm \Gamma^{\alpha}_{n+1} - \bm I}{2} \right\rvert^2 - \left\lvert \frac{\bm \Gamma^{\alpha}_{n} - \bm I}{2} \right\rvert^2  \right) - \left( \frac{\tilde{\bm S}_{\mathrm{iso}}^{\alpha}(\tilde{\bm C}_{n+1}) + \tilde{\bm S}_{\mathrm{iso}}^{\alpha}(\tilde{\bm C}_{n})}{2} - \hat{\bm S}^{\alpha}_{0} \right) : \frac{ \bm \Gamma^{\alpha}_{n+1} - \bm \Gamma^{\alpha}_{n} }{2}, \nonumber \\
=& \bm D^{\alpha}_{n} : \bm Q^{\alpha}_{\mathrm{alg2}}, \nonumber
\end{align}
according to \eqref{eq:def-alg-Q2}. This reveals the link between the algorithm of \cite{Martin2014} and the second-order scheme developed here. However, there are a few drawbacks of the scheme based on $\bm Q^{\alpha}_{\mathrm{alg3}}$. First, the integration of the constitutive relation $\eta^{\alpha}\left( \bm \Gamma^{\alpha}_{n+1} - \bm \Gamma^{\alpha}_{n} \right)/\Delta t_n = \bm Q^{\alpha}_{\mathrm{alg3}}$ is a nonlinear equation for $\bm \Gamma^{\alpha}_{n+1}$. It thus requires the local Newton-Raphson iteration at each quadrature point to determine the value of $\bm \Gamma^{\alpha}_{n+1}$. In this regard, the discrete gradient formula is not particularly advantageous to Simo's collocation approach \cite{Laursen2001,Simo1992c}. Second, the discrete gradient formula suffers from a numerical robustness issue. When the value of $\left \lvert \bm D^{\alpha}_n \right \rvert$ becomes too small (e.g. in steady-state calculations), the quotient term in the discrete gradient formula needs to be turned off to avoid floating-point exceptions \cite{Mohr2008,Liu2023}. 

The constitutive integration algorithm proposed here utilized the particular structure of the configurational free energy \eqref{eq:def-flv-Upsilon}, leading to a simple explicit update formula \eqref{eq:isv-update-formula-2}. The aforementioned two shortcomings are circumvented in this approach. When the configurational free energy is completely nonlinear with respect to $\bm \Gamma^{\alpha}$, it is inevitable to perform Newton-Raphson iterations locally at each quadrature point. Yet, we still advocate equations that do not involve a quotient form. Noticing that the essence of the algorithmic design of $\bm Q^{\alpha}$ is the satisfaction of the relation \eqref{eq:contraction-D-Q-alg3}, the equations for the internal state variables can be given as
\begin{align*}
\eta^{\alpha} \frac{\left\lvert \bm \Gamma^{\alpha}_{n+1} - \bm \Gamma^{\alpha}_{n} \right\rvert^2 }{\Delta t_n} + \Upsilon^{\alpha}(\tilde{\bm C}_{n+1}, \bm \Gamma^{\alpha}_{n+1}) - \Upsilon^{\alpha}(\tilde{\bm C}_{n+1}, \bm \Gamma^{\alpha}_{n}) + \Upsilon^{\alpha}(\tilde{\bm C}_{n}, \bm \Gamma^{\alpha}_{n+1}) - \Upsilon^{\alpha}(\tilde{\bm C}_{n}, \bm \Gamma^{\alpha}_{n}) = 0.
\end{align*}
In our opinion, the above equation is still superior to the integration scheme $\eta^{\alpha}\left( \bm \Gamma^{\alpha}_{n+1} - \bm \Gamma^{\alpha}_{n} \right)/\Delta t_n = \bm Q^{\alpha}_{\mathrm{alg3}}$, because numerically non-robust quotient terms like $\bm D^{\alpha}_n / \left\lvert \bm D^{\alpha}_{n} \right\rvert$ are not involved.

\section{Implementation}
\label{sec:implementation}
In this section, we outline the steps involved in the implementation of the fully-discrete energy-momentum consistent algorithms for viscoelastodynamics. In Section \ref{subsec:segregatd-predictor-multi-corrector-algorithm}, the solution algorithm for integrating the fully discrete balance equations is presented, which relies on a segregated predictor multi-corrector algorithm \cite{Liu2018,Rossi2016}. The segregation is achieved by leveraging the block factorization of the consistent tangent matrix and an assumption that the kinematic residual is zero over the Newton-Raphson iterations (see Proposition 5 of \cite{Liu2018}). Here in this study, a stronger consistency result is obtained due to the use of the specific temporal scheme, which is given in Proposition \ref{prop:kinematic-residual-remain-zero} below. In Section \ref{subsec:algorithm-constitutive-integration}, the implementation detail of the constitutive integration is presented.

\subsection{Segregated predictor multi-corrector algorithm}
\label{subsec:segregatd-predictor-multi-corrector-algorithm}
The discrete equations \eqref{eq:em-ts-kinematic}-\eqref{eq:em-ts-momentum} constitute a system of nonlinear algebraic equations to be solved over each time step, and we utilize the Newton-Raphson method to address this problem.  We denote 
\begin{align*}
\bm Y_{n+1,(l)} := \left\lbrace  \bm U_{n+1,(l)}, P_{n+1,(l)}, \bm V_{n+1,(l)} \right\rbrace^T
\end{align*}
as the solution vector at the Newton-Raphson iteration step $l$ for $0 \leq l \leq l_{\mathrm{max}}$. Moreover, let $\bm E_i$, $i = 1, 2, 3$, be the Cartesian basis vectors in the reference configuration, $\left\lbrace N_A  \bm E_i \right\rbrace$ be the basis functions that span the discrete function space $\mathcal S_h$, and $\left \lbrace M_B \right \rbrace$ be the basis functions that span the discrete function space $\mathcal P_h$. We define the residual vectors at the $l$-th iteration corresponding to the kinematics, mass, and momentum equations by substituting the basis functions in place of the test functions as follows,
\begin{align*}
& \boldsymbol{\mathrm R}_{(l)} := \left\lbrace \boldsymbol{\mathrm R}^k_{(l)}, \boldsymbol{\mathrm R}^p_{(l)}, \boldsymbol{\mathrm R}^m_{(l)} \right\rbrace^T, \nonumber \displaybreak[2] \\
& \boldsymbol{\mathrm R}^k_{(l)}\left(\bm Y_{n+1,(l)}\right) := \mathbf B^k\left( \bm Y_{n+1, (l)}, \bm Y_{n} \right), \nonumber \displaybreak[2] \\
& \boldsymbol{\mathrm R}^p_{(l)}\left(\bm Y_{n+1,(l)}\right) := \left\lbrace \mathbf B^p\left( M_B; \bm Y_{n+1, (l)}, \bm Y_{n} \right) \right \rbrace, \nonumber \displaybreak[2] \\
& \boldsymbol{\mathrm R}^m_{(l)}\left(\bm Y_{n+1,(l)}\right) := \left\lbrace \mathbf B^m\left( N_A \bm E_i; \bm Y_{n+1, (l)}, \bm Y_{n} \right) \right \rbrace.
\end{align*}
The consistent tangent matrix can be reduced to a segregated solution approach \cite{Liu2018,Liu2019a,Rossi2016}. The increments of the velocity and pressure are determined from the following linear problem,
\begin{align}
\label{eq:reduced-NR-linear-system}
\begin{bmatrix}
\boldsymbol{\mathrm K}^m_{(l), \bm V} + \frac{\Delta t_n}{2} \boldsymbol{\mathrm K}^m_{(l), \bm U} & \boldsymbol{\mathrm K}^m_{(l), P} \\[0.5mm]
\boldsymbol{\mathrm K}^p_{(l), \bm V} + \frac{\Delta t_n}{2} \boldsymbol{\mathrm K}^p_{(l), \bm U} &  \boldsymbol{\mathrm O}
\end{bmatrix} 
\begin{bmatrix}
\Delta \bm V_{n+1,(l)} \\[0.6mm]
\Delta P_{n+1,(l)}
\end{bmatrix}
= -
\begin{bmatrix}
\boldsymbol{\mathrm R}^m_{(l)} - \Delta t_n \boldsymbol{\mathrm K}^m_{(l), \bm U} \boldsymbol{\mathrm R}^k_{(l)} \\[0.38mm]
\boldsymbol{\mathrm R}^p_{(l)} - \Delta t_n \boldsymbol{\mathrm K}^p_{(l), \bm U} \boldsymbol{\mathrm R}^k_{(l)}
\end{bmatrix},
\end{align}
and the increment of displacement can be determined by
\begin{align}
\label{eq:reduced-displacement-increment}
\Delta \bm U_{n+1,(l)} = \frac{\Delta t_n}{2} \Delta \bm V_{n+1,(l)} - \Delta t_n \boldsymbol{\mathrm R}^k_{(l)}.
\end{align}
\begin{proposition}
\label{prop:kinematic-residual-remain-zero}
If the values of the solution vector at $l=0$ are chosen to be 
\begin{align*}
\bm Y_{n+1,(0)} = \left\lbrace  \bm U_{n}+\Delta t_n \bm V_n, P_{n}, \bm V_{n} \right\rbrace^T,
\end{align*}
one has $\boldsymbol{\mathrm R}^k_{(l)} = \bm 0$ for all $l \geq 0$.
\end{proposition}
\begin{proof}
With the given predictor, it can be shown that
\begin{align*}
\boldsymbol{\mathrm R}^k_{(0)} = \frac{\bm U_{n+1,(0)} - \bm U_n}{\Delta t_n} - \frac{1}{2}\left( \bm V_{n+1,(0)} + \bm V_n \right) = \bm 0.
\end{align*}
Given an increment of $\Delta \bm V_{n+1,(l)}$, the increment of the displacement field is given by 
\begin{align*}
\Delta \bm U_{n+1,(l)} = \frac{\Delta t_n}{2} \Delta \bm V_{n+1,(l)}.
\end{align*}
Consequently, one has
\begin{align*}
\boldsymbol{\mathrm R}^k_{(l+1)} =& \frac{\bm U_{n+1,(l+1)} - \bm U_n}{\Delta t_n} - \frac12 \left( \bm V_{n+1,(l+1)} + \bm V_n \right) \displaybreak[2] \nonumber \\
=& \frac{\bm U_{n+1,(l)} - \bm U_n + \Delta \bm U_{n+1,(l)}}{\Delta t_n} - \frac12 \left( \bm V_{n+1,(l)} + \bm V_n + \Delta \bm V_{n+1,(l)} \right) \displaybreak[2] \nonumber \\
=& \boldsymbol{\mathrm R}^k_{(l)} + \frac{1}{\Delta t_n} \Delta \bm U_{n+1,(l)} - \frac12 \Delta \bm V_{n+1,(l)} \displaybreak[2] \nonumber \\
=& \boldsymbol{\mathrm R}^k_{(l)}.
\end{align*}
By mathematical induction, it can be summarized that, with the choice of the predictor, the residual vector $\boldsymbol{\mathrm R}^k_{(l)}$ will remain zero for all iterations. 
\end{proof}

The result of Proposition \ref{prop:kinematic-residual-remain-zero} significantly simplifies the assembly of the right-hand side of \eqref{eq:reduced-NR-linear-system}. We may now summarize our above discussion into the following segregated predictor multi-corrector algorithm.

\begin{myenv}{Segregated predictor multi-corrector algorithm}
\noindent \textbf{Predictor stage:} Set
\begin{align*}
\bm Y_{n+1, (0)} = \left\lbrace  \bm U_{n}+\Delta t_n \bm V_n, P_{n}, \bm V_{n} \right\rbrace^T.
\end{align*}

\noindent \textbf{Multi-corrector stage:} Repeat the following steps for $l=0, 1, ..., l_{max}$
\begin{enumerate}
    \item Assemble the residual vector $\boldsymbol{\mathrm R}^m_{(l)}$ and $\boldsymbol{\mathrm R}^p_{(l)}$ using $\bm Y_{n+1, (l)}$ and $\bm Y_{n}$.
    \item Let $\| \boldsymbol{\mathrm R}_{(l)} \|_{\mathfrak l_2}$ denote the $\mathfrak l_2$-norm of the residual vector
\begin{align*}
\boldsymbol{\mathrm R}_{(l)} := \left\lbrace \bm 0, \boldsymbol{\mathrm R}^p_{(l)}, \boldsymbol{\mathrm R}^m_{(l)} \right\rbrace^T,
\end{align*}    
    and let $\mathrm{tol}_{\mathrm{R}}$ and $\mathrm{tol}_{\mathrm{A}}$ denote the prescribed relative and absolute tolerances, respectively. If either of the following stopping criteria
    \begin{align*}
    \frac{\| \boldsymbol{\mathrm R}_{(l)} \|_{\mathfrak l_2}}{\| \boldsymbol{\mathrm R}_{(0)} \|_{\mathfrak l_2}} \leq \mathrm{tol}_{\mathrm{R}}, \qquad
    \| \boldsymbol{\mathrm R}_{(l)} \|_{\mathfrak l_2} \leq \mathrm{tol}_{\mathrm{A}},
    \end{align*}
    is satisfied, set 
    \begin{align*}
    \bm Y_{n+1} = \bm Y_{n+1, (l)},
    \end{align*}
    and exit the multi-corrector stage. Otherwise, continue to step 4.
    \item Assemble the following sub-tangent matrices,
    \begin{align*}
    \boldsymbol{\mathrm A}_{(l)} := \boldsymbol{\mathrm K}^m_{(l), \bm V} + \frac{\Delta t_n}{2} \boldsymbol{\mathrm K}^m_{(l), \bm U}, \quad
    \boldsymbol{\mathrm B}_{(l)} := \boldsymbol{\mathrm K}^m_{(l), P}, \quad
    \boldsymbol{\mathrm C}_{(l)} := \boldsymbol{\mathrm K}^p_{(l), \bm V} + \frac{\Delta t_n}{2} \boldsymbol{\mathrm K}^p_{(l), \bm U}.
    \end{align*}
    \item Solve the following linear system for $\Delta \bm V_{n+1, (l)}$ and $\Delta P_{n+1, (l)}$,
    \begin{align}
    \label{eq:pred_multi_correct_linear_system}
    \begin{bmatrix}
    \boldsymbol{\mathrm A}_{(l)} & \boldsymbol{\mathrm B}_{(l)} \\[1mm]
    \boldsymbol{\mathrm C}_{(l)} & \boldsymbol{\mathrm O}
    \end{bmatrix}
    \begin{bmatrix}
    \Delta \bm V_{n+1, (l)} \\[1mm]
    \Delta P_{n+1, (l)}
    \end{bmatrix} = -
    \begin{bmatrix}
    \boldsymbol{\mathrm R}^m_{(l)} \\[1mm]
    \boldsymbol{\mathrm R}^p_{(l)}
    \end{bmatrix}.
    \end{align}
    \item Obtain $\Delta \bm U_{n+1, (l)}$ from $\Delta \bm V_{n+1, (l)}$ via the relation
    \begin{align*}
    \Delta \bm U_{n+1, (l)} = \frac{\Delta t_n}{2}\Delta \bm V_{n+1, (l)}.
    \end{align*}
    \item Update the solution vector as
    \begin{align*}
    & \bm Y_{n+1, (l+1)} = \bm Y_{n+1, (l)} + \Delta \bm Y_{n+1, (l)}.
    \end{align*}
\end{enumerate}
\end{myenv}

\subsection{The algorithm for the constitutive integration}
\label{subsec:algorithm-constitutive-integration}
In the above predictor multi-corrector algorithm, one needs to repeatedly evaluate the stress and elasticity tensor for constructing the linear system \eqref{eq:pred_multi_correct_linear_system}. We outline the algorithm for the evaluation of stress using the second-order formula given in Section \ref{sec:second-order-update-formula-ISV}. For notational simplicity, the following discussion will neglect the subscript $(l)$ that represents the iteration of the Newton-Raphson method.

\begin{myenv}{Algorithmic stress evaluation algorithm}
\noindent \textbf{Input:} Given the time step size $\Delta t_n$, the internal state variables $\lbrace \bm \Gamma^{\alpha}_{n} \rbrace_{\alpha=1}^{m}$, and
\begin{align*}
\bm Y_{n+1} = \left\lbrace  \bm U_{n+1}, P_{n+1}, \bm V_{n+1} \right\rbrace^T.
\end{align*}

\begin{enumerate}
    \item Calculate the deformation gradient and the strain measures based on the displacement $\bm U_{n+1}$,
    \begin{gather*}
    \bm F_{n+1} = \bm I + \nabla_{\bm X} \bm U_{n+1}, \quad J_{n+1} = \mathrm{det}\left( \bm F_{n+1} \right), \quad \bm C_{n+1} = \bm F_{n+1}^T \bm F_{n+1}, \quad \tilde{\bm C}_{n+1} = J_{n+1}^{-2/3} \bm C_{n+1}, \displaybreak[2] \\
    \bm C_{n+\frac12} = \frac12 \left( \bm C_{n+1} + \bm C_{n} \right), \quad  \tilde{\bm C}_{n+\frac12} = \mathrm{det}\left( \bm C_{n+\frac12} \right)^{-\frac13} \bm C_{n+\frac12}.
    \end{gather*}
    \item Calculate the following fictitious stresses and elasticity tensors for $\alpha = 1, \cdots, m$,
    \begin{align*}
    & \tilde{\bm S}^{\alpha}_{\mathrm{iso} \: n+1} = 2 \frac{\partial G^{\alpha}}{\partial \tilde{\bm C}}(\tilde{\bm C}_{n+1}), \quad \tilde{\bm S}^{\alpha}_{\mathrm{iso} \: n} = 2 \frac{\partial G^{\alpha}}{\partial \tilde{\bm C}}(\tilde{\bm C}_{n}).
    \end{align*}
    
    \item Calculate the internal state variables
    \begin{align*}
    & \bm \Gamma^{\alpha}_{n+1} = \frac{\Delta t_n}{\eta^{\alpha} + \mu^{\alpha} \Delta t_n / 2} \left( \frac{\tilde{\bm S}_{\mathrm{iso} \: n+1}^{\alpha} + \tilde{\bm S}_{\mathrm{iso} \: n}^{\alpha}}{2} - \hat{\bm S}^{\alpha}_{0} - \left( \frac{\mu^{\alpha}}{2} - \frac{\eta^{\alpha}}{\Delta t_n} \right) \bm \Gamma^{\alpha}_n  + \mu^{\alpha} \bm I \right), \\
    & \bm \Gamma^{\alpha}_{n+\frac12} = \frac12 \left(\bm \Gamma^{\alpha}_{n+1} + \bm \Gamma^{\alpha}_{n} \right).
    \end{align*}
    
    \item Calculate the following stresses,
    \begin{gather*}
    \tilde{\bm S}^{\infty}_{\mathrm{iso} \: n+\frac12} = 2 \frac{\partial G^{\infty}_{\mathrm{iso}}}{\partial \tilde{\bm C}}\left( \tilde{\bm C}_{n+\frac12} \right), \quad \tilde{\bm S}^{\alpha}_{\mathrm{iso} \: n+\frac12} = 2 \frac{\partial G^{\alpha}}{\partial \tilde{\bm C}}(\tilde{\bm C}_{n+\frac12}), \displaybreak[2] \\
    \tilde{\bm S}^{\alpha}_{\mathrm{neq} \: n+\frac12} = \left(  \frac{\partial \tilde{\bm S}^{\alpha}_{\mathrm{iso}}}{\partial \tilde{\bm C}}(\tilde{\bm C}_{n+\frac12}) \right) : \left( \frac{1}{\mu^{\alpha}} \left( \tilde{\bm S}^{\alpha}_{\mathrm{iso} \: n+\frac12} - \hat{\bm S}^{\alpha}_{0}\right) - \bm \Gamma^{\alpha}_{n+\frac12} + \bm I \right).
    \end{gather*}
    
    \item Calculate the fictitious stress
    \begin{align*}
    \tilde{\bm S}_{n+\frac12} = \tilde{\bm S}^{\infty}_{\mathrm{iso} \: n+\frac12} + \sum_{\alpha=1}^{m} \tilde{\bm S}^{\alpha}_{\mathrm{neq} \: n+\frac12}.
    \end{align*}
    
    \item Calculate the projection tensor
    \begin{align*}
    \mathbb P_{n+\frac12} = \mathbb I - \frac13 \bm C^{-1}_{n+\frac12} \otimes \bm C_{n+\frac12}
    \end{align*}
    and the isochoric part of the second Piola-Kirchhoff stress
    \begin{align*}
    \bm S_{\mathrm{iso} \: n+\frac12} = \mathrm{det}\left( \bm C_{n+\frac12} \right)^{-\frac13} \mathbb P_{n+\frac12} : \tilde{\bm S}_{n+\frac12}.
    \end{align*}
    
    \item Calculate $\bm Z_n = \left( \bm C_{n+1} - \bm C_n \right)/2$ and the energies
    \begin{align*}
   G_{\mathrm{iso} \: n+1} = G_{\mathrm{iso}}(\tilde{\bm C}_{n+1}, \bm \Gamma^1_{n+\frac12}, \cdots, \bm \Gamma^{m}_{n+\frac12}), \quad G_{\mathrm{iso} \: n} = G_{\mathrm{iso}}(\tilde{\bm C}_{n}, \bm \Gamma^1_{n+\frac12}, \cdots, \bm \Gamma^{m}_{n+\frac12}).
    \end{align*}
    
    \item Calculate the algorithmic stress
    \begin{align*}
    \bm S_{\mathrm{iso} \: \mathrm{alg}} := \bm S_{\mathrm{iso} \: n+\frac12} + \frac{ G_{\mathrm{iso} \: n+1} - G_{\mathrm{iso} \: n} - \bm S_{\mathrm{iso} \: n+\frac12} : \bm Z_{n} }{\left\lvert\bm Z_{n}\right\rvert} \frac{\bm Z_{n}}{\left\lvert \bm Z_{n} \right\rvert }.
    \end{align*}
\end{enumerate}
\end{myenv}

The implementation of the integration algorithm given in Section \ref{sec:first-order-update-formula-ISV} follows the same procedures except that one needs to invoke \eqref{eq:constitutive-integration-Gamma-1} in Step 3. Besides the evaluation of the algorithmic stress, the algorithmic elasticity tensor is also needed in implicit calculations. In \ref{sec:appendix-A} and \ref{sec:appendix-B}, the detailed form of the algorithmic elasticity tensor based on the two constitutive integration formulas is presented, respectively.

\section{Numerical results}
\label{sec:numerical_examples}
Unless otherwise specified, the following choices are made in the numerical investigations.
\begin{enumerate}
\item The meter-kilogram-second system of units is used;
\item The numerical schemes using the constitutive integration algorithms presented in Sections \ref{sec:first-order-update-formula-ISV} and \ref{sec:second-order-update-formula-ISV} are referred to as Scheme-1 and Scheme-2, respectively;
\item The discrete pressure function space is generated by $k$-refinement to achieve the highest possible continuity, and then the degree is elevated with $\mathsf{a} = 1$ and $\mathsf{b}=0$ to construct the velocity function space;
\item We use $\mathsf p + \mathsf a + 2$ Gaussian quadrature points in each direction;
\item We use $\mathrm{tol}_\mathrm{R} = 10^{-10}$, $\mathrm{tol}_\mathrm{A} = 10^{-10}$, $l_{\mathrm{max}} = 10$ as the stopping criteria in the predictor multi-corrector algorithm;
\item When $|\bm{Z}_n| <10^{-10}$, the stress enhancements in the two algorithmic stresses are turned off to avoid floating-point exceptions;
\item The grad-div stabilization parameter $\gamma$ is set to zero.
\end{enumerate}

\begin{table}[htbp]
  \centering 
  \begin{tabular}{ m{.35\textwidth} m{.35\textwidth} }
    \hline
    \begin{minipage}{.35\textwidth}
    \centering
      \includegraphics[width=0.6\linewidth, trim=220 260 200 240, clip]{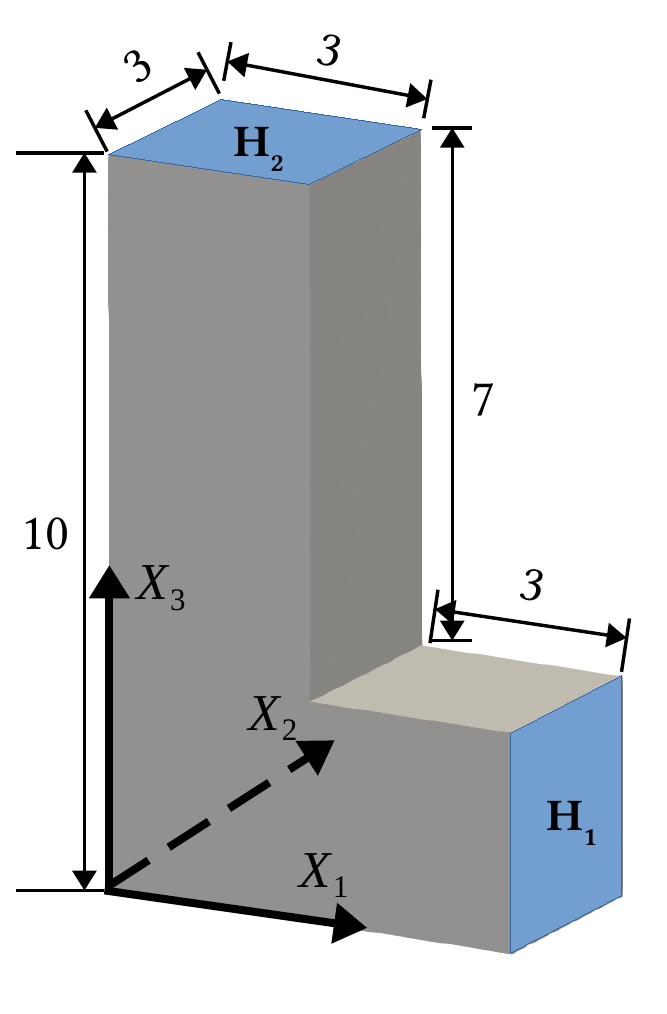}
    \end{minipage}
    &
    \begin{minipage}{.35\textwidth}
      \begin{itemize}
        \item[] Material properties:
        \item[] $G_{\mathrm{iso}}^{\infty} = \frac{c_1}{2}(\tilde{I}_1 -3) + \frac{c_2}{2}(\tilde{I}_2 -3)$
        \item[] $\rho_0 = 1.0\times 10^3$
        \item[] $E = 25000$
        \item[] $c_1=c_2=E/6$
        \item[] $m=1$
        \item[] $\beta^{\infty}_{1}=1$, $\mu^1=c_1$, $\eta^1=0.1 \mu^1$           
      \end{itemize}
    \end{minipage}   
    \\
    \hline
  \end{tabular}
  \caption{The three-dimensional L-shaped block: problem definition. The parameter $\beta^{\infty}_{1}$ is only used in the MIPC model. $\bm{H}_1$ and $\bm{H}_2$ denote the two applied surface tractions.}
\label{table:L-blcok}
\end{table}

\subsection{Tumbling L-shaped block}
We investigate the numerical performances of our proposed schemes by considering the tumbling L-shaped block problem \cite{Simo1992c}. The geometry of the L-shaped domain is constructed by three NURBS patches. The Mooney-Rivlin model is used to describe the material behavior of the equilibrium part, and two finite linear viscoelastic models, i.e., the HS and MIPC models, are of interest here. The problem definition is summarized in Table \ref{table:L-blcok}. The block is subjected to two traction loads for $0<t<5$ as follows,
\begin{gather*}
\bm{H}_1 = \mathcal H \times [-250, 100, -300]^T, \quad \bm{H}_2 = \mathcal H \times [150, -250, 350]^T, \quad
\mathcal H := 
\begin{cases}
t, & 0 \leq t \leq 2.5, \\
5-t, &  2.5 < t \leq 5, \\
0, & 5 < t.
\end{cases}
\end{gather*}
We discretize the geometry with 180 elements and $\mathsf p=1$. We have utilized finer meshes and higher-order basis functions to guarantee that the presented results are independent of the spatial mesh. The problem is integrated up to $T = 100$ with a uniform time step size $\Delta t_n = 0.1$. Figure \ref{fig:Ldoamin_deformation} displays deformed configurations at selected time instances with the pressure distributions depicted. 

\begin{figure}
\begin{center}
\begin{tabular}{cccc}
\includegraphics[angle=0, trim=250 100 100 160, clip=true, scale = 0.066]{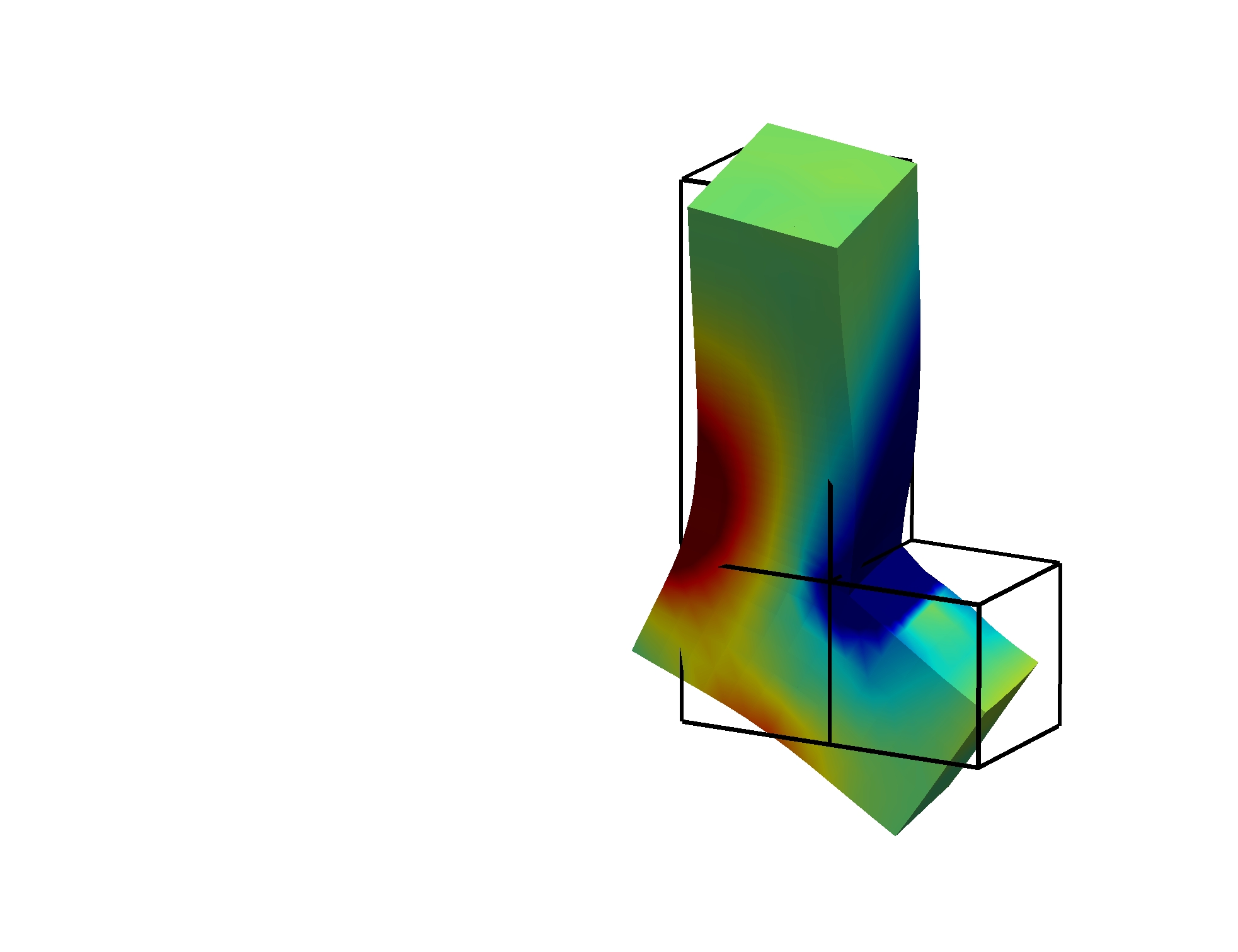} &
\includegraphics[angle=0, trim=250 100 100 160, clip=true, scale = 0.066]{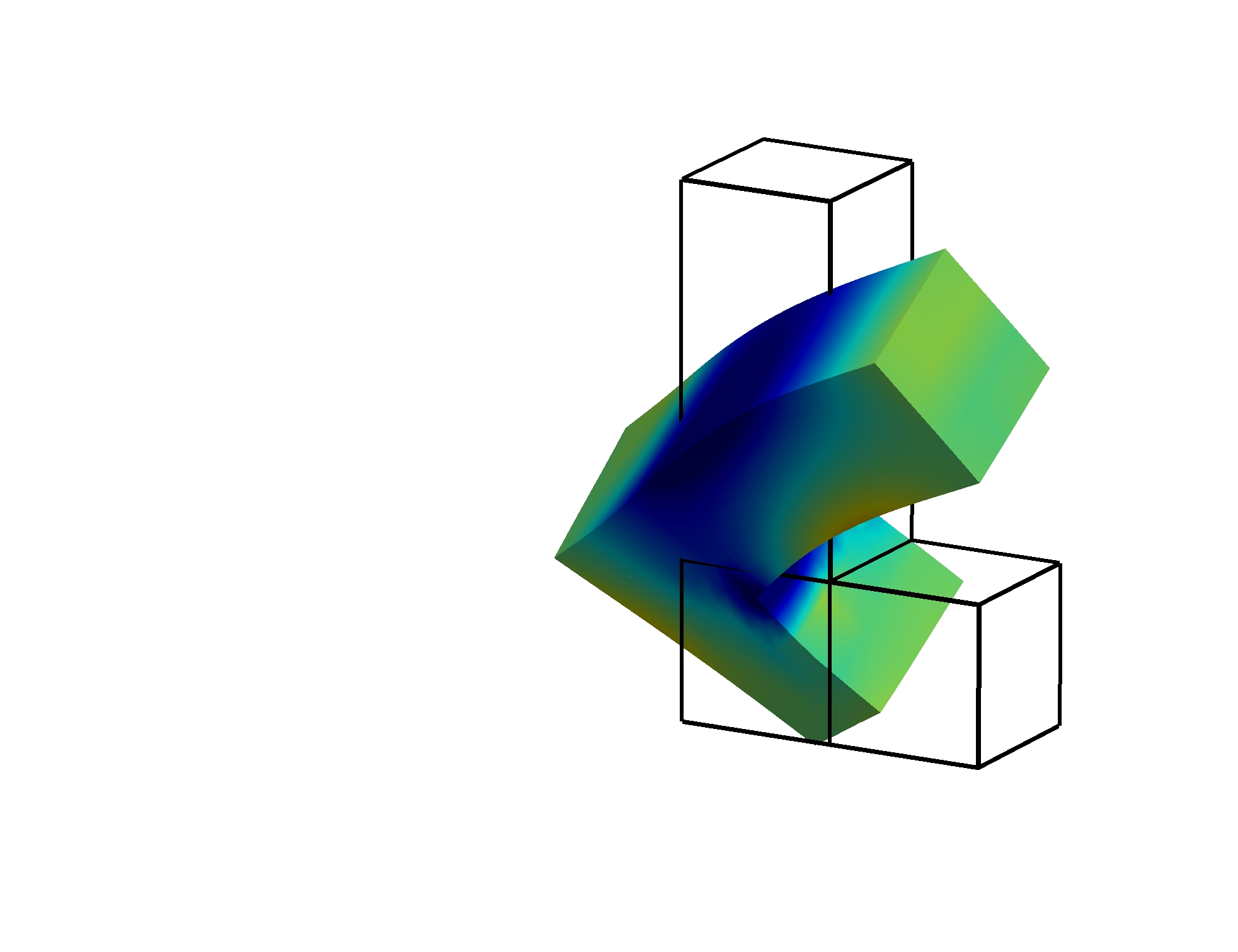} &
\includegraphics[angle=0, trim=250 100 100 160, clip=true, scale = 0.066]{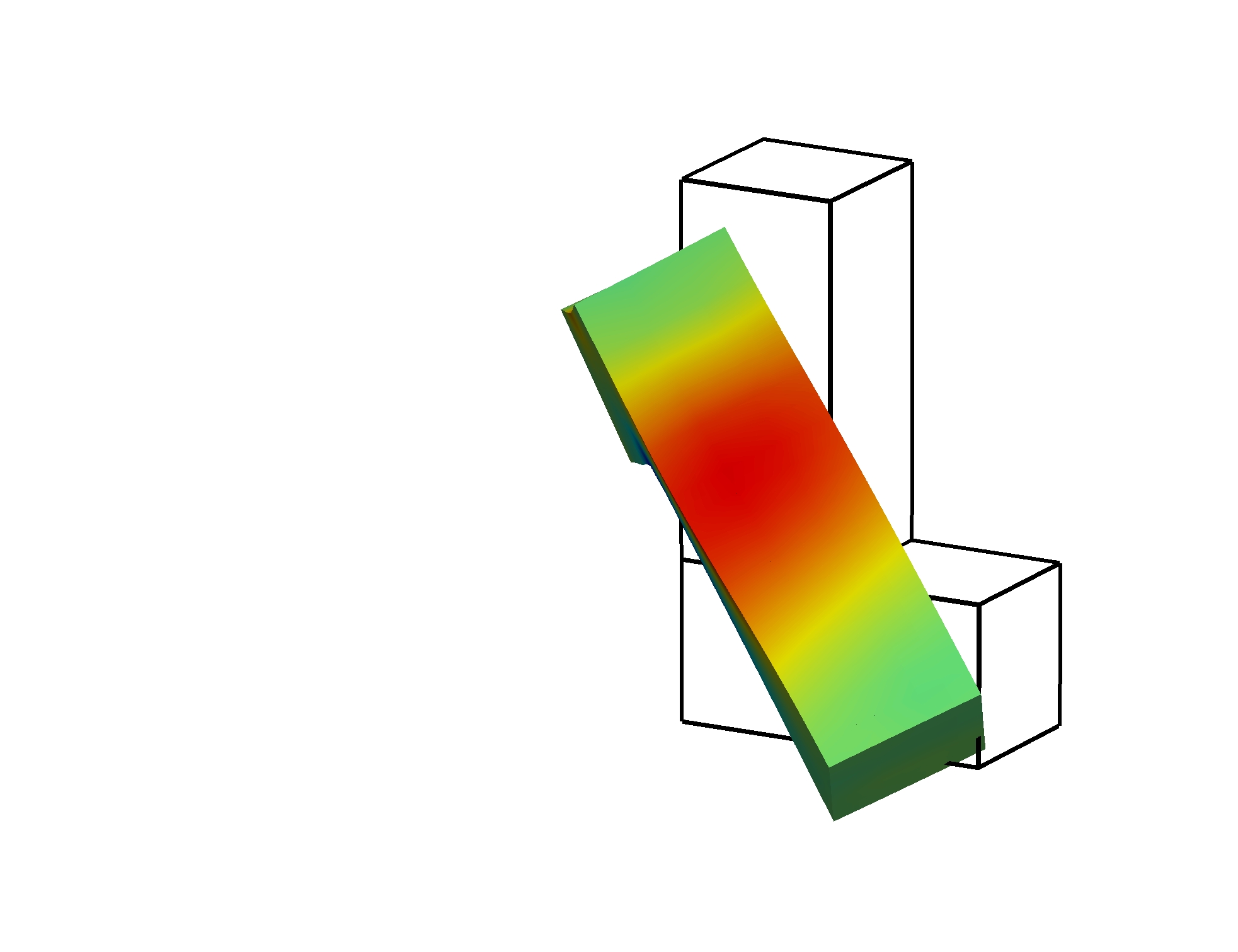} &
\includegraphics[angle=0, trim=250 100 100 160, clip=true, scale = 0.066]{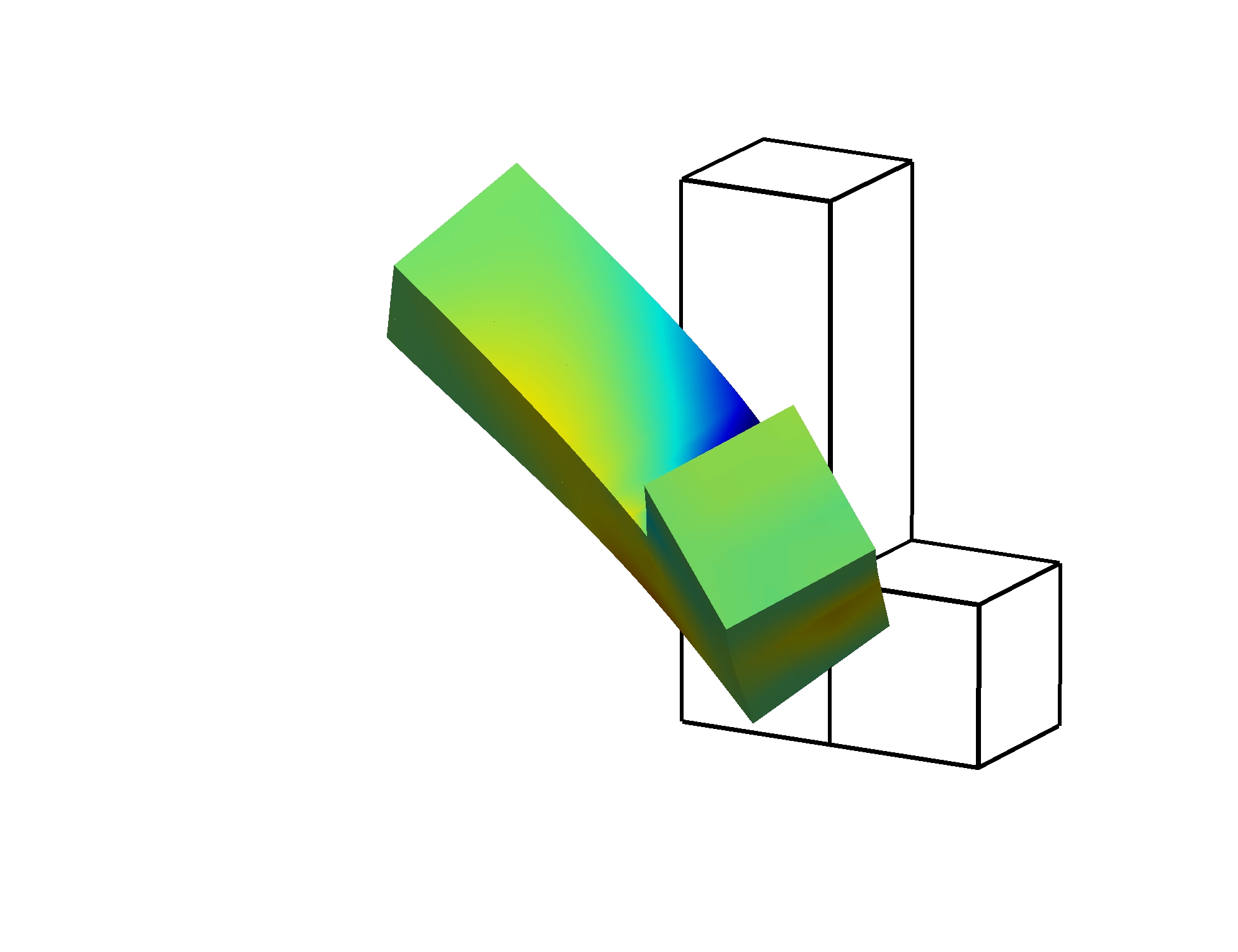} \\
$t = 5$ & $10$ & $20$ & $40$\\
\includegraphics[angle=0, trim=250 100 100 160, clip=true, scale = 0.066]{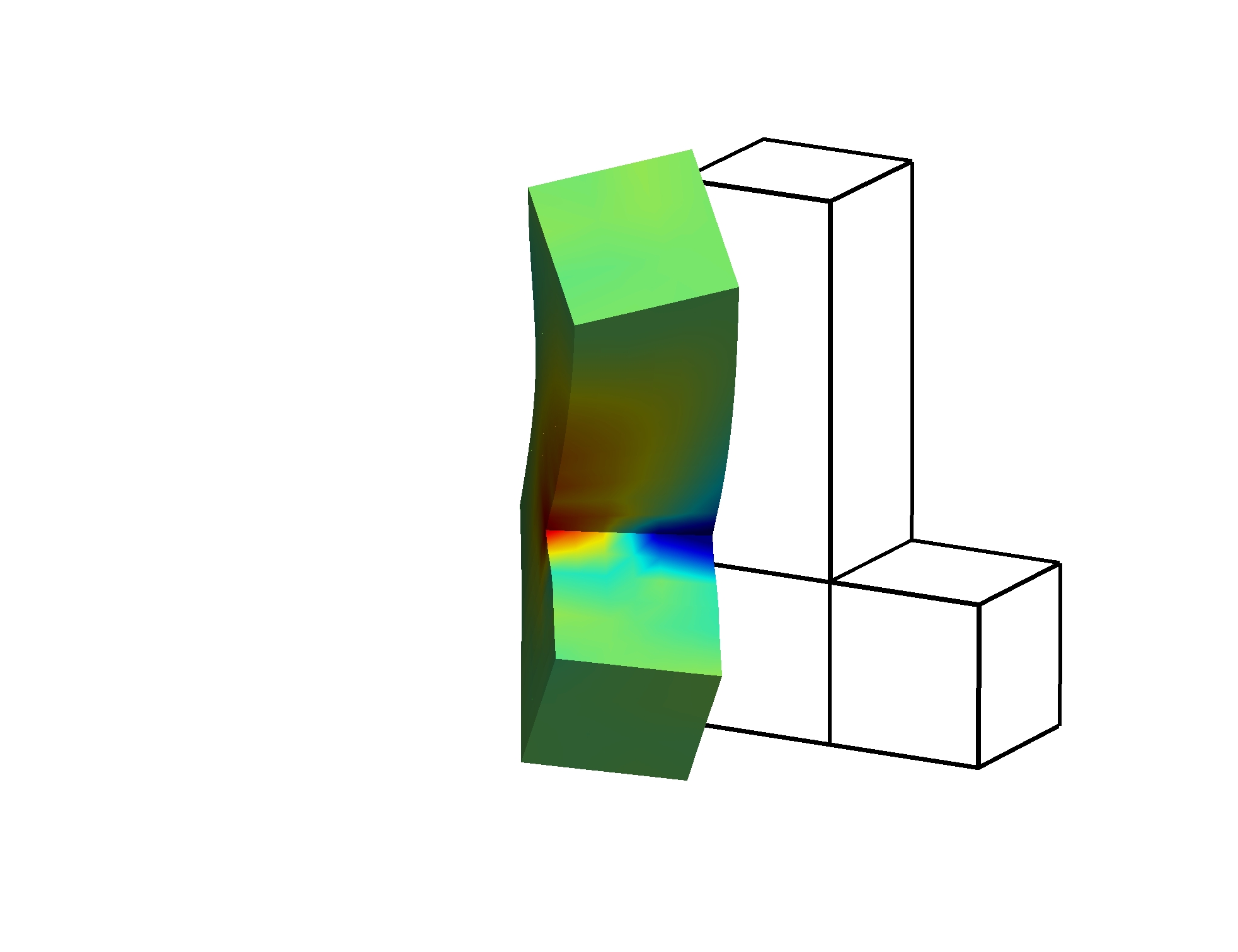} &
\includegraphics[angle=0, trim=250 100 100 160, clip=true, scale = 0.066]{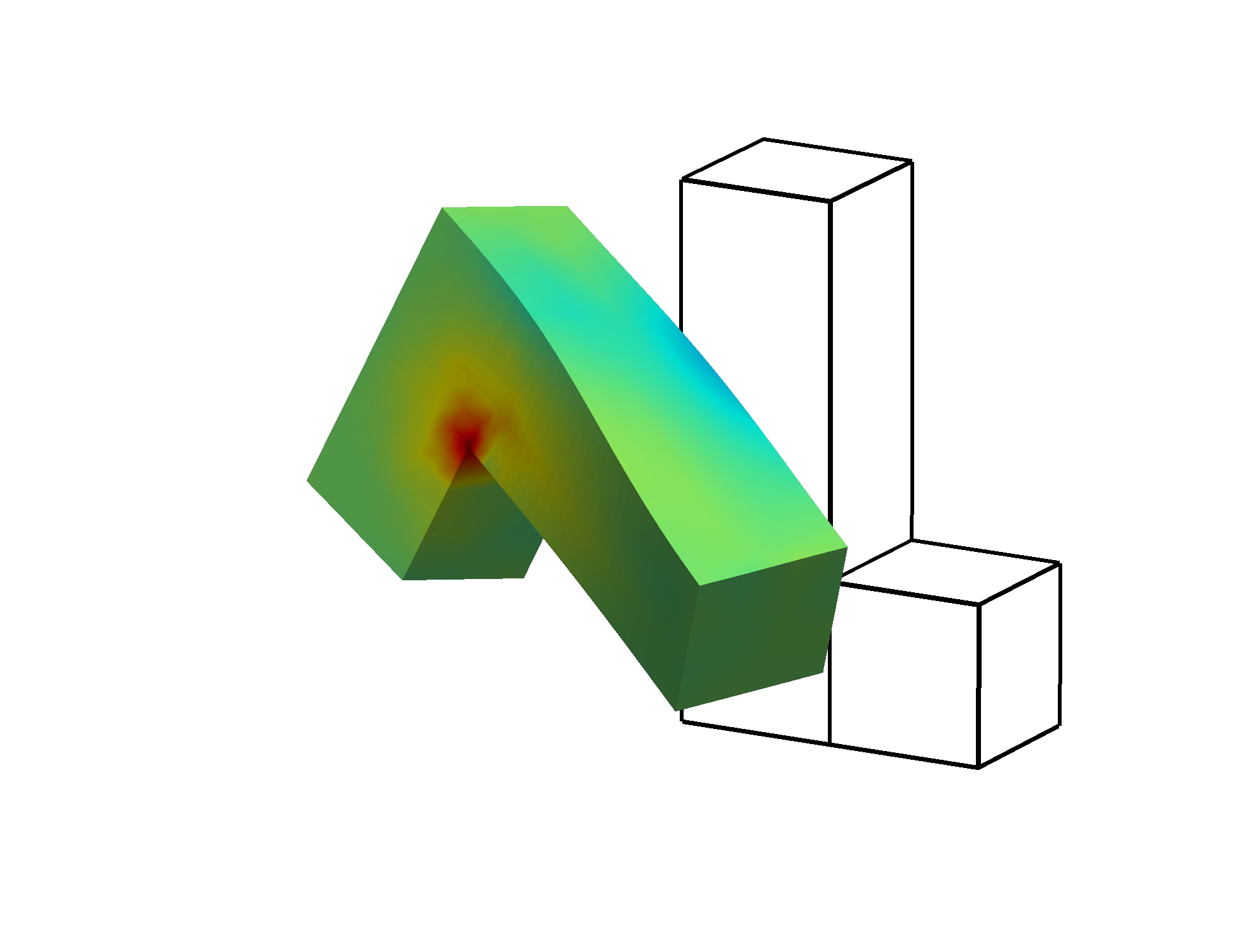} &
\includegraphics[angle=0, trim=250 100 100 160, clip=true, scale = 0.066]{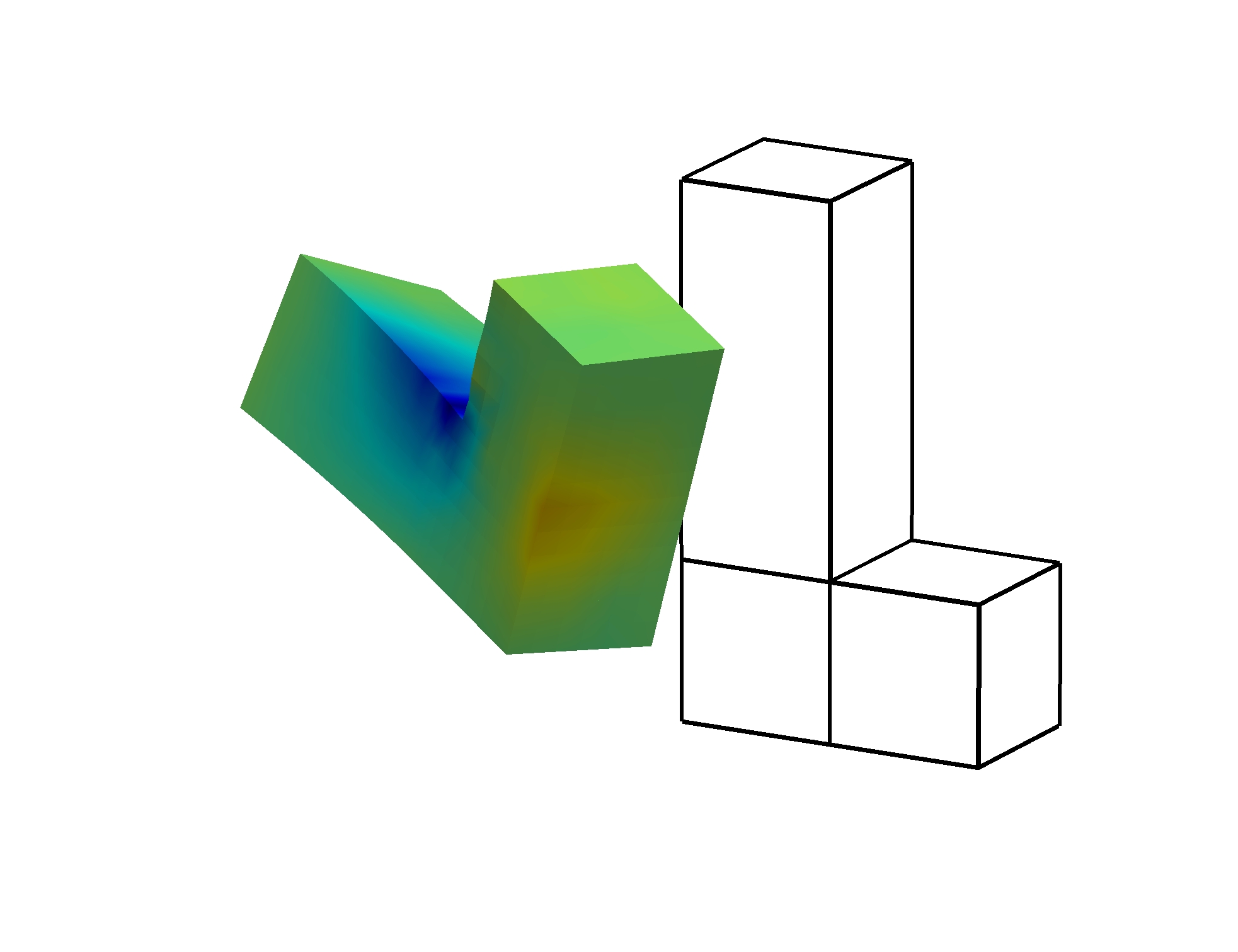} &
\includegraphics[angle=0, trim=250 100 100 160, clip=true, scale = 0.066]{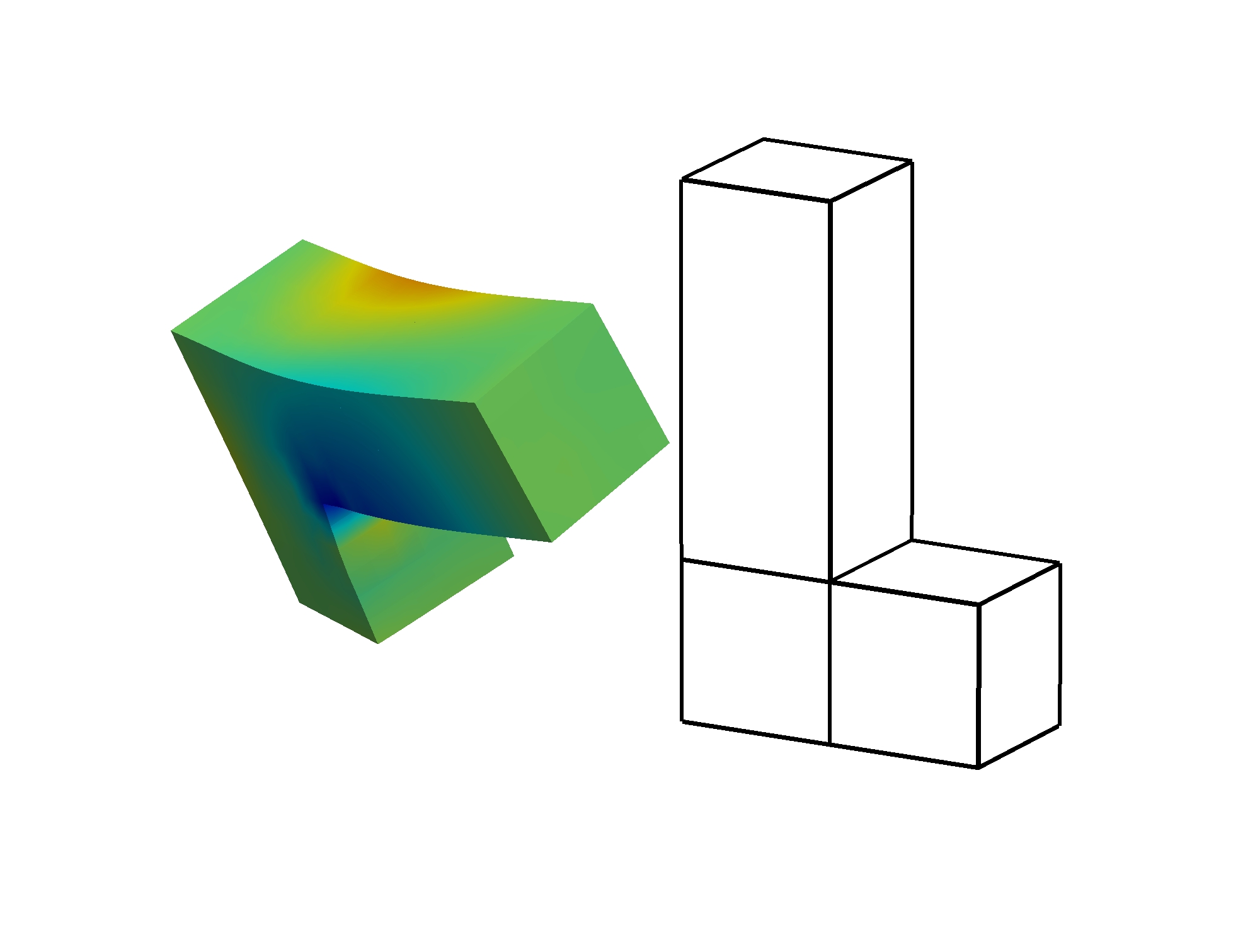} \\
$50$ & $60$ & $80$ & $100$\\
\multicolumn{4}{c}{ \includegraphics[angle=0, trim=320 430 320 720, clip=true, scale = 0.25]{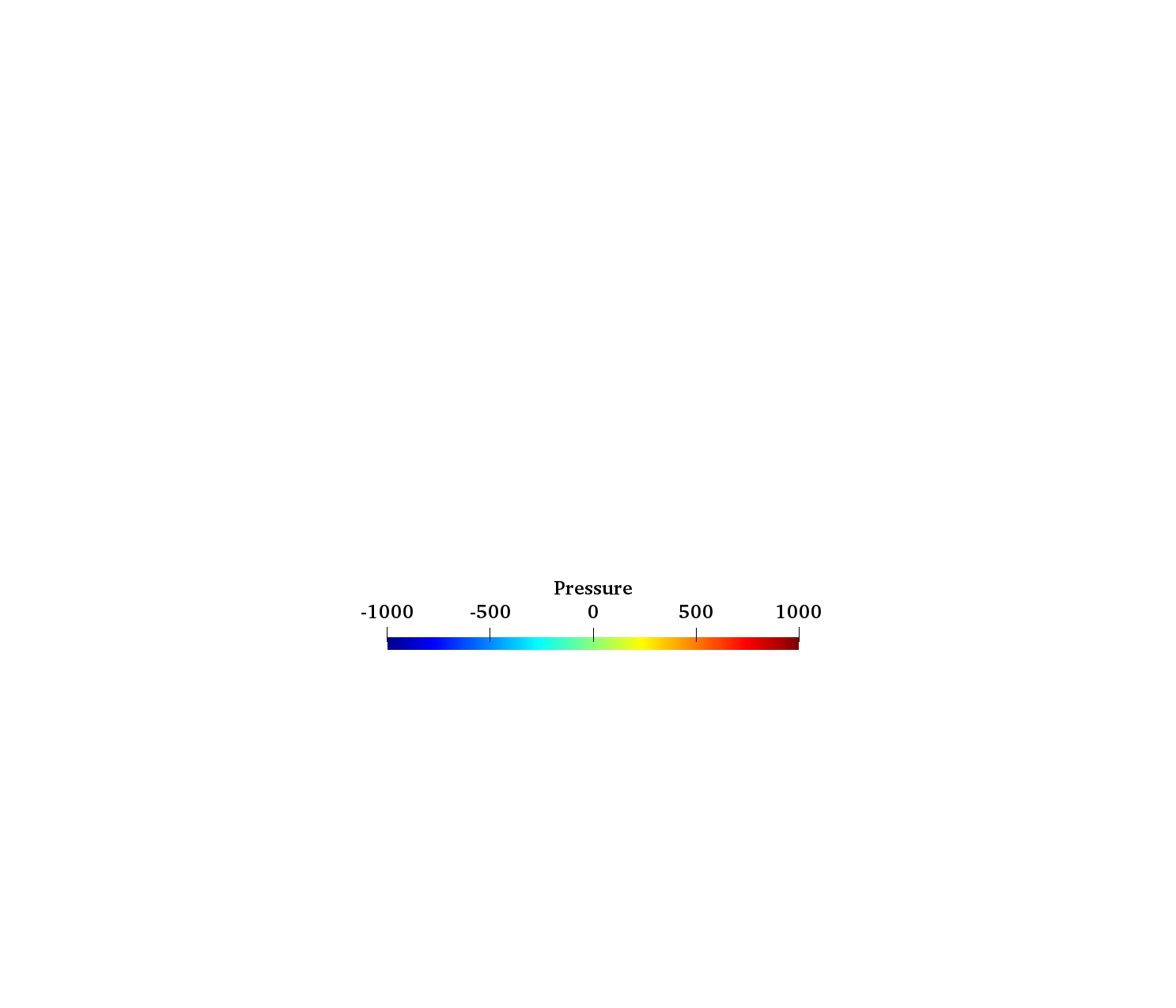} }
\end{tabular}
\caption{Snapshots of the deformed L-shaped block with the pressure distribution. The black grid depicts the initial configuration.}
\label{fig:Ldoamin_deformation}
\end{center}
\end{figure}

\begin{figure}
\begin{center}
\begin{tabular}{cc}
\includegraphics[angle=0, trim=80 50 100 100, clip=true, scale = 0.23]{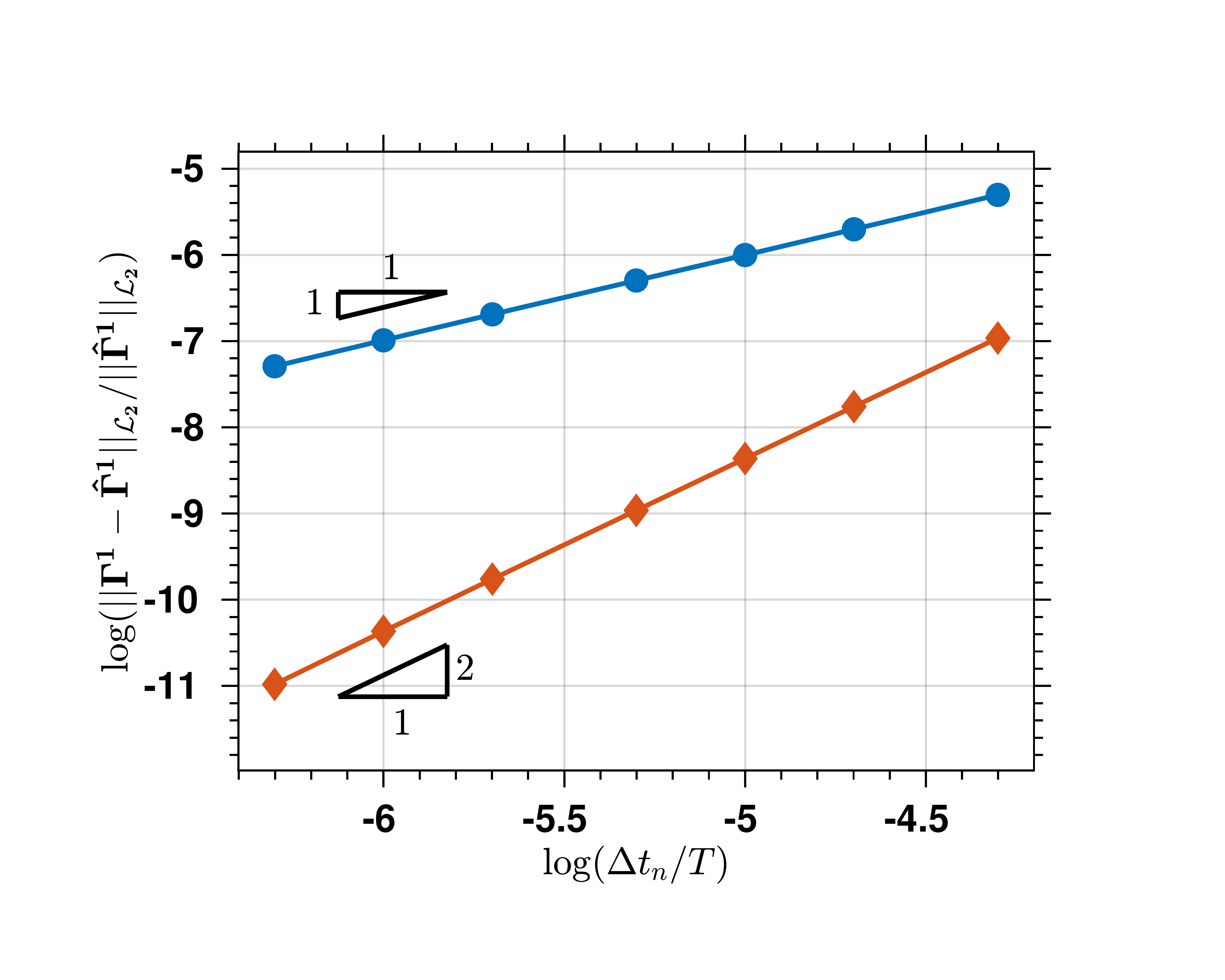} &
\includegraphics[angle=0, trim=80 50 100 100, clip=true, scale = 0.23]{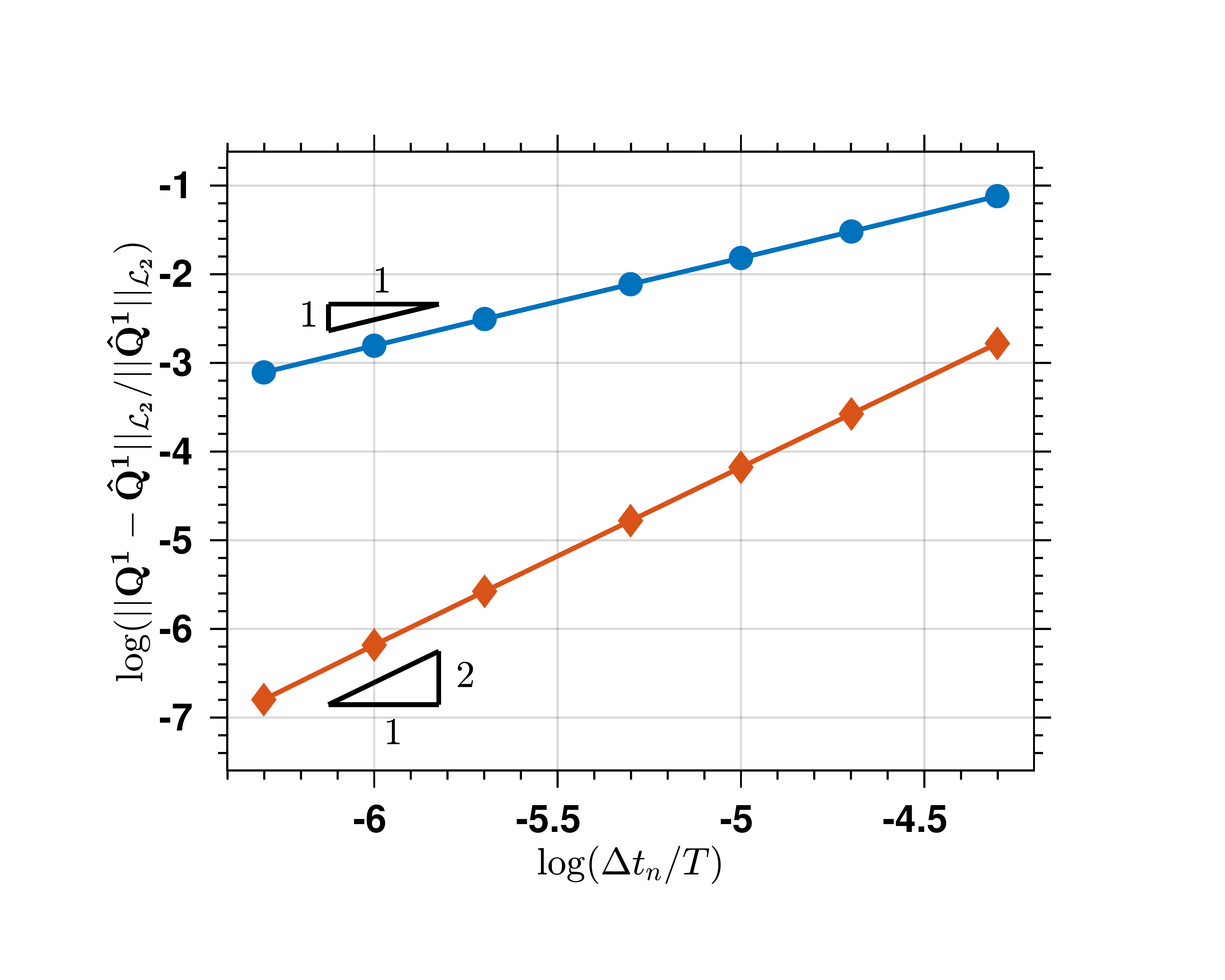} \\
\includegraphics[angle=0, trim=80 50 100 80, clip=true, scale = 0.23]{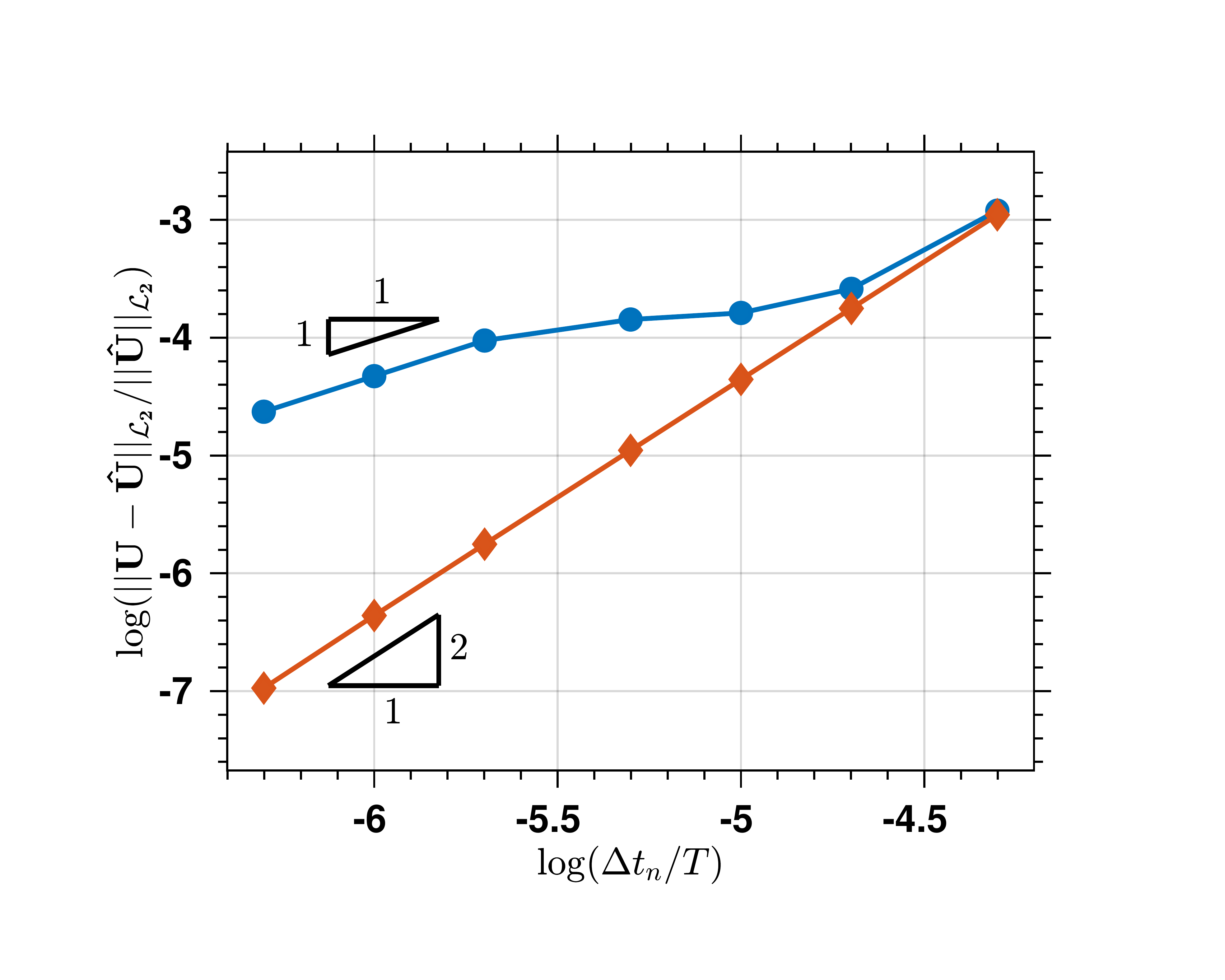} &
\includegraphics[angle=0, trim=80 50 100 80, clip=true, scale = 0.23]{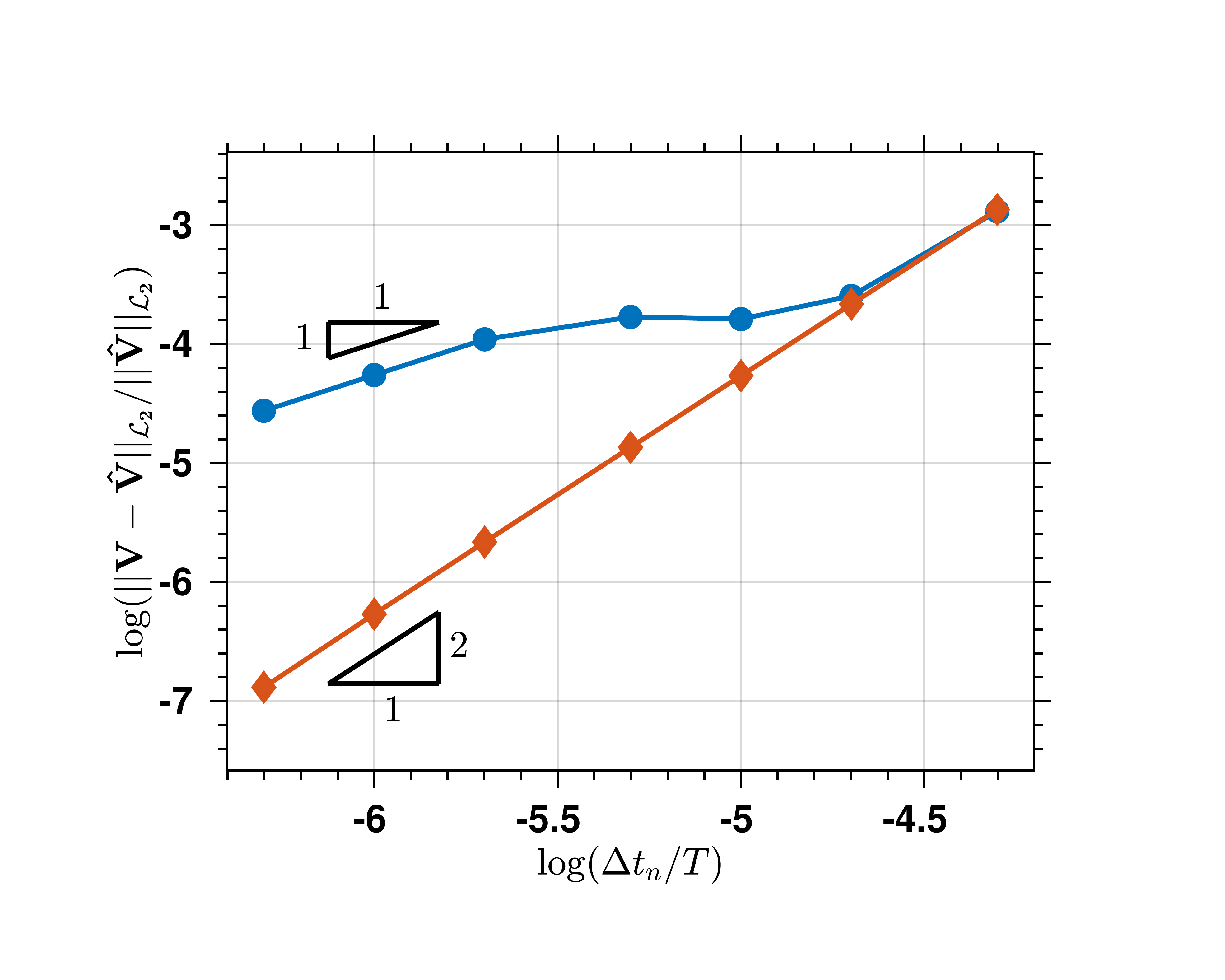}  \\
\multicolumn{2}{c}{\includegraphics[angle=0, trim=80 50 100 80, clip=true, scale = 0.23]{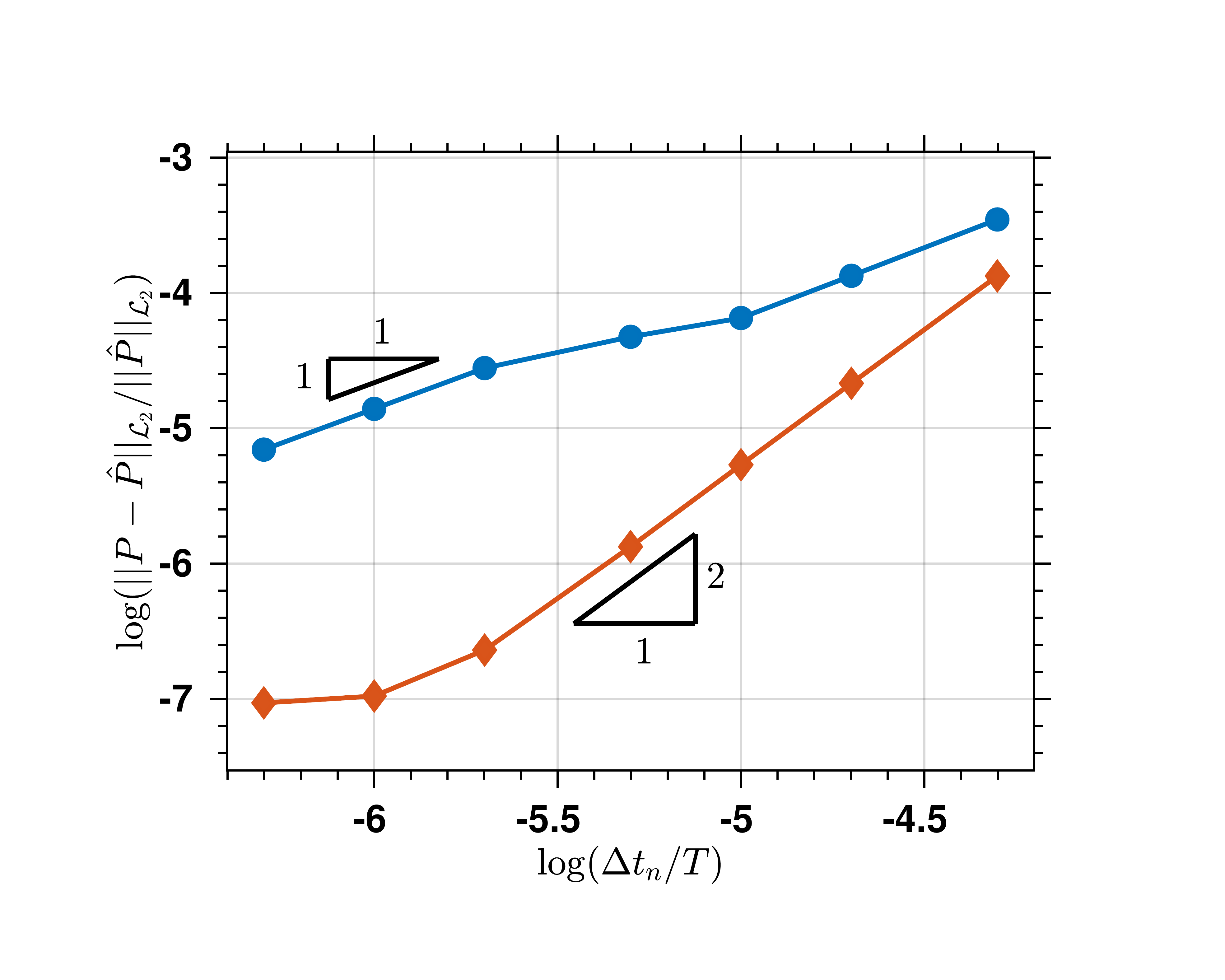}} \\
\multicolumn{2}{c}{ \includegraphics[angle=0, trim=0 170 0 680, clip=true, scale = 0.25]{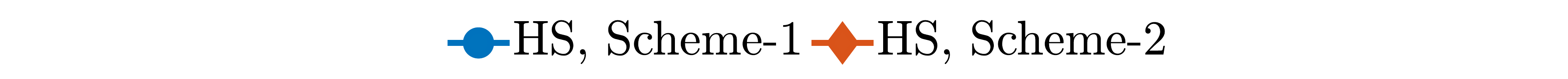} }
\end{tabular}
\caption{The convergence rates of the displacement $\bm{U}$,  velocity $\bm{V}$, pressure ${P}$, and internal state variables $\bm{\Gamma}^{1}$ and $\bm{Q}^{1}$ at time $t=0.1$. The hatted variables $\hat{(\bullet)}$ represent the overkill solutions obtained with the time step size $\Delta t_n = 1.0\times10^{-5}$.}
\label{fig:Ldomain_convergence}
\end{center}
\end{figure}

We first verify the convergence behavior of the displacement, velocity, pressure, and the internal state variables $\bm{Q}^{\alpha}$ and $\bm{\Gamma}^{\alpha}$ given in Section \ref{sec:time-integration}. An overkill solution at time $t=0.1$ is obtained by taking $\Delta t_n = 1.0\times10^{-5}$. The HS and MIPC models exhibit identical convergence behavior, and thus we choose to present the results of the HS model here. Figure \ref{fig:Ldomain_convergence} displays the convergence rates of the aforementioned variables. As can be observed, the internal state variables $\bm{\Gamma}^{1}$ and $\bm{Q}^{1}$ exhibit theoretical convergence rates, confirming the analysis made in Section \ref{sec:time-integration}. As for the displacement, velocity, and pressure, Scheme-1 enters into the asymptotical region when the time step size is smaller than $2\times10^{-4}$. An explanation is that the stress enhancement, which affects the temporal convergence rate, can be smaller than $\bm S_{\mathrm{iso} \: n+\frac12}$ in magnitude when the time step size is relatively large. This renders Scheme-1 exhibiting a quadratic convergence rate in the pre-asymptotical region. As for Scheme-2, the three variables converge quadratically, corroborating the numerical analysis made in Section \ref{sec:second-order-update-formula-ISV}.

\begin{figure}
\begin{center}
\begin{tabular}{cc}
\includegraphics[angle=0, trim=80 180 100 100, clip=true, scale = 0.17]{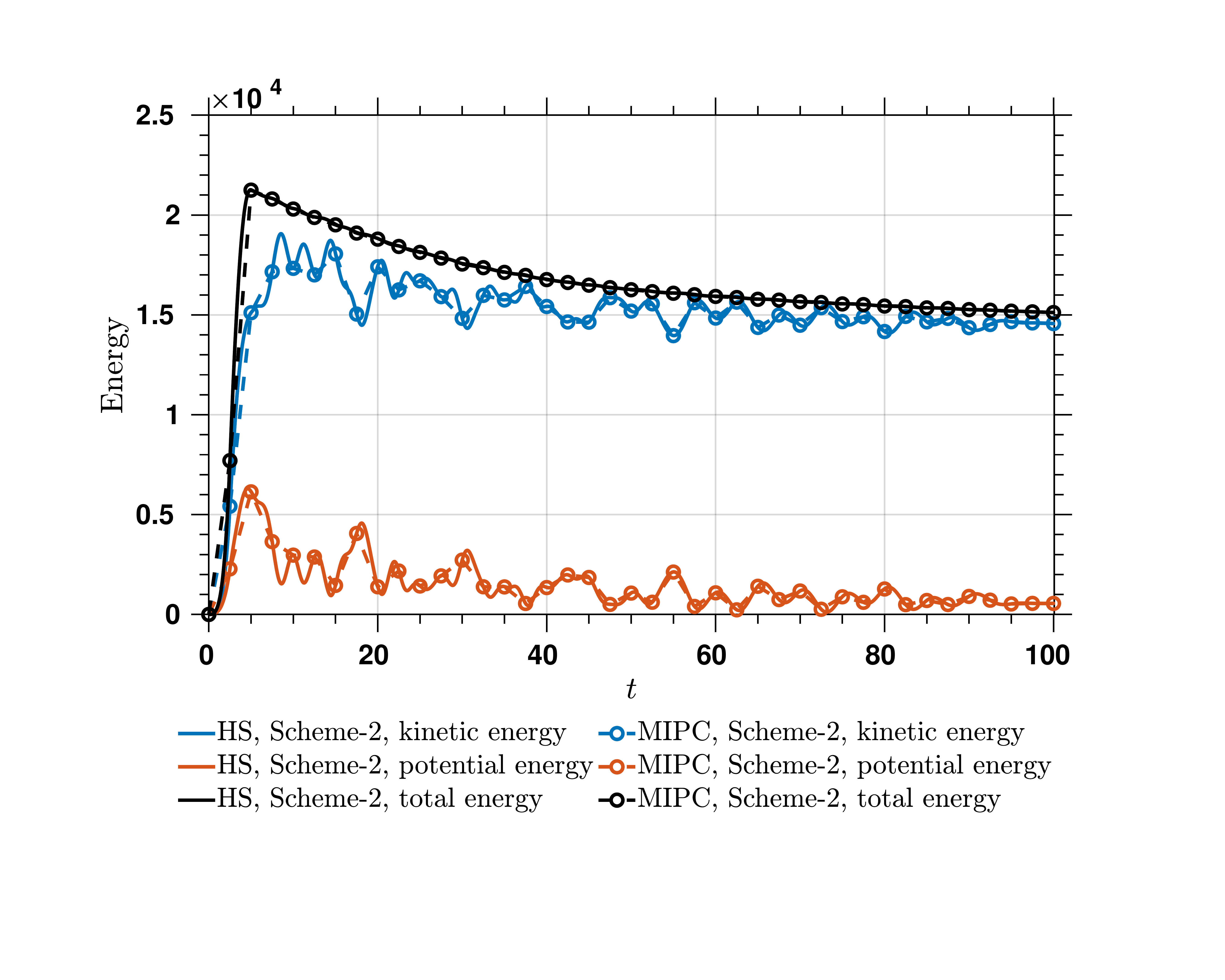} &
\includegraphics[angle=0, trim=80 180 100 100, clip=true, scale = 0.17]{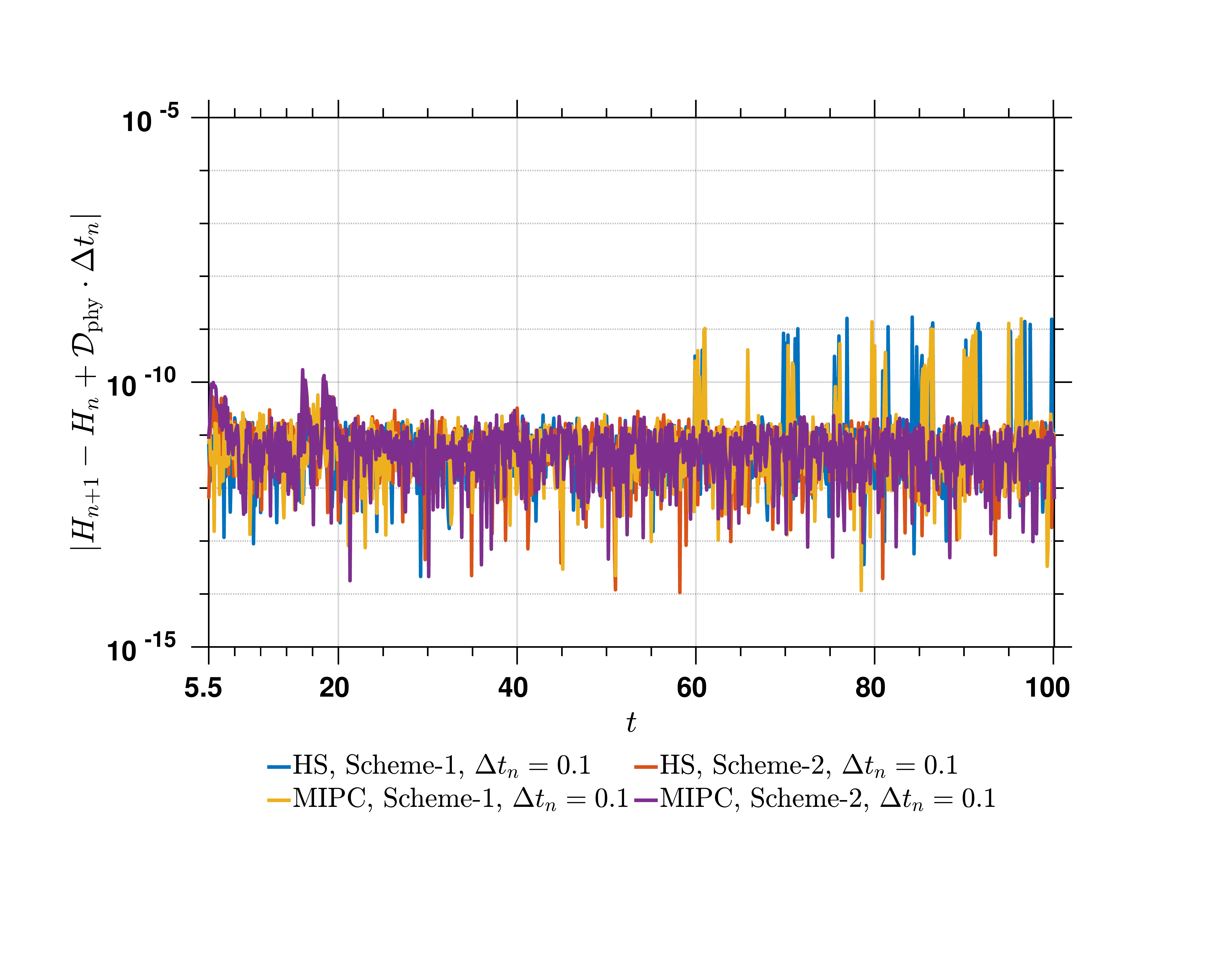} \\
(a) & (b)
\end{tabular}
\caption{ (a) The time histories of the kinetic, potential (i.e., $G^{\infty}_{\mathrm{iso}}(\tilde{\bm C}_n) + \Upsilon^1(\tilde{\bm C}_n, \bm \Gamma^{1}_n)$), and total energies; (b) the absolute errors of $H_{n+1} - H_{n} + \mathcal D_{\mathrm{phy}}\Delta t_n$ for the HS and MIPC models.}
\label{fig:Ldomain_energy_HS_MIPC}
\end{center}
\end{figure}

\begin{figure}
\begin{center}
\begin{tabular}{cc}
\includegraphics[angle=0, trim=80 180 100 100, clip=true, scale = 0.17]{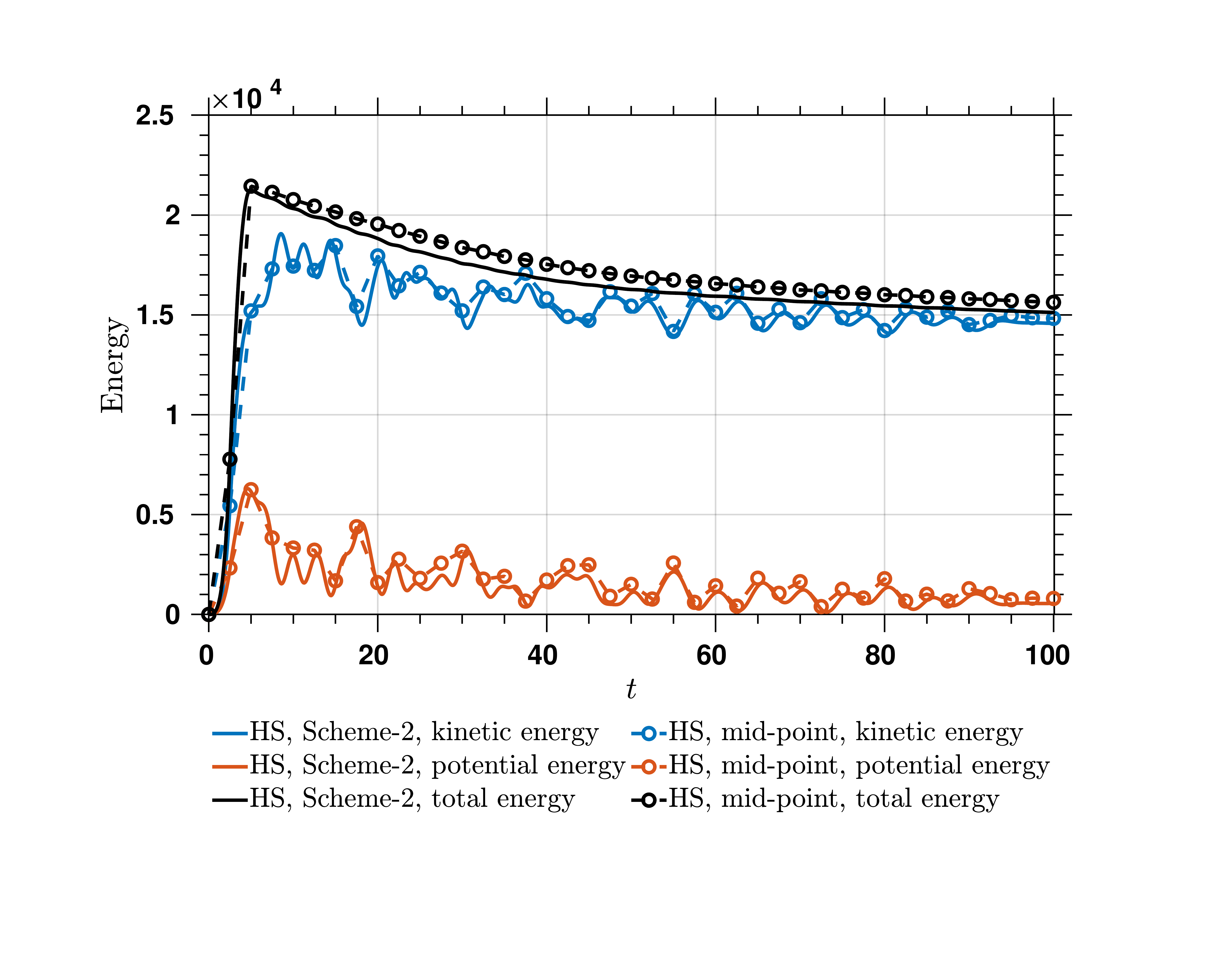} &
\includegraphics[angle=0, trim=80 180 100 100, clip=true, scale = 0.17]{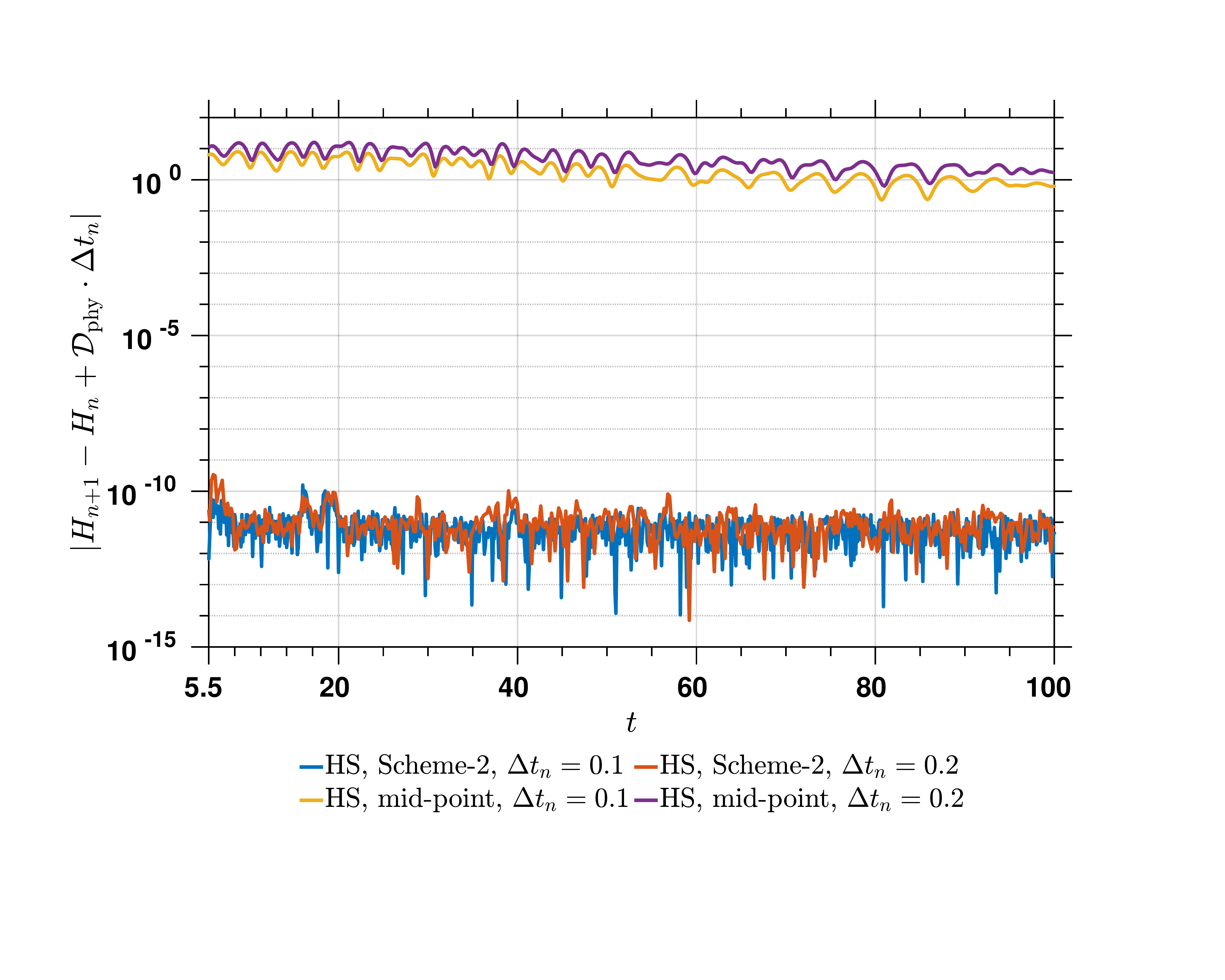} \\
(a) & (b) 
\end{tabular}
\caption{Comparison of the discrete energies calculated by Scheme-2 and the mid-point scheme using the HS model: (a) the time histories of the kinetic, potential (i.e., $G^{\infty}_{\mathrm{iso}}(\tilde{\bm C}_n) + \Upsilon^1(\tilde{\bm C}_n, \bm \Gamma^{1}_n)$, and total energies; (b) the absolute errors of $H_{n+1} - H_{n} + \mathcal D_{\mathrm{phy}}\Delta t_n$.}
\label{fig:Ldomain_energy_HS_MD}
\end{center}
\end{figure}

\begin{figure}
\begin{center}
\begin{tabular}{cc}
\includegraphics[angle=0, trim=80 180 100 100, clip=true, scale = 0.17]{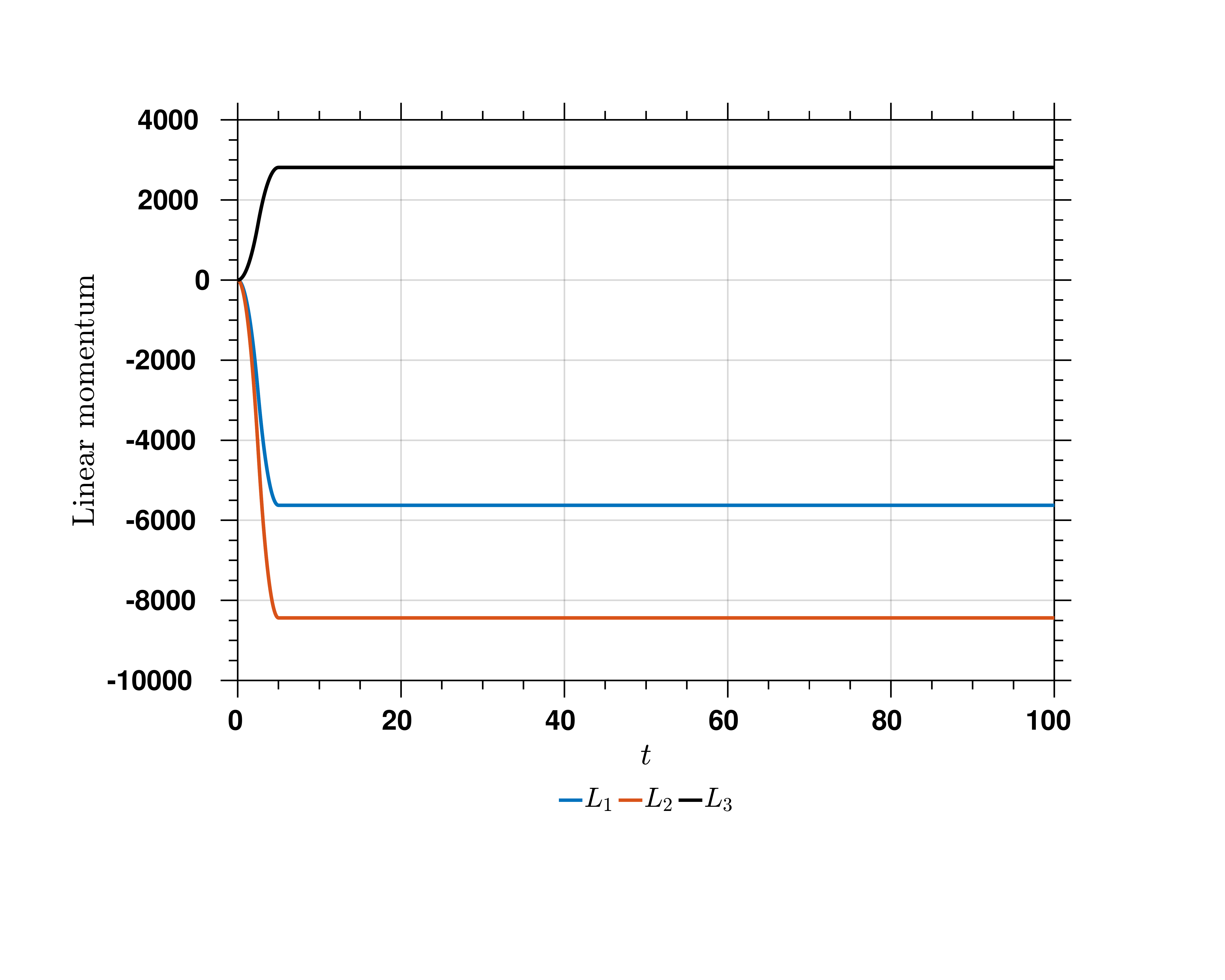} &
\includegraphics[angle=0, trim=80 180 100 100, clip=true, scale = 0.17]{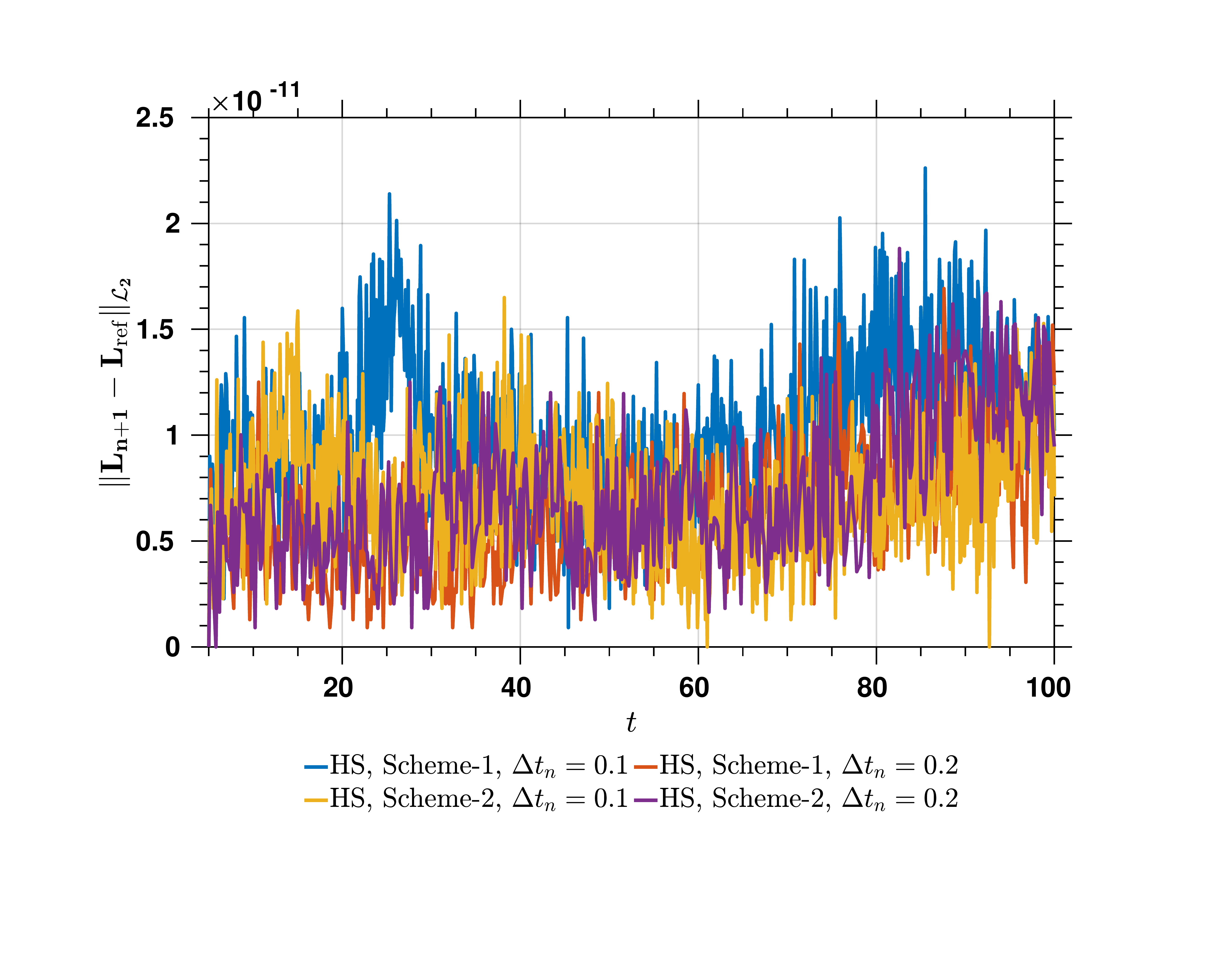}\\
(a) & (b) 
\end{tabular}
\caption{(a) The time histories of the total linear momentum; (b) the absolute errors of the total linear momentum calculated by Scheme-1 and Scheme-2 with $\Delta t_n = 0.1$ and $0.2$, respectively. The symbol $(\bullet)_{\mathrm{ref}}$ represents the solution at $t=5$.}
\label{fig:Ldomain_Lmom}
\end{center}
\end{figure}

\begin{figure}
\begin{center}
\begin{tabular}{cc}
\includegraphics[angle=0, trim=80 180 100 100, clip=true, scale = 0.17]{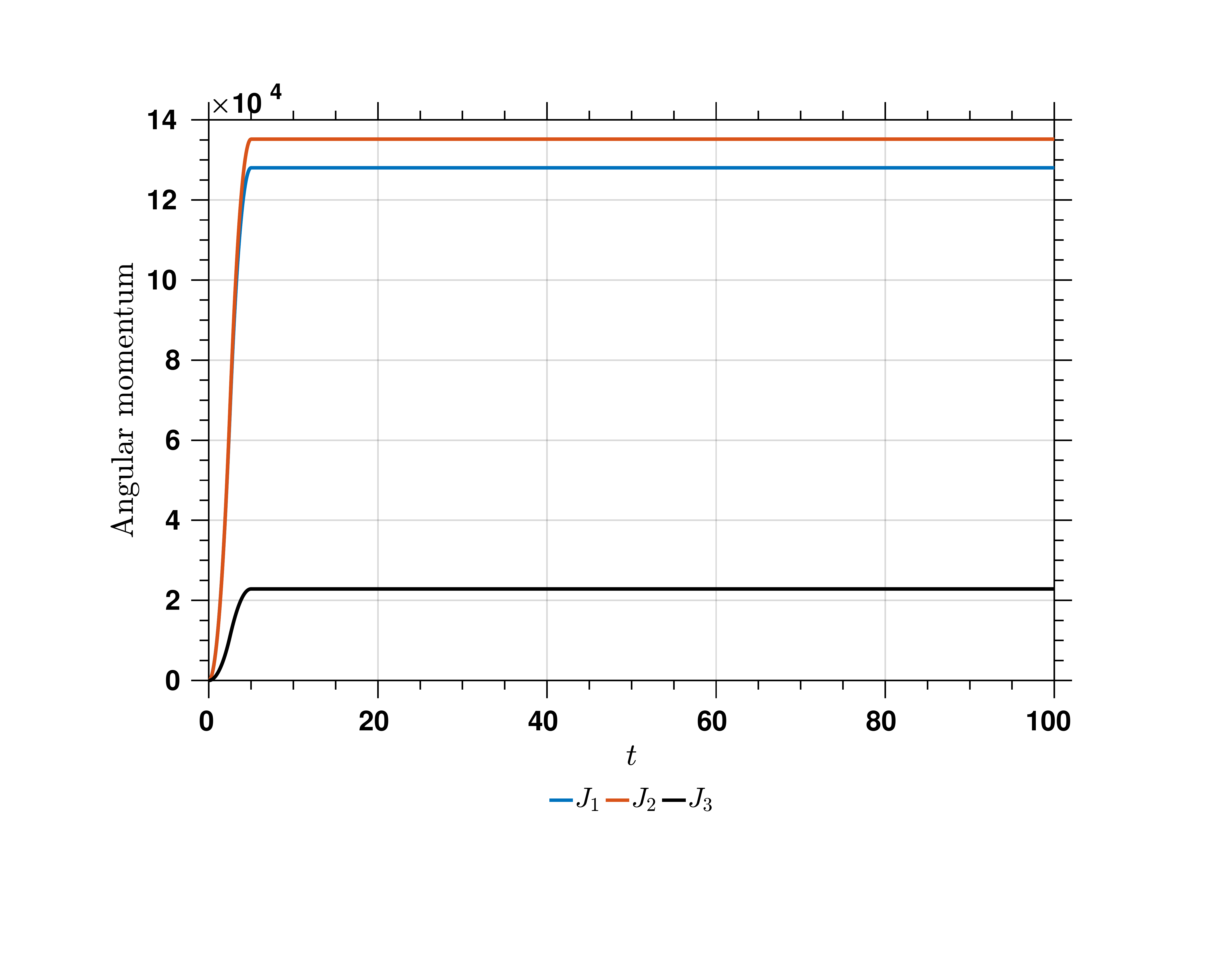} &
\includegraphics[angle=0, trim=80 180 100 100, clip=true, scale = 0.17]{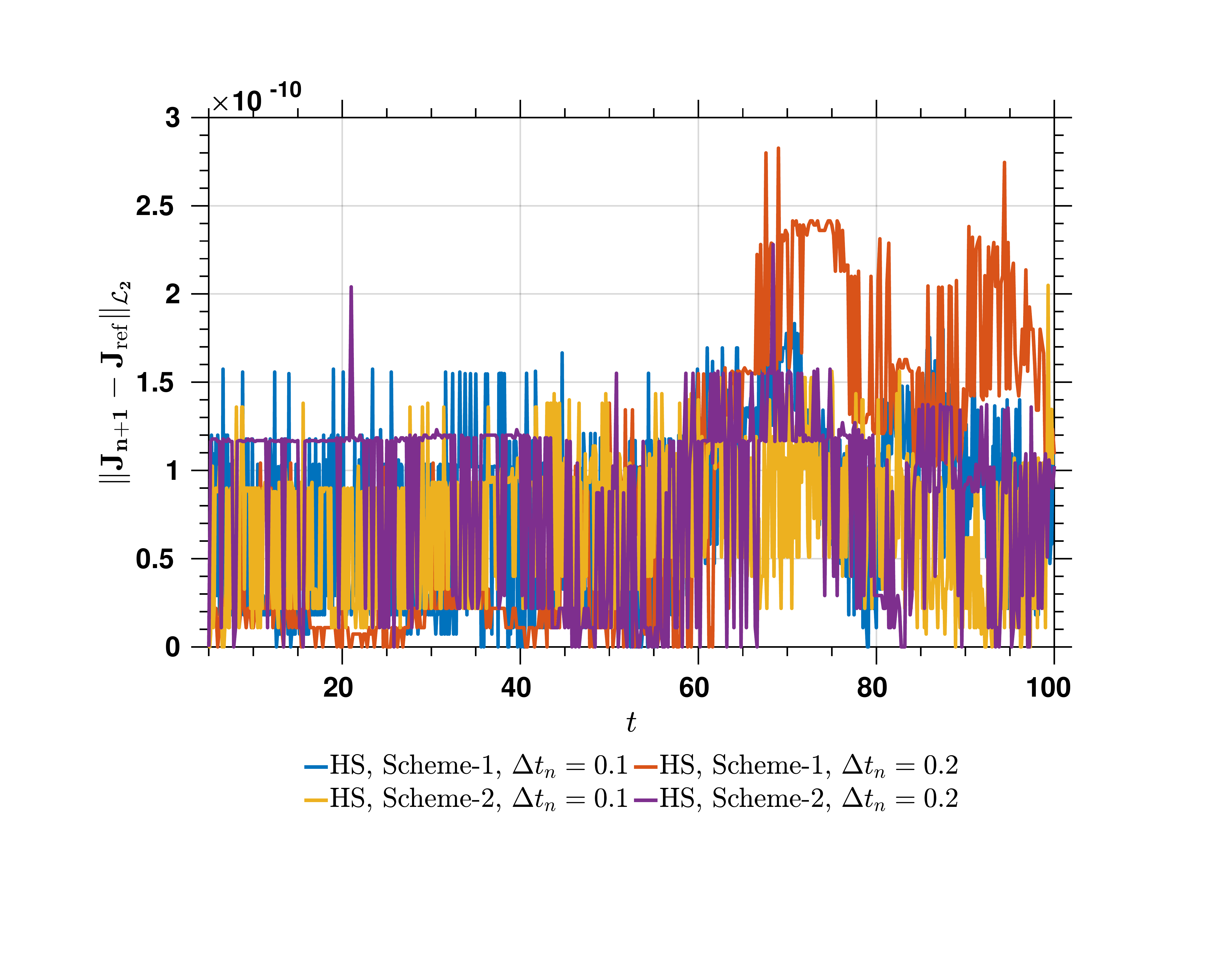}\\
(a) & (b) 
\end{tabular}
\caption{(a) The time histories of the total angular momentum; (b) the absolute errors of the total angular momentum calculated by Scheme-1 and Scheme-2 with $\Delta t_n = 0.1$ and $0.2$, respectively. The symbol $(\bullet)_{\mathrm{ref}}$ represents the solution at $t=5$.}
\label{fig:Ldomain_Jmom}
\end{center}
\end{figure}

\begin{figure}
\begin{center}
\begin{tabular}{cc}
\includegraphics[angle=0, trim=80 85 120 110, clip=true, scale = 0.22]{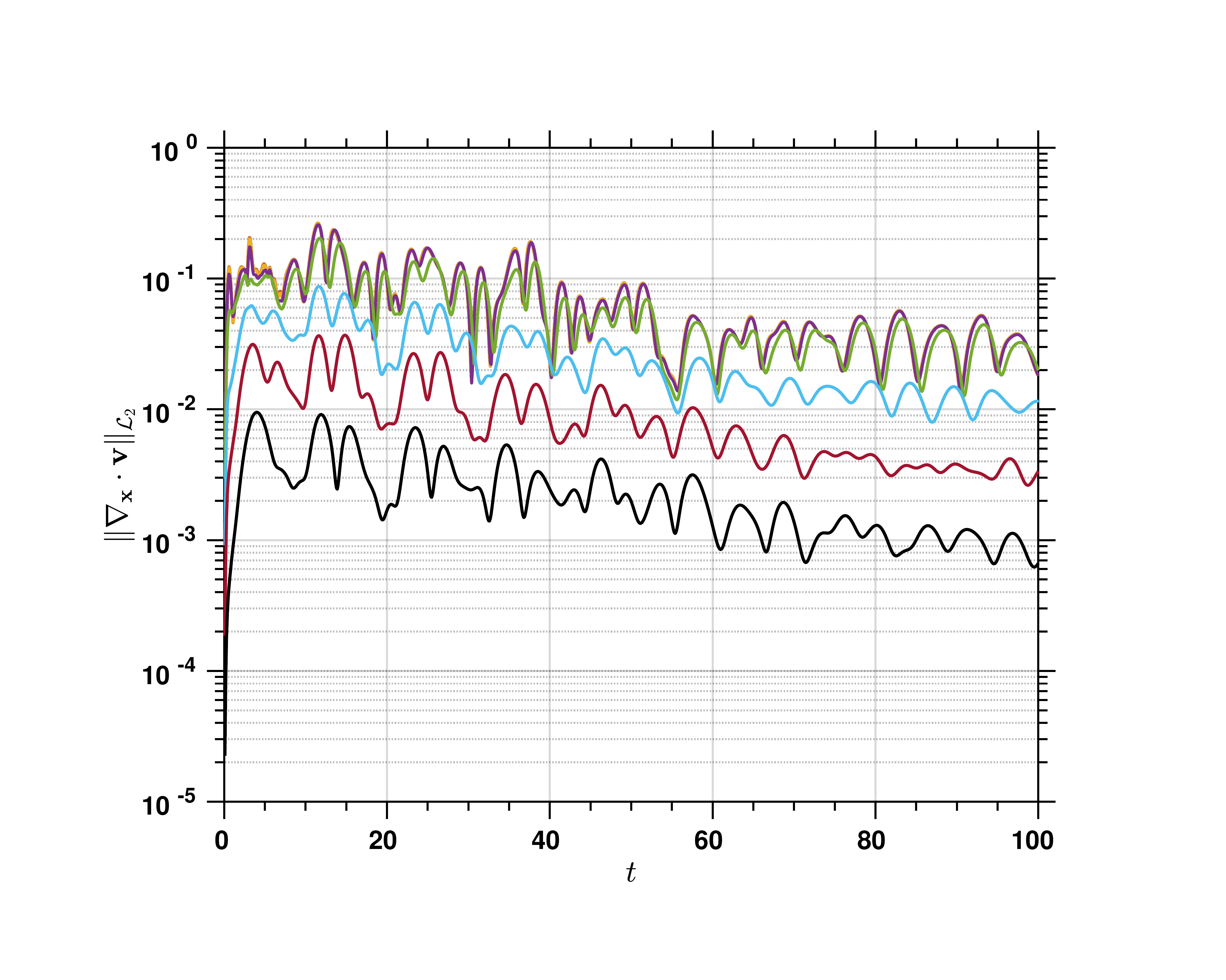} &
\includegraphics[angle=0, trim=80 85 120 110, clip=true, scale = 0.22]{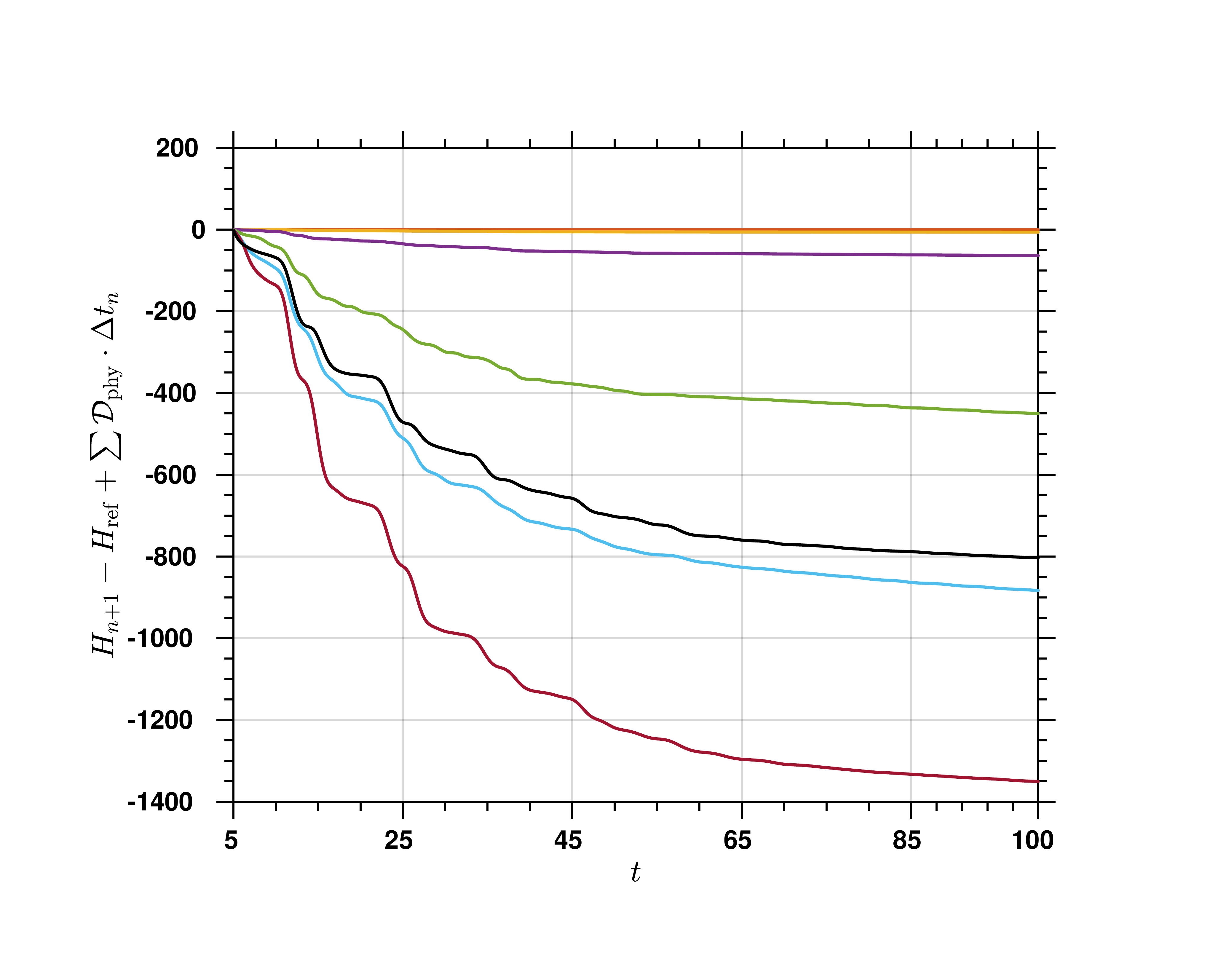} \\
(a) & (b) \\
\includegraphics[angle=0, trim=80 85 120 110, clip=true, scale = 0.22]{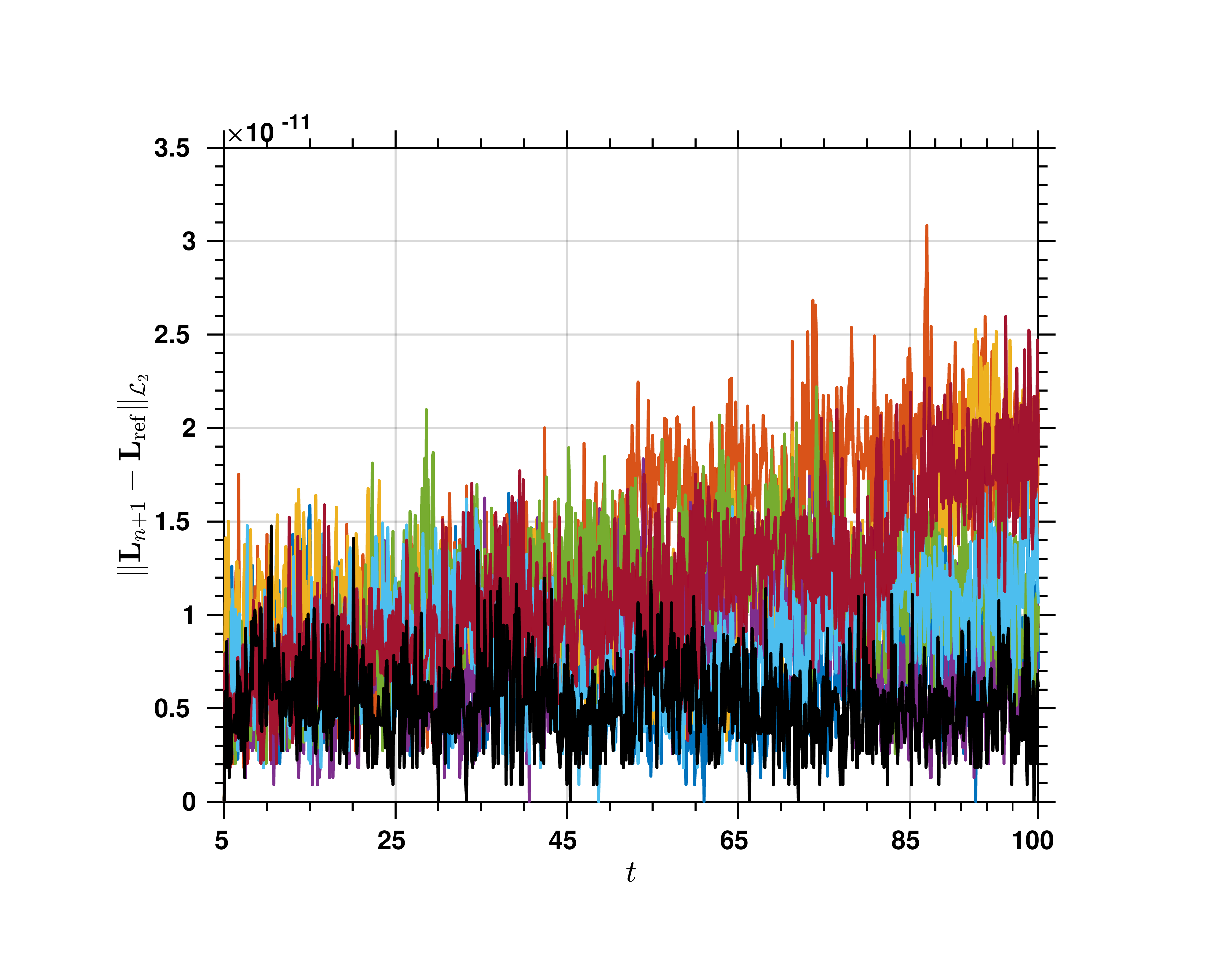} &
\includegraphics[angle=0, trim=80 85 120 110, clip=true, scale = 0.22]{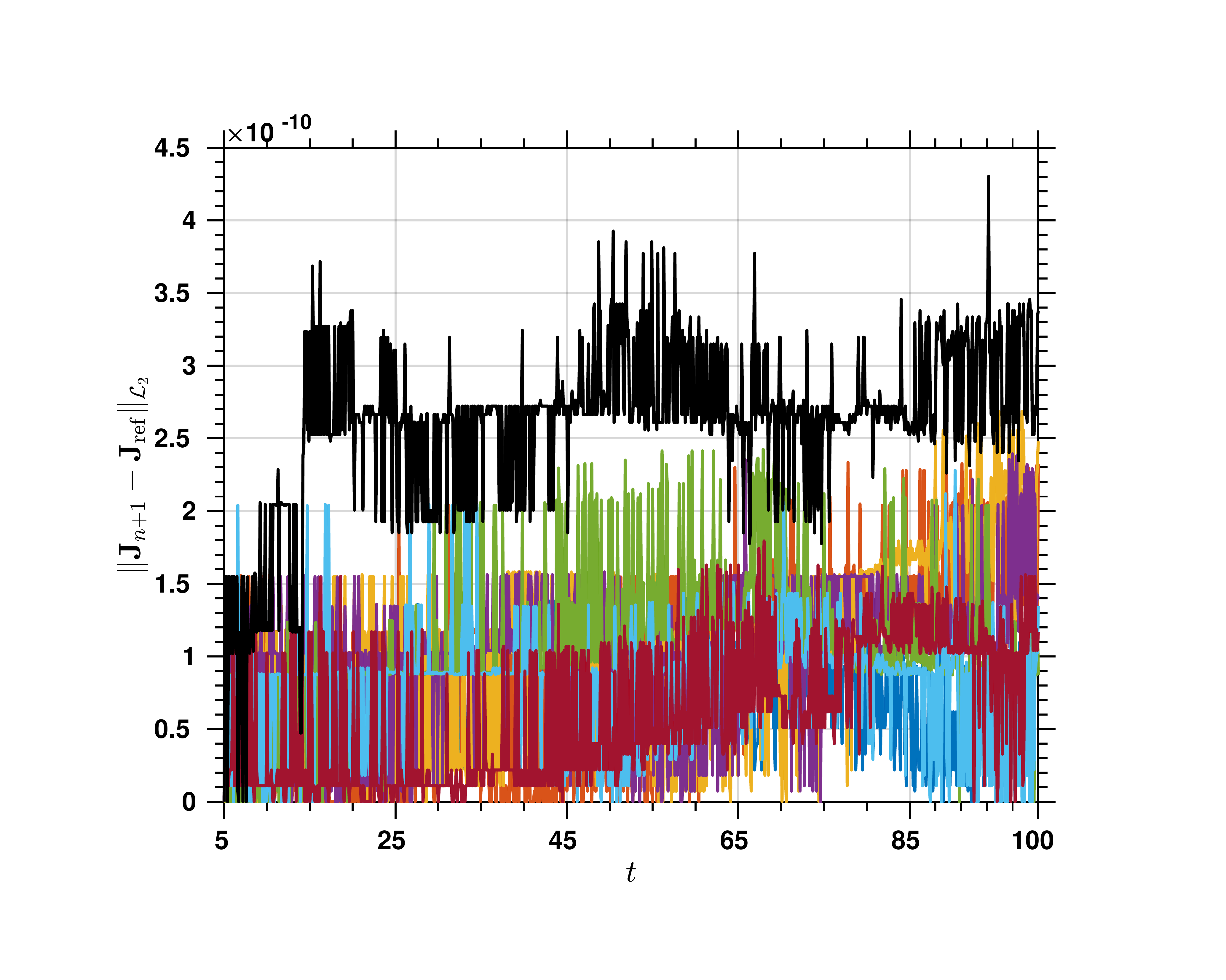} \\
(c) & (d) \\
\multicolumn{2}{c}{ \includegraphics[angle=0, trim=350 160 310 720, clip=true, scale = 0.35]{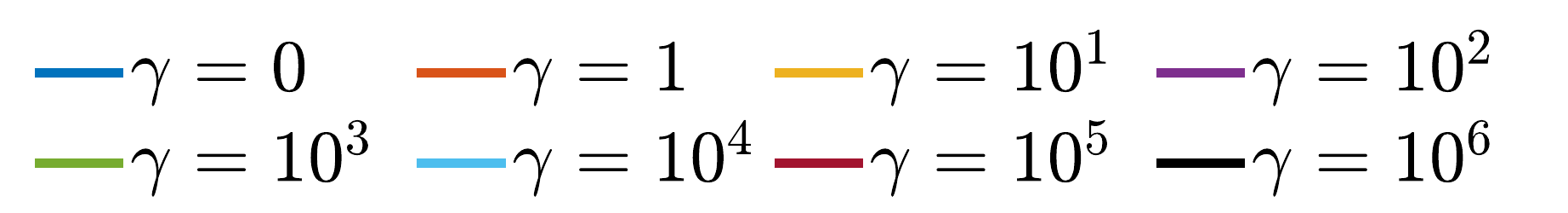} }
\end{tabular}
\caption{The effect of the grad-div stabilization with the HS model and Scheme-2: (a) the time histories of $\| \nabla_{\bm x} \cdot \bm{v} \|_{\mathcal{L}_2}$; (b) the dissipation of the total energy; (c) the absolute errors of the total linear momentum; (d) the absolute errors of the total angular momentum over time. The symbol $(\bullet)_{\mathrm{ref}}$ represents the solution at $t=5$.}
\label{fig:Ldomain_graddiv}
\end{center}
\end{figure}

Next, we examine the discrete energy stability and momentum conservation properties. Figure \ref{fig:Ldomain_energy_HS_MIPC} (a) illustrates the time histories of the kinetic, potential, and total energies for the HS and MIPC models with Scheme-2. It shows that after the initial loading phase, the total energy decreases with time due to the viscous dissipation. The energy histories of the two material models are almost identical in this case. In Figure \ref{fig:Ldomain_energy_HS_MIPC} (b), we examine the absolute errors of $H_{n+1} - H_{n} + \mathcal D_{\mathrm{phy}} \Delta t_n$ over time. It is demonstrated that the dissipation of both schemes matches the dissipation of the continuum system in the sense of \eqref{eq:em-prop-1} with an error at the magnitude of $10^{-10}$, which has the same order of magnitude as the tolerances used in the nonlinear solver. This justifies the energy stability property of both schemes. 

The differences between Scheme-2 and the mid-point scheme are illustrated in Figure \ref{fig:Ldomain_energy_HS_MD}. In Figure \ref{fig:Ldomain_energy_HS_MD} (a), one may observe that the total energy calculated by the mid-point scheme is slightly larger than that of Scheme-2. Figure \ref{fig:Ldomain_energy_HS_MD} (b) depicts the absolute error of $H_{n+1} - H_{n} + \mathcal D_{\mathrm{phy}} \Delta t_n$ using the time step sizes $\Delta t_n = 0.1$ and $\Delta t_n = 0.2$. It can be clearly observed from the figure that the absolute errors given by Scheme-2 are at the magnitude of $10^{-10}$, regardless of the time step size. The errors of the mid-point scheme are significantly larger than those of Scheme-2. In addition, the conservation of the total linear and angular momenta are plotted in Figures \ref{fig:Ldomain_Lmom} and \ref{fig:Ldomain_Jmom}, respectively. It can be gleaned from the two figures that the errors of the momenta are at the magnitude of $10^{-10}$, justifying the momentum conservation property.

To investigate the effects of the grad-div stabilization, the HS model is numerically studied with Scheme-2, and the time step size is fixed to be $\Delta t_n = 0.1$. The parameter $\gamma$ varies from 0 to $10^{6}$. In Figure \ref{fig:Ldomain_graddiv} (a), we plot the $\mathcal{L}_2$-norm of $\nabla_{\bm x} \cdot \bm{v}$ over time. It can be observed that $\| \nabla_{\bm x} \cdot \bm{v} \|_{\mathcal{L}_2}$ decreases with the increase of the parameter $\gamma$, signifying the efficacy of the grad-div stabilization in enhancing the discrete mass conservation. In Figure \ref{fig:Ldomain_graddiv} (b), the time histories of $H_{n+1}-H_{\mathrm{ref}}+\sum\mathcal{D}_{\mathrm{phy}}\cdot \Delta t_n$ are plotted, from which one may evaluate the dissipation induced by the grad-div stabilization. When the value of $\gamma$ increases from 0 to $10^{5}$, the amount of the numerical dissipation gets stronger. Yet, when the value of $\gamma$ increases from $10^{5}$ to $10^{6}$, it is tempting to observe that there is a reduction of the dissipation. Indeed, there is a competing mechanism in the definition of the numerical dissipation $\mathcal{D}_{\mathrm{num}}$. With the increase of $\gamma$, $(\nabla_{\bm x} \cdot \bm v)^2$ approaches zero at a faster rate \cite{Case2011}.  Therefore, we expect that the numerical dissipation $\mathcal D_{\mathrm{num}}$ will asymptotically vanish as the value of $\gamma$ approaches infinity. The effects of the grad-div stabilization on the momentum conservation are illustrated in Figures \ref{fig:Ldomain_graddiv} (c) and (d). The results justify that the conservation of the momenta is not disturbed by the grad-div stabilization.

\begin{table}[htbp]
  \centering 
  \begin{tabular}{ m{.35\textwidth} m{.4\textwidth} }
    \hline
    \begin{minipage}{.35\textwidth}
    \centering
      \includegraphics[width=0.95\linewidth, trim=160 280 170 260, clip]{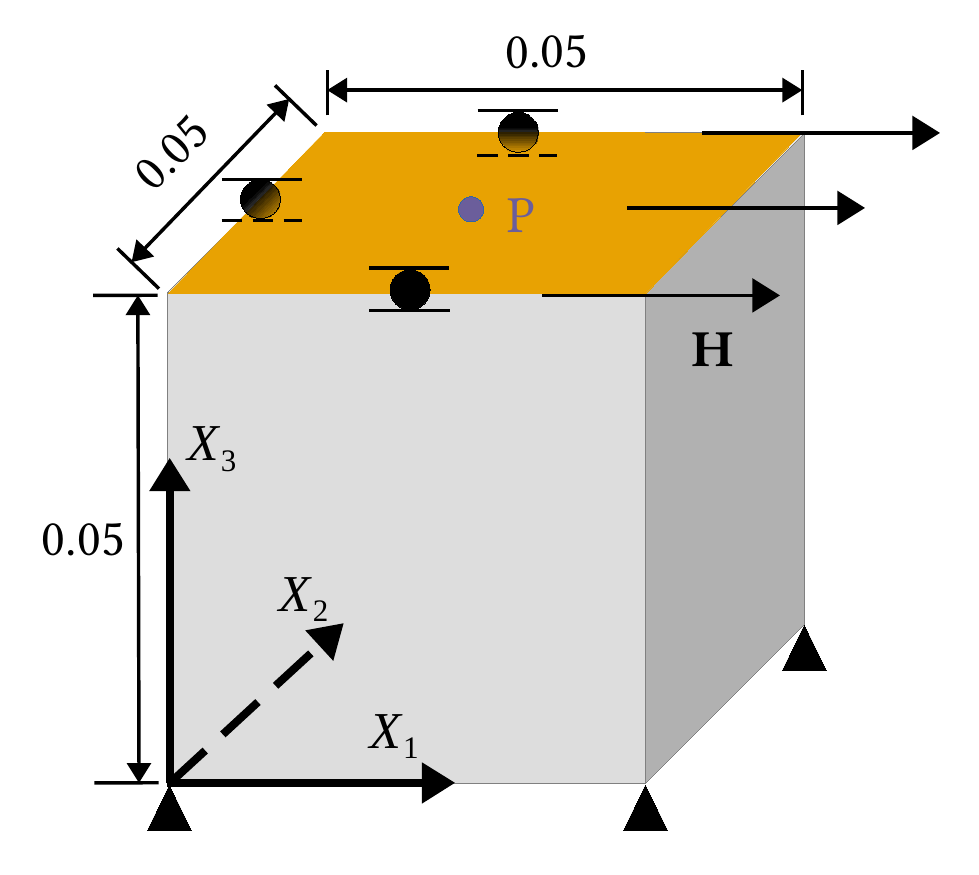}
    \end{minipage}
    &
    \begin{minipage}{.4\textwidth}
      \begin{itemize}
        \item[] Material properties:
        \item[] $G_{\mathrm{iso}}^{\infty} = \frac{c_1}{2}(\tilde{I}_1 -3) + \frac{c_2}{2}(\tilde{I}_2 -3)$
        \item[] $\rho_0 = 1.0\times 10^3$
        \item[] $E = 625.72\times10^{3}$
        \item[] $c_1=c_2=E/6$
        \item[] $m=1$
        \item[] $\mu^1=536.224\times10^{3}$, $\eta^1=0.5 \times\mu^1$       
      \end{itemize}
    \end{minipage}   
    \\
    \hline
  \end{tabular}
  \caption{The problem definition of the shear test with sinusoidal loading: on the bottom surface, $\bm G = \bm 0$; on the top surface, $G_{2} = G_{3} = 0$, and the surface traction $\bm H$ is applied. The coordinate of the point P is $(0.025, 0.025, 0.05)$, at which we monitor the displacement and traction force over time.}
\label{table:cubic_block}
\end{table}

\begin{figure}
\begin{center}
\begin{tabular}{cc}
\includegraphics[angle=0, trim=80 50 80 100, clip=true, scale = 0.22]{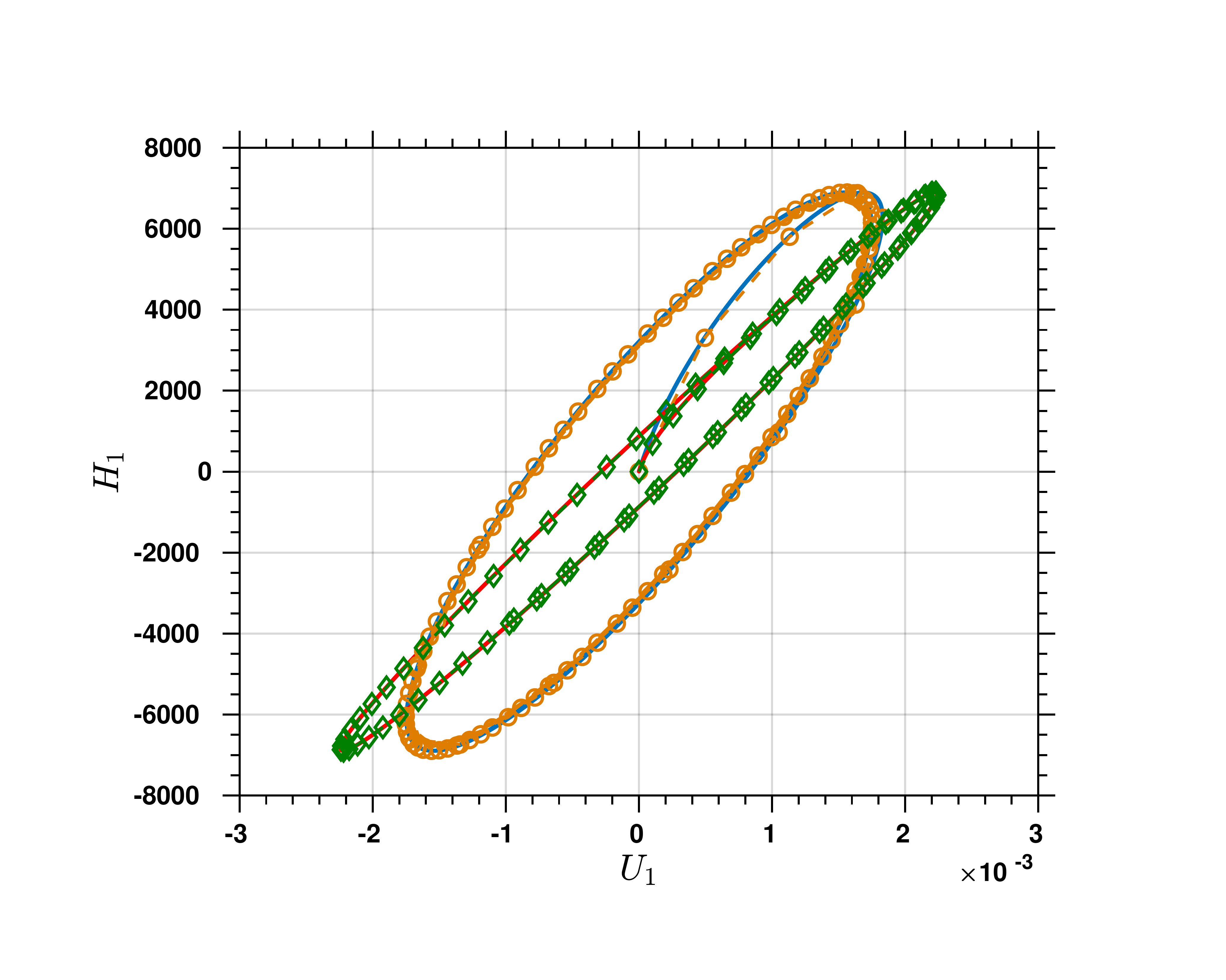} &
\includegraphics[angle=0, trim=80 50 80 100, clip=true, scale = 0.22]{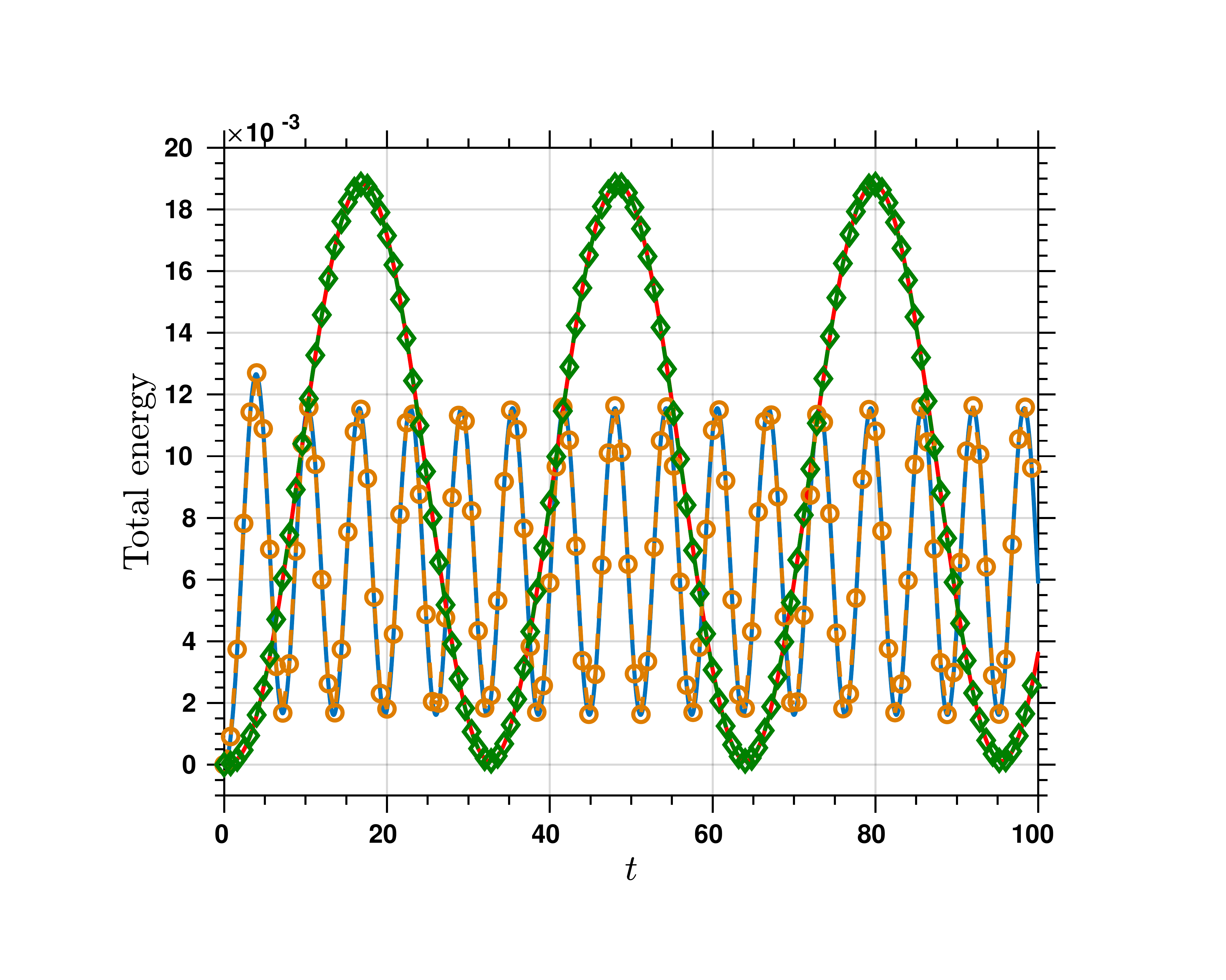} \\
(a) & (b) \\
\multicolumn{2}{c}{\includegraphics[angle=0, trim=0 50 50 120, clip=true, scale = 0.25]{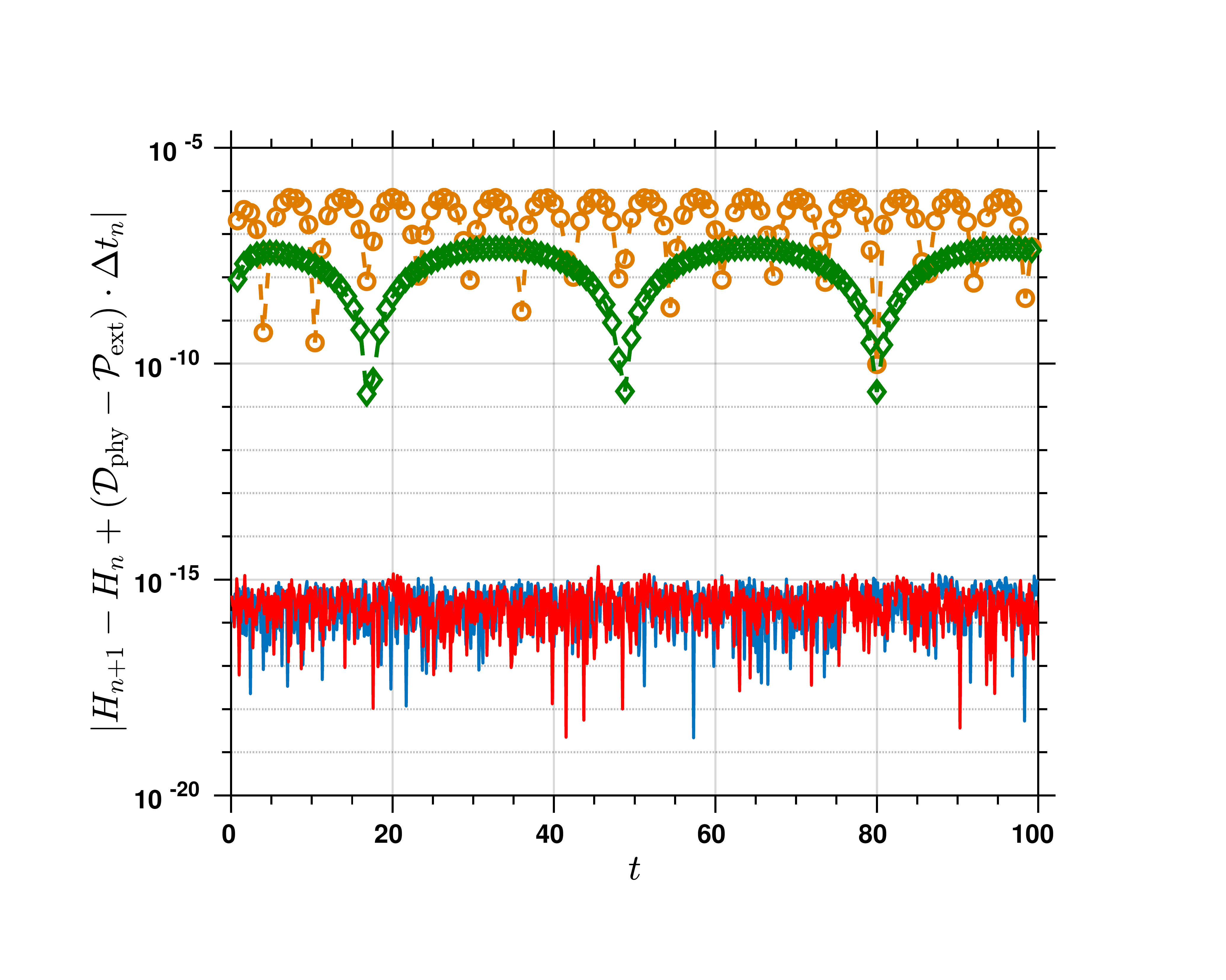}} \\
\multicolumn{2}{c}{ (c) } \\
\multicolumn{2}{c}{ \includegraphics[angle=0, trim=0 200 0 650, clip=true, scale = 0.23]{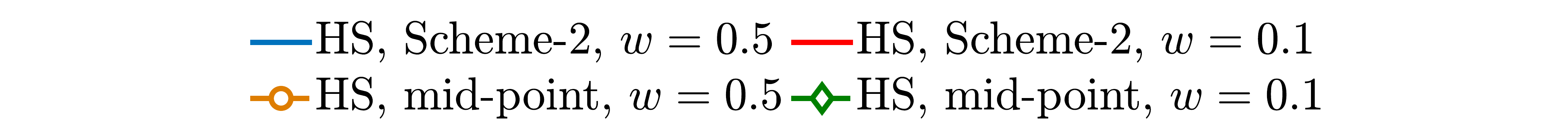} }
\end{tabular}
\caption{(a) The hysteresis loops of point P; (b) the time histories of the total energy; (c) the absolute errors of $H_{n+1} - H_{n} + \left( \mathcal D_{\mathrm{phy}} - \mathcal P_{\mathrm{ext}} \right)\Delta t_n$ given by Scheme-2 and the mid-point scheme.}
\label{fig:results_EM2_GA}
\end{center}
\end{figure}

\subsection{Shear test with sinusoidal loading}
In this example, we consider a three-dimensional cubic block with a sinusoidal surface load applied on its top. We use the Mooney-Rivlin model and the HS model to describe the equilibrium and viscous non-equilibrium parts of the material behavior, respectively. Parameters for the problem definition are summarized in Table \ref{table:cubic_block}. The bottom surface $X_{3}=0$ is fully clamped. On the top surface, roller boundary conditions are applied such that the displacement in the $X_2$ and $X_3$ directions are disallowed; a sinusoidal traction is applied on the top surface, taking the form $\bm{H} = [6.895\times10^{3} \mathrm{sin}(wt), 0, 0]^T$. On the rest four surfaces, traction-free boundary conditions are applied. In this study, the rate effects are investigated with two frequencies, that is $w=0.1$ and $0.5$, respectively. The problem is integrated in time up to $T=100$, with a uniform time step size $\Delta t_{n} =0.01$. The spatial mesh consists of eight elements with $\mathsf p=1$. 

Figure \ref{fig:results_EM2_GA} (a) illustrates the applied traction against the displacement of the centroid of the top surface (i.e., point P in Table \ref{table:cubic_block}). The curves are known as the hysteresis loops, and the loops calculated by Scheme-2 and the mid-point scheme are almost indistinguishable. In Figure \ref{fig:results_EM2_GA} (b), the total energy curves exhibit a sinusoidal form for both frequencies. In the low-frequency case, the amplitude of the energy oscillation is larger. The difference between the results of Scheme-2 and the mid-point scheme is negligible from the energy curves in Figure \ref{fig:results_EM2_GA} (b). Since the boundary data $\bm{G}$ is time-independent, the evolution of the total energy follows the estimate
\begin{align*}
\frac{1}{\Delta t_n} \left( H_{n+1} - H_{n} \right) = \mathcal P_{\mathrm{ext}} - \mathcal D_{\mathrm{phy}},
\end{align*}
according to Lemma \ref{lemma:ts-energy-dissipation}. In Figure \ref{fig:results_EM2_GA} (c), the absolute errors of $H_{n+1} - H_{n} + \left(\mathcal D_{\mathrm{phy}} -  \mathcal P_{\mathrm{ext}}  \right)\Delta t_n$ are monitored over time. The errors of Scheme-2 are of the order $10^{-15}$ and are independent of the frequency values. In other words, the energy in the solutions strictly matches the physical dissipation and the input of the external loading. The errors calculated by the mid-point scheme, however, are several orders of magnitude larger than those of Scheme-2, signifying an alerting physical inconsistency issue in the results of the mid-point scheme.

\section{Conclusion}
\label{sec:conclusion}
In this work, we continued our investigation of the finite deformation linear viscoelastic models. To construct energy-momentum consistent schemes, the notion of \textit{directionality property} is extended in two aspects. First, the dissipation effect is taken into account for the design of the algorithmic second Piola-Kirchhoff stress. This leads to an integration scheme for the balance equations. Second, the directionality property is demanded for the algorithmic approximation of the stress-like conjugate variables $\bm Q^{\alpha}$, which results in an integration rule for the internal state variables (i.e., the inelastic constitutive equations). Different from the widely adopted approach \cite{Simo1987,Simo2006,Holzapfel2000} that relies on a recurrence formula for the stress-like tensor $\bm Q^{\alpha}$, the integration of the constitutive equations here is in the strain-driven format. The combined approach leads to a suite of numerical strategies that provides an energy consistent, momentum conservative, and second-order accurate scheme for finite deformation linear viscoelastic models. Additionally, the grad-div stabilization technique is explored, and it is demonstrated to be rather effective in enhancing the discrete mass conservation property for the mixed elements considered here. In our future study, a completely nonlinear theory will be pursued.

\section*{Acknowledgments}
This work is supported by the National Natural Science Foundation of China [Grant Numbers 12172160], Shenzhen Science and Technology Program [Grant Number JCYJ20220818100600002], Southern University of Science and Technology [Grant Number Y01326127], and the Department of Science and Technology of Guangdong Province [Grant Numbers  2020B1212030001, 2021QN020642]. Computational resources are provided by the Center for Computational Science and Engineering at the Southern University of Science and Technology.

\appendix
\section{ The elasticity tensor for the algorithmic stress $\bm S_{\mathrm{iso} \: \mathrm{alg1}}$ }
\label{sec:appendix-A}
We provide an explicit expression for the algorithmic elasticity tensor associated with the algorithmic stress given in Section \ref{sec:first-order-update-formula-ISV}. We recall that the formula of the algorithmic stress \eqref{eq:constitutive-integration-Salg-1}-\eqref{eq:constitutive-integration-Salg-1-detailed} can be expressed as
\begin{align}
\label{eq:A0}
\bm S_{\mathrm{iso} \: \mathrm{alg1}} = \bm S_{\mathrm{iso} \: n+\frac12} + \bm S_{\mathrm{iso} \:\mathrm{enh1}}, \quad
\bm S_{\mathrm{iso} \: \mathrm{enh1}} := \frac{ \Delta G_{\mathrm{iso}}^{n+1} - \bm S_{\mathrm{iso} \: n+\frac12} : \bm Z_n }{ |\bm Z_n|} \frac{\bm Z_n}{ |\bm Z_n| },
\end{align}
with the notation
\begin{align}
\label{eq:A1}
&\Delta G_{\mathrm{iso}}^{n+1} := G_{\mathrm{iso}} \left( \tilde{\bm C}_{n+1}, \bm{\Gamma}_{n+1}^{1}, \cdots, \bm{\Gamma}_{n+1}^{m} \right) - G_{\mathrm{iso}} \left( \tilde{\bm C}_{n}, \bm{\Gamma}_{n+1}^{1}, \cdots, \bm{\Gamma}_{n+1}^{m} \right).
\end{align}
The algorithmic elasticity tensor $\mathbb{C}_{\mathrm{iso} \: \mathrm{alg1}}$ is defined as 
\begin{align*}
\mathbb{C}_{\mathrm{iso} \: \mathrm{alg1}} := 2\frac{\partial \bm S_{\mathrm{iso} \: \mathrm{alg1}}}{\partial \bm{C}_{n+1}}. 
\end{align*}
In view of the structure of $\bm S_{\mathrm{iso} \: \mathrm{alg1}}$, the elasticity tensor can be  re-written in the following form,
\begin{align}
\label{eq:A2}
\mathbb{C}_{\mathrm{iso} \: \mathrm{alg1}} = \mathbb{C}_{ \mathrm{iso} \: n+1 } + \mathbb{C}_{\mathrm{iso} \: \mathrm{enh1}}, \quad \mbox{with} \quad \mathbb{C}_{ \mathrm{iso} \: n+1 } := 2 \frac{\partial \bm S_{\mathrm{iso} \: n+\frac12} }{ \partial \bm C_{n+1}} \quad \mbox{and} \quad \mathbb{C}_{\mathrm{iso} \: \mathrm{enh1}} := 2 \frac{\partial \bm S_{\mathrm{iso} \: \mathrm{enh1}}}{ \partial \bm C_{n+1} }. 
\end{align}
Using the chain rule, we obtain
\begin{align}
\label{eq:2-1}
\mathbb{C}_{ \mathrm{iso} \: n+1 } = 2 \frac{\partial \bm S_{\mathrm{iso} \: n+\frac12} }{ \partial \bm C_{n+\frac12}} : \frac{\partial \bm C_{n+\frac12}}{\partial \bm C_{n+1}} = \frac12 \mathbb{C}_{\mathrm{iso}\:n+\frac12},
\end{align}
where we have introduced
\begin{align}
\mathbb C_{\mathrm{iso}\:n+\frac12} := 2 \frac{\partial \bm S_{\mathrm{iso} \: n+\frac12} }{ \partial \bm C_{n+\frac12}}.
\end{align}
The definition \eqref{eq:A0}$_2$ leads to
\begin{align}
\label{eq:A3}
\mathbb{C}_{\mathrm{iso}\:\mathrm{enh1}} = 2\bm Z_n \otimes \frac{\partial}{ \partial \bm C_{n+1} } \left( \frac{ \Delta G_{\mathrm{iso}}^{n+1} - \bm S_{\mathrm{iso} \: n+\frac12} : \bm Z_n }{ |\bm Z_n|^2 } \right) + \frac{ \Delta G_{\mathrm{iso}}^{n+1} - \bm S_{\mathrm{iso} \: n+\frac12} : \bm Z_n }{ |\bm Z_n|^2 } \mathbb{I}.
\end{align}
The first term on the right-hand side in \eqref{eq:A3} can be explicitly derived as follows,
\begin{align}
\label{eq:A4}
& 2\bm Z_n \otimes \frac{\partial}{ \partial \bm C_{n+1} } \left( \frac{ \Delta G_{\mathrm{iso}}^{n+1} - \bm S_{\mathrm{iso} \: n+\frac12} : \bm Z_n }{ |\bm Z_n|^2 } \right) \nonumber \\
&= 2 \bm Z_n \otimes \left( \frac{ \partial \Delta G_{\mathrm{iso}}^{n+1}/ \partial \bm{C}_{n+1} - \frac{1}{4} \mathbb{C}_{ \mathrm{iso}\:n+\frac12 }:\bm Z_n - \frac{1}{2} \bm S_{ \mathrm{iso}\:n+\frac12 }}{ |\bm Z_n|^2 }  -2\frac{ \Delta G_{\mathrm{iso}}^{n+1} - \bm S_{\mathrm{iso}\:n+\frac12 } : \bm Z_n  }{ |\bm Z_n|^3 } \frac{ \partial |\bm Z_n| }{ \partial \bm C_{n+1} } \right) \nonumber \\
&= \frac{\bm Z_n}{|\bm Z_n |^2} \otimes \left( 2\frac{\partial \Delta G_{\mathrm{iso}}^{n+1}}{\partial \bm{C}_{n+1}} -\frac{1}{2}\mathbb{C}_{\mathrm{iso}\:n+\frac12}: \bm Z_n - \bm S_{\mathrm{iso}\:n+\frac12} \right) -2\frac{\Delta G_{\mathrm{iso}}^{n+1}- \bm S_{\mathrm{iso}\:n+\frac12 }:\bm Z_n}{|\bm Z_n|^4} \bm Z_n \otimes \bm Z_n,
\end{align}
in which we have used the fact that
\begin{align*}
\frac{ \partial |\bm Z_n| }{ \partial \bm C_{n+1} } &= \frac{ \partial \left(\bm Z_n : \bm Z_n \right)^{\frac{1}{2}} }{ \partial \bm C_{n+1} } = \frac{1}{2}( \bm Z_n : \bm Z_n )^{-\frac{1}{2}} \frac{\partial \left(\bm Z_n : \bm Z_n \right)}{\partial \bm C_{n+1}} = \frac{1}{2|\bm Z_n|} \left(\frac{\partial \bm Z_n}{\partial \bm C_{n+1}} : \bm Z_n + \bm Z_n : \frac{\partial \bm Z_n}{\partial \bm C_{n+1}} \right) \nonumber \\
&=\frac{1}{2|\bm Z_n|} \left( \frac{1}{2}\mathbb{I}:\bm Z_n + \bm Z_n:\frac{1}{2}\mathbb{I} \right) = \frac{\bm Z_n}{2|\bm Z_n|}.
\end{align*}
In view of the definition of $\Delta G_{\mathrm{iso}}^{n+1}$ \eqref{eq:A1} and $G_{\mathrm{iso}}$ \eqref{eq:gibbs-energy}$_3$, the term $\partial \Delta G_{\mathrm{iso}}^{n+1} / \partial \bm{C}_{n+1}$ can be written as
\begin{align*}
\frac{\partial \Delta G_{\mathrm{iso}}^{n+1}}{ \partial \bm C_{n+1} } = \frac{\partial G_{\mathrm{iso}} \left( \tilde{\bm C}_{n+1}, \bm{\Gamma}_{n+1}^{1}, \cdots, \bm{\Gamma}_{n+1}^{m} \right)}{ \partial \bm C_{n+1} } - \sum_{\alpha=1}^{m} \frac{\partial \Upsilon^{\alpha}\left( \tilde{\bm C}_{n}, \bm{\Gamma}_{n+1}^{\alpha} \right)}{ \partial \bm C_{n+1} } = \frac12 \bm S_{\mathrm{iso}\:n+1} - \sum_{\alpha=1}^{m} \frac{\partial \Upsilon^{\alpha}\left( \tilde{\bm C}_{n}, \bm{\Gamma}_{n+1}^{\alpha} \right)}{ \partial \bm C_{n+1} }.
\end{align*}
Recalling the evolution equations \eqref{eq:constitutive-integration-Gamma-1}, the tensors $\bm{\Gamma}_{n+1}^{\alpha}$ are independent of $\bm{C}_{n+1}$. Hence, the last term of the above relation vanishes, and we get
\begin{align}
\label{eq:A5}
\frac{\partial \Delta G_{\mathrm{iso}}^{n+1}}{ \partial \bm C_{n+1} } = \frac12 \bm S_{\mathrm{iso}\:n+1}.
\end{align}
Combining \eqref{eq:A2}-\eqref{eq:A5}, we obtain the final explicit expression for the elasticity tensor as
\begin{align}
\label{eq:A6}
\mathbb{C}_{\mathrm{iso}\:\mathrm{alg1}}
=& \frac{1}{2}\mathbb{C}_{ \mathrm{iso}\:n+\frac12 } + \frac{ \Delta G_{\mathrm{iso}}^{n+1} - \bm S_{\mathrm{iso} \: n+\frac12} : \bm Z_n }{ |\bm Z_n|^2 } \left(\mathbb{I} -2\frac{\bm Z_n \otimes \bm Z_n}{|\bm Z_n|^2} \right) \nonumber \\
& + \frac{\bm Z_n}{|\bm Z_n |^2} \otimes \left( \bm S_{\mathrm{iso}\:n+1} -\frac{1}{2}\mathbb{C}_{\mathrm{iso}\:n+\frac12}: \bm Z_n - \bm S_{\mathrm{iso}\:n+\frac12} \right).
\end{align}

\section{ The elasticity tensor for the algorithmic stress $\bm S_{\mathrm{iso} \: \mathrm{alg2}}$ }
\label{sec:appendix-B}
Here, we derive an explicit expression for the algorithmic elasticity tensor associated with the algorithmic stress given in Section \ref{sec:second-order-update-formula-ISV}. We recall that the algorithmic stress given in \eqref{eq:def-S-alg2} and \eqref{eq:def-S-enh2} can be expressed as
\begin{align*}
\bm S_{\mathrm{iso}\:\mathrm{alg2}} = \bm S_{\mathrm{iso} \: n+\frac12} + \bm S_{\mathrm{iso}\:\mathrm{enh2}}, \quad \bm S_{\mathrm{iso}\:\mathrm{enh2}} := \frac{ \Delta G_{\mathrm{iso}}^{n+\frac12} - \bm S_{\mathrm{iso} \: n+\frac12} : \bm Z_n }{ |\bm Z_n|} \frac{\bm Z_n}{ |\bm Z_n| },
\end{align*}
with the notation
\begin{align}
\label{eq:B1}
&\Delta G_{\mathrm{iso}}^{n+\frac12} := G_{\mathrm{iso}} \left( \tilde{\bm C}_{n+1}, \bm{\Gamma}_{n+\frac12}^{1}, \cdots, \bm{\Gamma}_{n+\frac12}^{m} \right) -G_{\mathrm{iso}} \left( \tilde{\bm C}_{n}, \bm{\Gamma}_{n+\frac12}^{1}, \cdots, \bm{\Gamma}_{n+\frac12}^{m} \right).
\end{align}
The algorithmic elasticity tensor $\mathbb{C}_{\mathrm{iso} \: \mathrm{alg2}}$ can be consequently expressed as
\begin{align}
\label{eq:B2}
\mathbb{C}_{\mathrm{iso}\:\mathrm{alg2}} = \mathbb{C}_{ \mathrm{iso}\:n+1 } + \mathbb{C}_{\mathrm{iso}\:\mathrm{enh2}}, \quad \mbox{with} \quad \mathbb{C}_{\mathrm{iso} \: \mathrm{enh2}} := 2 \frac{\partial \bm S_{\mathrm{iso} \: \mathrm{enh2}}}{ \partial \bm C_{n+1} }.
\end{align}
The expression of $\mathbb{C}_{ \mathrm{iso} \: n+1}$ has been given in \eqref{eq:2-1}. As for $\mathbb{C}_{\mathrm{iso}\:\mathrm{enh2}}$, we have 
\begin{align}
\label{eq:B3}
\mathbb{C}_{\mathrm{iso}\:\mathrm{enh2}} = 2\bm Z_n \otimes \frac{\partial}{ \partial \bm C_{n+1} } \left( \frac{ \Delta G_{\mathrm{iso}}^{n+\frac12} - \bm S_{\mathrm{iso} \: n+\frac12} : \bm Z_n }{ |\bm Z_n|^2 } \right) + \frac{ \Delta G_{\mathrm{iso}}^{n+\frac12} - \bm S_{\mathrm{iso} \: n+\frac12} : \bm Z_n }{ |\bm Z_n|^2 } \mathbb{I}
\end{align}
and 
\begin{align}
\label{eq:B4}
& 2\bm Z_n \otimes \frac{\partial}{ \partial \bm C_{n+1} } \left( \frac{ \Delta G_{\mathrm{iso}}^{n+\frac12} - \bm S_{\mathrm{iso} \: n+\frac12} : \bm Z_n }{ |\bm Z_n|^2 } \right) \nonumber \\
&= \frac{\bm Z_n}{|\bm Z_n |^2} \otimes \left( 2\frac{\partial \Delta G_{\mathrm{iso}}^{n+\frac12}}{\partial \bm{C}_{n+1}} -\frac{1}{2}\mathbb{C}_{\mathrm{iso}\:n+\frac12}: \bm Z_n - \bm S_{\mathrm{iso}\:n+\frac12} \right) -2\frac{\Delta G_{\mathrm{iso}}^{n+\frac12}- \bm S_{\mathrm{iso}\:n+\frac12 }:\bm Z_n}{|\bm Z_n|^4} \bm Z_n \otimes \bm Z_n.
\end{align}
In view of the definition of $\Delta G_{\mathrm{iso}}^{n+\frac12}$ \eqref{eq:B1} and $G_{\mathrm{iso}}$ \eqref{eq:gibbs-energy}$_3$, the term $\partial \Delta G_{\mathrm{iso}}^{n+\frac12} / \partial \bm{C}_{n+1}$ can be written as
\begin{align}
\label{eq:B5}
\frac{\partial \Delta G_{\mathrm{iso}}^{n+\frac12}}{\partial \bm C_{n+1}} &= \frac{\partial G_{\mathrm{iso}} \left( \tilde{\bm C}_{n+1}, \bm{\Gamma}_{n+\frac12}^{1}, \cdots, \bm{\Gamma}_{n+\frac12}^{m} \right)}{\partial \bm C_{n+1}} - \sum_{\alpha=1}^{m} \frac{\partial \Upsilon^{\alpha}\left( \tilde{\bm C}_{n}, \bm{\Gamma}_{n+\frac12}^{\alpha} \right)}{ \partial \bm C_{n+1} } \nonumber \\
&= \frac12 \bm{S}_{\mathrm{iso}} \left( \tilde{\bm C}_{n+1}, \bm{\Gamma}_{n+\frac12}^{1}, \cdots, \bm{\Gamma}_{n+\frac12}^{m} \right) - \sum_{\alpha=1}^{m} \frac{\partial \Upsilon^{\alpha}\left( \tilde{\bm C}_{n}, \bm{\Gamma}_{n+\frac12}^{\alpha} \right)}{ \partial \bm C_{n+1} }.
\end{align}
Recalling the definition of $\Upsilon^{\alpha}$ \eqref{eq:def-flv-Upsilon}, one has 
\begin{align}
\label{eq:B6}
\frac{\partial \Delta G_{\mathrm{iso}}^{n+\frac12}}{\partial \bm C_{n+1}} &= \frac12 \bm{S}_{\mathrm{iso}} \left( \tilde{\bm C}_{n+1}, \bm{\Gamma}_{n+\frac12}^{1}, \cdots, \bm{\Gamma}_{n+\frac12}^{m} \right) - \sum_{\alpha=1}^{m} \frac{\partial}{\partial \bm C_{n+1}} \left( \frac{1}{4\mu^{\alpha}} \left\lvert \tilde{\bm S}^{\alpha}_{\mathrm{iso}\:n} - \hat{\bm S}^{\alpha}_{0} - \mu^{\alpha} \left( \bm \Gamma^{\alpha}_{n+\frac12} - \bm I \right)\right\lvert^2  \right) \nonumber \\
&=\frac12 \bm{S}_{\mathrm{iso}} \left( \tilde{\bm C}_{n+1}, \bm{\Gamma}_{n+\frac12}^{1}, \cdots, \bm{\Gamma}_{n+\frac12}^{m} \right) - \sum_{\alpha=1}^{m} \frac{\partial}{\partial \bm C_{n+1}} \left( \frac{1}{4\mu^{\alpha}} \lvert \bm{\bar{Q}} \lvert^2\right),
\end{align}
in which
\begin{align*}
\bm{\bar{Q}} := \tilde{\bm S}^{\alpha}_{\mathrm{iso}\:n} - \hat{\bm S}^{\alpha}_{0} - \mu^{\alpha} \left( \bm \Gamma^{\alpha}_{n+\frac12} - \bm I \right).
\end{align*}
Different from the derivation in \ref{sec:appendix-A}, the variables $\bm{\Gamma}_{n+1}^{\alpha}$ depend on $\bm{C}_{n+1}$, according to the evolution equations \eqref{eq:isv-update-formula-2}. We have
\begin{align}
\label{eq:B7}
\frac{\partial \bm{\bar{Q}}}{\partial \bm C_{n+1}} &= -\mu^{\alpha} \frac12 \frac{\partial \bm{\Gamma}^{\alpha}_{n+1}}{\partial \bm C_{n+1}} = -\frac{\mu^{\alpha} \Delta t}{4\eta^{\alpha}+2\mu^{\alpha}\Delta t} \frac{\partial \bm{\tilde{S}}^{\alpha}_{\mathrm{iso\:n+1}}}{\partial \bm C_{n+1}} = -\frac{\mu^{\alpha} \Delta t}{4\eta^{\alpha}+2\mu^{\alpha}\Delta t} \frac{\partial \bm{\tilde{S}}^{\alpha}_{\mathrm{iso\:n+1}}}{\partial \bm{\tilde{C}}_{n+1}}:J^{-\frac23}_{n+1}\mathbb{P}^T_{n+1},
\end{align}
Consequently, the second term on the right-hand side of \eqref{eq:B6} can be expressed as
\begin{align}
\label{eq:B8}
\sum_{\alpha=1}^{m} \frac{\partial}{\partial \bm C_{n+1}} \left( \frac{1}{4\mu^{\alpha}} \lvert \bm{\bar{Q}} \lvert^2\right) &= \sum_{\alpha=1}^{m} \frac{1}{4\mu^{\alpha}} \left( \frac{\partial \bm{\bar{Q}}}{\partial \bm C_{n+1}} : \bm{\bar{Q}} + \bm{\bar{Q}} : \frac{\partial \bm{\bar{Q}}}{\partial \bm C_{n+1}} \right) \nonumber \\
&= -\sum_{\alpha=1}^{m} \frac{J^{-\frac23}_{n+1} \Delta t}{16\eta^{\alpha}+8\mu^{\alpha}\Delta t} \left( \frac{\partial \bm{\tilde{S}}^{\alpha}_{\mathrm{iso\:n+1}}}{\partial \bm{\tilde{C}}_{n+1}} : \mathbb{P}^T_{n+1}:\bm{\bar{Q}} + \bm{\bar{Q}} : \frac{\partial \bm{\tilde{S}}^{\alpha}_{\mathrm{iso\:n+1}}}{\partial \bm{\tilde{C}}_{n+1}} : \mathbb{P}^T_{n+1} \right).
\end{align}
Combining \eqref{eq:B6} and \eqref{eq:B8}, we have
\begin{align}
\label{eq:B9}
\frac{\partial \Delta G_{\mathrm{iso}}^{n+\frac12}}{\partial \bm C_{n+1}} = &\frac12 \bm{S}_{\mathrm{iso}} \left( \tilde{\bm C}_{n+1}, \bm{\Gamma}_{n+\frac12}^{1}, \cdots, \bm{\Gamma}_{n+\frac12}^{m} \right) \nonumber \\
&+\sum_{\alpha=1}^{m} \frac{J^{-\frac23}_{n+1} \Delta t}{16\eta^{\alpha}+8\mu^{\alpha}\Delta t} \left( \frac{\partial \bm{\tilde{S}}^{\alpha}_{\mathrm{iso\:n+1}}}{\partial \bm{\tilde{C}}_{n+1}} : \mathbb{P}^T_{n+1}:\bm{\bar{Q}} + \bm{\bar{Q}} : \frac{\partial \bm{\tilde{S}}^{\alpha}_{\mathrm{iso\:n+1}}}{\partial \bm{\tilde{C}}_{n+1}} : \mathbb{P}^T_{n+1} \right).
\end{align}
Invoking \eqref{eq:B2}-\eqref{eq:B4} and \eqref{eq:B9}, we obtain the final explicit expression for the elasticity tensor,
\begin{align}
\label{eq:B10}
\mathbb{C}_{\mathrm{iso}\:\mathrm{alg2}}
&= \frac{1}{2}\mathbb{C}_{ \mathrm{iso}\:n+\frac12 } + \frac{ \Delta G_{\mathrm{iso}}^{n+\frac12} - \bm S_{\mathrm{iso} \: n+\frac12} : \bm Z_n }{ |\bm Z_n|^2 } \left(\mathbb{I} -2\frac{\bm Z_n \otimes \bm Z_n}{|\bm Z_n|^2} \right) \nonumber \\
&\hspace{0.3cm} + \frac{\bm Z_n}{|\bm Z_n |^2} \otimes \left( \bm{S}_{\mathrm{iso}} \left( \tilde{\bm C}_{n+1}, \bm{\Gamma}_{n+\frac12}^{1}, \cdots, \bm{\Gamma}_{n+\frac12}^{m} \right)-\frac{1}{2}\mathbb{C}_{\mathrm{iso}\:n+\frac12}: \bm Z_n - \bm S_{\mathrm{iso}\:n+\frac12} \right) \nonumber \\
&\hspace{0.3cm} + \frac{\bm Z_n}{|\bm Z_n |^2} \otimes \left( \sum_{\alpha=1}^{m} \frac{J^{-\frac23}_{n+1} \Delta t}{8\eta^{\alpha}+4\mu^{\alpha}\Delta t} \left( \frac{\partial \bm{\tilde{S}}^{\alpha}_{\mathrm{iso\:n+1}}}{\partial \bm{\tilde{C}}_{n+1}} : \mathbb{P}^T_{n+1}:\bm{\bar{Q}} + \bm{\bar{Q}} : \frac{\partial \bm{\tilde{S}}^{\alpha}_{\mathrm{iso\:n+1}}}{\partial \bm{\tilde{C}}_{n+1}} : \mathbb{P}^T_{n+1} \right) \right).
\end{align}
The explicit form of $\partial \bm{\tilde{S}}^{\alpha}_{\mathrm{iso\:n+1}} / \partial \bm{\tilde{C}}_{n+1}$ is determined by the viscoelasticity model. For the HS model, we have
\begin{align*}
\frac{\partial \bm{\tilde{S}}^{\alpha}_{\mathrm{iso\:n+1}}}{\partial \bm{\tilde{C}}_{n+1}} = \frac{\partial}{\partial \bm{\tilde{C}}_{n+1}} \left( \mu^{\alpha} \left( \tilde{\bm{C}}_{n+1} - I \right) \right) = \mu^{\alpha} \mathbb{I};
\end{align*}
for the MIPC model, we have 
\begin{align*}
\frac{\partial \bm{\tilde{S}}^{\alpha}_{\mathrm{iso\:n+1}}}{\partial \bm{\tilde{C}}_{n+1}} =  \frac{\partial}{\partial \bm{\tilde{C}}_{n+1}} \left(  \beta^{\infty}_{\alpha} 2\frac{\partial G^{\infty}_{\mathrm{iso\:n+1}}}{\partial \tilde{\bm{C}}_{n+1}} \right) = \frac12 \beta^{\infty}_{\alpha} J^{\frac43}_{n+1} \tilde{\mathbb{C}}^{\infty}_{\mathrm{iso}\:n+1}, \quad \mbox{with} \quad \tilde{\mathbb{C}}^{\infty}_{\mathrm{iso}\:n+1} := 4J^{-\frac43}_{n+1} \left( \frac{\partial^2 G^{\infty}_{\mathrm{iso}}}{\partial \bm{\tilde{C}} \partial \bm{\tilde{C}}} \right)_{n+1}.
\end{align*}
The final explicit form of the elasticity tensor reads as
\begin{align}
\label{eq:B11}
\mathbb{C}_{\mathrm{iso}\:\mathrm{alg2}}
&= \frac{1}{2}\mathbb{C}_{ \mathrm{iso}\:n+\frac12 } + \frac{ \Delta G_{\mathrm{iso}}^{n+\frac12} - \bm S_{\mathrm{iso} \: n+\frac12} : \bm Z_n }{ |\bm Z_n|^2 } \left(\mathbb{I} -2\frac{\bm Z_n \otimes \bm Z_n}{|\bm Z_n|^2} \right) \nonumber \\
&\hspace{0.3cm} + \frac{\bm Z_n}{|\bm Z_n |^2} \otimes \left( \bm{S}_{\mathrm{iso}} \left( \tilde{\bm C}_{n+1}, \bm{\Gamma}_{n+\frac12}^{1}, \cdots, \bm{\Gamma}_{n+\frac12}^{m} \right)-\frac{1}{2}\mathbb{C}_{\mathrm{iso}\:n+\frac12}: \bm Z_n - \bm S_{\mathrm{iso}\:n+\frac12} \right) \nonumber \\
&\hspace{0.3cm} +
\begin{cases}
\frac{\bm Z_n}{|\bm Z_n |^2} \otimes \left( \sum_{\alpha=1}^{m} \frac{J^{-\frac23}_{n+1} \mu^{\alpha} \Delta t}{8\eta^{\alpha}+4\mu^{\alpha}\Delta t} \left( \mathbb{P}^T_{n+1}:\bm{\bar{Q}} + \bm{\bar{Q}} : \mathbb{P}^T_{n+1} \right) \right) & \mbox{HS model}, \\
\frac{\bm Z_n}{|\bm Z_n |^2} \otimes \left( \sum_{\alpha=1}^{m} \frac{J^{\frac23}_{n+1} \beta^{\infty}_{\alpha} \Delta t}{16\eta^{\alpha}+8\mu^{\alpha}\Delta t} \left( \tilde{\mathbb{C}}^{\infty}_{\mathrm{iso}\:n+1} : \mathbb{P}^T_{n+1}:\bm{\bar{Q}} + \bm{\bar{Q}} : \tilde{\mathbb{C}}^{\infty}_{\mathrm{iso}\:n+1} : \mathbb{P}^T_{n+1} \right) \right) & \mbox{MIPC  model}.
\end{cases}
\end{align}

\section{Glossary of Terms}
\label{sec:appendix-C}
\renewcommand{\arraystretch}{1.3}
\begin{longtable}{p{3.5cm} p{6cm} p{3cm}}
\hline
Symbol & Name or description & Place of definition or first occurrence\\
\hline
$\bm \Gamma^{\alpha}$, $\bm Q^{\alpha}$ & The $\alpha$-th internal state variable and its conjugate variable & Section \ref{sec:continuum-basis}, \eqref{eq:def-Q} \\
$\mu^{\alpha}$, $\eta^{\alpha}$ & The shear modulus and viscosity associated with the $\alpha$-th relaxation process & \eqref{eq:def-flv-Upsilon}, \eqref{eq:constitutive-flv-Q}  \\
$\beta^{\infty}_{\alpha}$ & The strain-energy factor & \eqref{eq:G_alpha_def_MIPC} \\
$P$, $p$ & The referential and spatial thermodynamic pressure & Section \ref{sec:continuum-basis} \\
$G(\tilde{\bm C}, P, \bm \Gamma^{1}, \cdots, \bm \Gamma^{m})$ & The Gibbs free energy & \eqref{eq:gibbs-energy} \\
$G_{\mathrm{vol}}^{\infty}(P)$ & The volumetric part of the Gibbs free energy & \eqref{eq:gibbs-energy} \\
$G_{\mathrm{iso}}(\tilde{\bm C}, \bm \Gamma^{1}, \cdots, \bm \Gamma^{m})$ & The isochoric part of the Gibbs free energy & \eqref{eq:gibbs-energy} \\
$G_{\mathrm{iso}}^{\infty}(\tilde{\bm C})$ & The equilibrium part of $G_{\mathrm{iso}}$ & \eqref{eq:gibbs-energy} \\
$\Upsilon^{\alpha}\left( \tilde{\bm C}, \bm \Gamma^{\alpha} \right)$ & The non-equilibrium part of $G_{\mathrm{iso}}$ associated with the $\alpha$-th relaxation process, configurational free energy & \eqref{eq:gibbs-energy} \\
$G^{\alpha}(\tilde{\bm C})$, $\tilde{\bm S}^{\alpha}_{\mathrm{iso}}$ & An energy-like function used in the definition of $\Upsilon^{\alpha}$ and its associated fictitious stress & \eqref{eq:def-flv-Upsilon-terms}$_1$ \\
$\hat{\bm S}^{\alpha}_{0}$ & A constant stress-like tensor in the definition of $\Upsilon^{\alpha}$  & \eqref{eq:def-flv-Upsilon-terms}$_2$ \\
$\tilde{\bm S}$ & The fictitious second Piola-Kirchhoff stress & \eqref{eq:split-fictitious-S}$_1$ \\
$\tilde{\bm S}_{\mathrm{iso}}^{\infty}$ & The equilibrium part of the fictitious isochoric second Piola-Kirchhoff stress & \eqref{eq:split-fictitious-S}$_2$  \\
$\tilde{\bm S}^{\alpha}_{\mathrm{neq}}$ & The $\alpha$-th non-equilibrium part of the fictitious isochoric second Piola-Kirchhoff stress & \eqref{eq:non-equilbrium-fictitious-S-def} \\
$\mathbb V^{\alpha}$ & The viscosity tensor associated with the $\alpha$-th relaxation process & \eqref{eq:constitutive-flv-Q} \\
$\bm S_{\mathrm{vol}}$ & The volumetric part of the second Piola-Kirchhoff stress & \eqref{eq:constitutive_S_vol} \\
$\bm S_{\mathrm{iso}}$ & The isochoric part of the second Piola-Kirchhoff stress & \eqref{eq:constitutive_S_iso} \\
$\bm S_{\mathrm{iso}}^{\infty}$ & The equilibrium part of the isochoric second Piola-Kirchhoff stress & \eqref{eq:constitutive_S_infty_iso_S_neq_alpha}$_1$  \\
$\bm S^{\alpha}_{\mathrm{neq}}$ & The $\alpha$-th non-equilibrium part of the isochoric second Piola-Kirchhoff stress & \eqref{eq:constitutive_S_infty_iso_S_neq_alpha}$_2$ \\
$\bm P_{\mathrm{iso}}^{\infty}$ & The isochoric part of the First Piola-Kirchhoff stress & \eqref{eq:constitutive-P-iso} \\
\hline
\end{longtable}

\bibliographystyle{elsarticle-num}
\bibliography{viscoelaticity-EM.bib}

\end{document}